\documentclass[12pt]{amsart}
\usepackage{ amsmath, amsthm, amsfonts, amssymb, color}
 \usepackage{mathrsfs}
\usepackage{amsfonts, amsmath}
 \usepackage{amsmath,amstext,amsthm,amssymb,amsxtra}
 \usepackage{txfonts} 
 \usepackage[colorlinks, citecolor=blue,pagebackref,hypertexnames=false]{hyperref}
 \allowdisplaybreaks
 \usepackage{pgf,tikz}

\begin{document}

 \baselineskip 16.6pt
\hfuzz=6pt

\widowpenalty=10000

\newtheorem{cl}{Claim}
\newtheorem{theorem}{Theorem}[section]
\newtheorem{proposition}[theorem]{Proposition}
\newtheorem{coro}[theorem]{Corollary}
\newtheorem{lemma}[theorem]{Lemma}
\newtheorem{definition}[theorem]{Definition}
\newtheorem{assum}{Assumption}[section]
\newtheorem{example}[theorem]{Example}
\newtheorem{remark}[theorem]{Remark}
\renewcommand{\theequation}
{\thesection.\arabic{equation}}

\def\SL{\sqrt H}

\newcommand{\mar}[1]{{\marginpar{\sffamily{\scriptsize
        #1}}}}

\newcommand{\as}[1]{{\mar{AS:#1}}}

\newcommand\R{\mathbb{R}}
\newcommand\RR{\mathbb{R}}
\newcommand\CC{\mathbb{C}}
\newcommand\NN{\mathbb{N}}
\newcommand\ZZ{\mathbb{Z}}
\def\RN {\mathbb{R}^n}
\renewcommand\Re{\operatorname{Re}}
\renewcommand\Im{\operatorname{Im}}

\newcommand{\mc}{\mathcal}
\newcommand\D{\mathcal{D}}
\def\hs{\hspace{0.33cm}}
\newcommand{\la}{\alpha}
\def \l {\alpha}
\newcommand{\eps}{\varepsilon}
\newcommand{\pl}{\partial}
\newcommand{\supp}{{\rm supp}{\hspace{.05cm}}}
\newcommand{\x}{\times}
\newcommand{\lag}{\langle}
\newcommand{\rag}{\rangle}

\newcommand\wrt{\,{\rm d}}
\newcommand{\botimes}{\bar{\otimes}}

\title[]{Sharp endpoint $L_p$ estimates of quantum Schr\"{o}dinger groups}

\author{Zhijie Fan}
\address{Zhijie Fan, Department of Mathematics, Wuhan University, }
\email{ZhijieFan@whu.edu.cn}

\author{Guixiang Hong}
\address{Guixiang Hong, Department of Mathematics, Wuhan University, }
\email{guixiang.hong@whu.edu.cn}

\author{Liang Wang}
\address{Liang Wang, Department of Mathematics, Wuhan University, }
\email{wlmath@whu.edu.cn}


  \date{\today}

\begin{abstract}
In this article, we establish sharp endpoint $L_p$ estimates of Schr\"odinger groups on general measure spaces which may not be equipped with good metrics but admit submarkovian semigroups satisfying purely algebraic assumptions. One of the key ingredients of our proof is to introduce and investigate a new noncommutative high-cancellation BMO space by constructing an abstract form of P-metric codifying some sort of underlying metric and position. This provides the first form of Schr\"odinger group theory on arbitrary von Neumann algebras and can be applied to many models, including Schr\"odinger groups associated with non-negative self-adjoint operators satisfying purely Gaussian upper bounds on doubling metric spaces, standard Schr\"odinger groups on quantum Euclidean spaces, matrix algebras and group von Neumann algebras with finite dimensional cocycles.
\end{abstract}

\maketitle


\tableofcontents

\section{Introduction}
\subsection{Background and motivation}
Schr\"{o}dinger equation is a basic equation in quantum mechanic. Denote by $\Delta$ the Laplacian operator on the Euclidean space $\mathbb{R}^n$ given by $\Delta=\sum_{i=1}^n\partial_{x_i}^2$, then the solution of initial value  problem for Schr\"{o}dinger equation
\begin{eqnarray} \label{Sch}
\left\{
\begin{array}{ll}
    i { \partial_t u (x,t)  } -\Delta u(x, t)=0, \quad x\in \mathbb{R}^n,\ \  t>0,  \\[5pt]
 u(x, 0)=f(x),
\end{array}
\right.
\end{eqnarray}
with initial value $f\in L_2(\mathbb{R}^n)$ can be represented as
$$u(x,t)=e^{-it\Delta}f(x).$$
It can be seen from H\"{o}rmander's result \cite{MR121655} that the family of solution operators $(e^{-it\Delta})_{t\in\mathbb R}$, also called Schr\"odinger group, are bounded on $L^p(\mathbb{R}^n)$ only when $p=2$. For general $1<p<\infty$, based on the proceeding inspiring work \cite{MR324249,MR0234133,MR270219}, Miyachi \cite{MR573335,MR633000} showed that the operator $e^{-it\Delta}$  is bounded from the Sobolev space $W_p^{2s}(\mathbb{R}^n)$ to $L_p(\mathbb{R}^n)$ for $s\geq \sigma_p:=n|1/p-1/2|$ with the {\it optimal} bound $C(1+|t|)^{\sigma_p}$ (see also  \cite{MR447953} due to Fefferman and Stein). Equivalently, the  operator $e^{-it\Delta}(I-\Delta)^{-s}$ is bounded on $L_p(\mathbb{R}^n)$. This result is said to be an {\it endpoint} $L_p$-estimate since such a mapping property does not hold for $0<s<\sigma_p$. The proof of the result depends heavily on the theory of Hardy spaces and the Fourier transform estimate of Fourier multiplier $e^{it|\cdot|^2}$---the symbol of the operator $e^{-it\Delta}$.

Motivated by further applications to partial differential equations and geometric analysis, harmonic analysis associated to generalized semigroups, or equivalently generalized Laplacians, such as Laplace-Beltrami operators, divergence forms and Schr\"odinger operators {\it etc.}, has been developing rapidly in the past twenty years (see  \cite{MR1933726,MR2163867,MR2162784,MR2868142} and the references therein). In particular, over the last decade, the $L_p$-regularity problem for the Schr\"odinger groups $e^{itL}$ associated to a large class of self-adjoint operators $L$ has attracted extensive attention. Let $L$ be a non-negative self-adjoint operator on $L_2(X)$, where $X$ is a metric measure space  with a distance function $d$, a measure  $\mu$ and satisfying
 the doubling property:
\begin{align}\label{doubling}
\mu(B(x,\lambda r))\leq C\lambda^{n}\mu(B(x,r))
\end{align}
 for some $C, n>0$ independent of $\lambda\geq 1$, $r>0$ and $x\in X$. Here we denote by $B(x,r)$ the ball centred at $x$ with radius $r$. The smallest $n$ so that the above inequality valid is called the dimension of $X$. Then by spectral theory, the solution of Schr\"{o}dinger equation \eqref{Sch} with $-\Delta$ replaced by $L$ and
with initial value $f\in L_2(X)$ can be represented as
$$u(x,t)=e^{itL}f(x).$$
After several hard attempts \cite{MR2301958,MR4082915, MR1930609, MR3556060, MR1275402,MR1317282,  MR2124040}, Chen--Duong--Li--Yan \cite{MR4150932} established the sharp endpoint  $L_p$-regularity estimate of $e^{itL}$ for a large class of $L$.  Specifically, they used the Littlewood-Paley inequalities and a variant of Fefferman-Stein's sharp maximal function to show that if the heat kernel $p_t(x,y)$ associated with $L$ satisfies
$m$-order Gaussian upper bound
\begin{equation*}
 \label{GE}
 \tag{${\rm GE}_m$}
|p_t(x,y)| \leq {C\over \mu(B(x,t^{1/m}))} \exp\left(-c \, {  \left({d(x,y)^{m}\over    t}\right)^{1\over m-1}}\right)
\end{equation*}
for some constants $m\geq 2$ and $C,c>0$ independent of $t>0$ and $x,y\in X$, then for any $1<p<\infty$,
$$\|e^{itL}(I+L)^{-s}f\|_{L_p(X)}\leq C(1+|t|)^{\sigma_p}\|f\|_{L_p(X)},\ {\rm for}\ {\rm any}\ s\geq \sigma_p.$$
See also \cite{2022On,CDLY,MR4150932,MR4371094,MR4182975} for further progresses.

On the other hand, inspired by operator algebras, noncommutative martingale inequalities and quantum information theory, there have appeared several fundamental work on harmonic analysis associated to quantum semigroups $(e^{-tL})_{t\geq0}$  (see e.g. \cite{MR3920831,MR2885593,MR2265255,MR2276775,MR2327840}). Moreover, this part of noncommutative harmonic analysis plays an instrumental role in Fourier analysis on group algebras (see e.g. \cite{MR3283931,MR3776274,10.1215/00127094-2021-0042}). However, regarding the Schr\"odinger group $(e^{itL})_{t\in\mathbb R}$---an analytic extension of quantum semigroup to imaginary time---on the truly noncommutative measure spaces, there seems no result in the existing literature. The reason behind might be that the key ingredients such as the Fourier transform/kernel estimates like \eqref{GE}, Hardy space theory and the sharp maximal function techniques {\it etc.} are not available or do not admit perfect analogues in the noncommutative setting.

The main task of this article is to investigate the mapping properties of $e^{itL}$ for a fixed $t$ and thus to provide the $L_p$-regularity estimate for solution of the following Schr\"{o}dinger equation  on general noncommutative measure spaces:
 \begin{eqnarray}\label{nSch}
\left\{
\begin{array}{ll}
    i { \partial_t u (t)  } +  Lu(t)=0, \  t>0,  \\[5pt]
 u( 0)=f\in L_2(\mathcal{M}),
\end{array}
\right.
\end{eqnarray}
where $\mathcal{M}$ is semifinite von Neumann algebra with a normal semifinite faithful trace $\tau$ and $L$ is a non-negative self-adjoint operator on $L_2(\mathcal{M})$, the noncommutative $L_2$ space over $\mathcal{M}$. We refer the reader to \cite{PX} for more terminologies from noncommutative integration theory.
It is worthwhile to mention that the $L_p$-regularity estimate is one of the most important estimates for the regularity problem of solution of Schr\"{o}dinger equation \eqref{Sch}, while a generalization to the noncommutative setting can be regarded as a starting point for the research of the subsequent related problems, including noncommutative local smoothing estimate and noncommutative Schr\"{o}dinger maximal inequality, both of which lie in the central topics of modern harmonic analysis even when going back to the commutative setting (see for example \cite{MR3702674,MR3961084,MR4432280,MR2629687}).



%
%
In order to achieve our goal, we first observe that even in the Euclidean setting, the target operator $(I+L)^{-\sigma_p}e^{itL}$  is a strong singular integral operator beyond Calder\'{o}n-Zygmund framework (see for example \cite{MR4439910,MR719660,MR4390220}).
Over the last twenty decades, lots of significant progress toward noncommutative singular integral operator theory has been achieved \cite{MR4320770,MR3178596,MR4220746,MR2666903,MR3283931,MR4178915,MR2496770,MR2476951}. Among these results, we would like to highlight a pioneering work due to Junge--Mei--Parcet--Xia \cite{MR4178915}, who first introduced a noncommutative form of Calder\'{o}n-Zygmund theory under algebraic framework. Specifically, in the lack of good metrics, they established the $L_p$-boundedness of algebraic Calder\'{o}n-Zygmund operators on semi-finite von Neumann algebra admitting a Markov semigroup which only satisfies purely algebraic assumptions. This result covers many concrete examples in both classical setting and noncommutative setting. Besides, it is also particularly interesting for measure spaces with poor geometric information.

Although our target operator is a class of singular integral, it is a strong type beyond Calder\'{o}n-Zygmund framework,  which means that Calder\'{o}n-Zygmund theory cannot be applied directly to this operator. Instead, the extra oscillation information about $e^{itL}$ should be dug out and the singularity of the target operator should be handled  more delicately. Inspired by the establishment of algebraic Calder\'{o}n-Zygmund theory, we will introduce the first form of Schr\"{o}dinger group theory which is valid for general semi-finite von Neumann algebras under purely algebraic framework.


\subsection{Main result}
A semigroup $\mathcal{S}=(S_t)_{t\geq0}$ over $\mathcal{M}$ is a family of operators $S_t:\mathcal{M}\rightarrow\mathcal{M}$ such that $S_0=id_\mathcal{M}$ and $S_tS_s=S_{t+s}$ for any $t,s\geq 0$. It is said to be a bounded semigroup if it is uniformly bounded and satisfies the following properties:

(i) each $S_t$ is a weak-$*$ continuous and completely positive map;

(ii) each $S_t$ is self-adjoint in the sense that
\begin{align}\label{symmetry}
\tau(S_t(f)g)=\tau(f S_t(g)),\ \ {\rm for\ all}\ f,g\in\mathcal{M}\cap L_1(\mathcal{M});
\end{align}

(iii) the map $t\rightarrow S_t$ is a pointwise weak-$*$ continuous map.

A bounded semigroup $\mathcal{S}$ is said to be submarkovian if each $S_t$ is a contractive map. Besides, a submarkovian semigroup $\mathcal{S}$ is said to be Markovian if it satisfies an additional unital property, i.e. $S_t(1)=1$.   Throughout the article, unless we mention the contrary, we will assume that the underlying noncommutative measure space $(\mathcal{M},\tau)$ comes equipped with a bounded semigroup $\mathcal{S}=(S_t)_{t\geq 0}$.  Note that $\mathcal{S}$ admits an infinitesimal negative generator
$$Af=\lim_{t\rightarrow 0}\frac{S_t(f)-f}{t}$$
defined on $dom(A)=\cup_{1\leq p\leq \infty}dom_p(A)$, where $dom_p(A)$ is given by
\begin{align*}
dom_p(A)=\left\{f\in L_p(\mathcal{M}):\lim_{t\rightarrow 0}\frac{S_t(f)-f}{t}\ {\rm converges}\ {\rm in}\ L_p(\mathcal{M})\right\}.
\end{align*}
It can be verified easily that for any $s>0$ and $f\in L_p(\mathcal{M})$, one has $\frac{1}{s}\int_0^s S_t(f)dt\in dom_p(A)$, which means that $dom_p(A)$ is dense in $L_p(\mathcal{M})$ for $p<\infty$, and is weak-$*$ dense in $\mathcal{M}$ for $p=\infty$. We denote $A_p$ by the restriction of $A$ on $dom_p(A)$.
It follows from \eqref{symmetry} that $L:=-A$ is a non-negative self-adjoint operator acting on $L_2(\mathcal{M})$. By the spectral theorem, for any bounded Borel function $F:[0,\infty)\rightarrow \mathbb{C}$, one can define an $L_2(\mathcal{M})$-bounded operator $F(L)$ given by
$$F(L)=\int_0^\infty F(\lambda)dE_L(\lambda),$$
where $E_L(\lambda)$ is the projection-valued measure supported on the spectrum of $L$. By this spectral decomposition, the solution of noncommutative Schr\"{o}dinger equation \eqref{nSch} can be represented as
$$u(t)=e^{itL}f,\ {\rm for}\ {\rm any}\ t>0.$$

In order to establish an algebraic Schr\"{o}dinger group theory as general as possible, many difficulties occur due to the non-commutativity and the lack of geometric concepts. To the best of our knowledge, even in the Euclidean setting, there are only two methods in the literatures to establish the $L_p$ estimate of Schr\"{o}dinger groups. The first one is to show the $L_p$-boundedness directly, which relies heavily on the $L_p$-boundedness of sharp maximal function \cite{MR4150932}. Another one is to  establish the $H^1\rightarrow L_1$ boundedness of targeted operator $e^{itL}(I+L)^{-n/2}$ for appropriate Hardy space first \cite{CDLY}, which relies on atom (or molecule) decomposition of Hardy space, and then use complex interpolation with the trivial $L_2$ boundedness. However, in the general non-commutative setting, these approaches are not valid. To overcome these difficulties, our strategy is to establish the $L_\infty\rightarrow BMO$ boundedness of $e^{itL}(I+L)^{-n(1+iy)/2}$ for appropriate BMO space directly without using its duality with the corresponding Hardy space, which is a complete new proof even in the most classical setting. The implementation of this strategy can be divided into three steps:

\  {\bf Step 1:} Construct a `P-metric' and show that the associated metric-type BMO space can be embedded into semigroup BMO space.

We shall introduce a notion called P-metric, which codifies some sort of underlying metric and position in $(\mathcal{M},\tau)$. This concept is a modification of the Markov metric introduced in \cite{MR4178915}, but it contains certain extra 'position' information of the underlying metric in $(\mathcal{M},\tau)$, so we call this a P-metric. It is determined by a family $$\mathcal{Q}=\{(R_{j,t},\sigma_{j,t}):j\in\mathbb{N},t>0\},$$
where each $R_{j,t}:\mathcal{M}\rightarrow \mathcal{M}$ is a normal completely positive unital map satisfying $R_{j,t}(f)=f$ for $f\in ker A_\infty$, the fixed-point subspace of the semigroup $\mathcal{S}$, and $\sigma_{j,t}$ are elements of the von Neumann algebra $\mathcal{M}$ such that the estimates below hold for some $m\geq 2$:
\begin{align*}
&{\rm (i)}\ semigroup\ majorization:\ S_t(|\xi|^2)\leq \sum_{j\geq 0}\sigma_{j,t}^*R_{j,\sqrt[m]{t}}(|\xi|^2)\sigma_{j,t},\forall\ t>0\ {\rm and}\ \xi\in\mathcal{M};\\
&{\rm (ii)}\ average\ domination\ condition:\ \big\|R_{j,t}(|\xi|^2)\big\|_{\mathcal{M}}\lesssim \big\|R_{t}(|\xi|^2)\big\|_{\mathcal{M}},\forall j\in\mathbb{N},t>0\ {\rm and}\ \xi\in\mathcal{M};\\
&{\rm (iii)}\ metric\ integrability\ condition:\ k_{\mathcal{Q}}:=\sup\limits_{t>0}\Big\|\sum_{j\geq 0}|\sigma_{j,t}|^2\Big\|_\mathcal{M}^{1/2}<\infty.
\end{align*}
Here and in what follows, in order to simplify the notation, we set $R_t(f):=R_{0,t}(f)$ for any $f\in\mathcal{M}$. By constructing concrete P-metric in specific examples, we shall explain more in Section \ref{Hilbert modules} and Section \ref{Appi} what motivates our definition.

Let
\begin{align*}
&\|f\|_{{ BMO}_{\mathcal{S}}^c(\mathcal{M})}:=\sup_{t> 0}\big\|S_t|(I-S_t)f|^2\big\|_{\mathcal{M}}^{1/2}
\end{align*}
and $\|f\|_{{BMO}_{\mathcal{S}}(\mathcal{M})}:=\max\big\{\|f\|_{{ BMO}_{\mathcal{S}}^c(\mathcal{M})},\|f^*\|_{{BMO}_{\mathcal{S}}^c(\mathcal{M})}\big\}$. Inspired by \cite{MR2885593,MR4178915,MR2469026}, the semigroup type BMO space ${ BMO}_{\mathcal{S}}(\mathcal{M})$ is defined as the weak-$*$ closure of $\mathcal{M}$ in certain direct sum of Hilbert modules determined by $\mathcal{S}$ (see Section \ref{semigroupBMO} for more details). $BMO_\mathcal{S}(M)$ interpolates with $L_p(\mathcal{M})$ provided $\mathcal{S}$ is a Markov semigroup admitting a Markov dilation (see Section \ref{lpbod}).  Given a P-metric associated to this semigroup and $M\in\mathbb{Z}_+$, we define metric-type BMO (semi-)norm of $f\in\mathcal{M}$ as follows:
$$\|f\|_{BMO_{\mathcal{Q}}(\mathcal{M})}=\max\big\{\|f\|_{BMO_{\mathcal{Q}}^c(\mathcal{M})},\|f^*\|_{BMO_{\mathcal{Q}}^c(\mathcal{M})}\big\},$$
\begin{align*}
\|f\|_{BMO_{\mathcal{Q}}^c(\mathcal{M})}:=\sup_{t> 0}\big\|R_{\sqrt[m]{t}}|(I-S_t)f|^2\big\|_{\mathcal{M}}^{1/2}.
\end{align*}
Then it is not difficult to see that $BMO_{\mathcal{Q}}^c(\mathcal{M})\subset BMO_{\mathcal{S}}^c(\mathcal{M})$ (see Lemma \ref{embed} for more details).

\  {\bf Step 2:} Provide suitable operator conditions on $L$ such that metric-type BMO space coincides with  high-cancellation metric-type BMO space.

In order to ensure all the definitions in our abstract framework are well-defined, we will impose a priori assumption that for any function $F$ mapping $[ 0,\infty)$ to  $\mathbb{C}$ and having bounded derivatives at any order, the spectral multiplier $F(L):\mathcal{A}_\mathcal{M}\rightarrow \mathcal{A}_\mathcal{M}$ for some weak-$*$ dense subalgebra $\mathcal{A}_\mathcal{M}\subset \mathcal{M}$ such that $\mathcal{A}_\mathcal{M}$ is dense in $L_p(\mathcal{M})$ for all $1\leq p<\infty$. One shall see in the body paragraph that this assumption holds trivially for many non-commutative models and can be relaxed in the semi-commutative setting that we shall consider.

Motivated by some classical ideas initially originated from \cite{MR2868142,MR2481054}, we shall introduce a non-commutative high-cancellation BMO space in this article, which has not been studied in the previous literature in the field of non-commutative analysis. Given a P-metric associated to this semigroup and $M\in\mathbb{Z}_+$, we define $M$-order metric-type BMO (semi-)norm of $f\in\mathcal{M}$ as follows:
$$\|f\|_{BMO_{\mathcal{Q},M}(\mathcal{M})}=\max\big\{\|f\|_{BMO_{\mathcal{Q},M}^c(\mathcal{M})},\|f^*\|_{BMO_{\mathcal{Q},M}^c(\mathcal{M})}\big\},$$
\begin{align*}
\|f\|_{BMO_{\mathcal{Q},M}^c(\mathcal{M})}:=\sup_{t> 0}\big\|R_{\sqrt[m]{t}}|(I-S_t)^Mf|^2\big\|_{\mathcal{M}}^{1/2}.
\end{align*}
We will show that under suitable algebraic, analytic and topological conditions (explained more in Step 3), for any $f\in\mathcal{A}_\mathcal{M}$,
$$\|f\|_{BMO_{\mathcal{Q},1}(\mathcal{M})}\simeq \|f\|_{BMO_{\mathcal{Q},M}(\mathcal{M})}.$$
This self-improved property plays a crucial role in establishing spectral multiplier theorem associated with operators of various types (e.g. H\"{o}rmander-type spectral multipliers \cite{MR2747017}, Schr\"{o}dinger groups multipliers \cite{MR4150932}) since it endows BMO function a higher cancellation property in the low frequency.

\  {\bf Step 3:} Provide suitable operator conditions on $L$ to establish the $L_\infty\rightarrow BMO$ boundedness for high-cancellation metric-type BMO space.

Inspired by the idea of introducing Markov metric in \cite{MR4178915}, the commutative idea behind the notion of P-metric is also to find pointwise majorants of the heat kernels of semigroups $\mathcal{S}$, such that $S_t$ can be dominated by a sum of averaging operators $R_{j,\sqrt[m]{t}}$ over a distinguished family of algebraic balls `with radius $2^{j}\sqrt[m]{t}$'. In the noncommutative setting, the maps $R_{t}$ (resp. $R_{j,t}$) must be averages over certain projections $q_{t}$ (resp. $q_{j,t}$). It is too technical to explain this concept in full generality at this point of our article. However, the authors may first care about a non-trivial model example given by
$$R_tf=\Big((id\otimes \tau)(q_t)\Big)^{-\frac{1}{2}}(id\otimes\tau)\big(q_t(1\otimes f)q_t\big)\Big((id\otimes \tau)(q_t)\Big)^{-\frac{1}{2}}$$
for certain increasing family of projections ${\bf q}:=\{q_t\}_{t\geq 0}\in\mathcal{M}\botimes\mathcal{M}$ determined later on.

As it happens in classical Schr\"{o}dinger group theory and algebraic Calder\'{o}n-Zygmund theory \cite{MR4178915}, one need to impose some additional but appropriate conditions (which will be split into algebraic, analytic  and topological conditions) on the P-metric to connect closely with the underlying (noncommutative) measure. We would like to just mention at this stage that the imposed algebraic conditions are very similar to that imposed in the algebraic Calder\'{o}n-Zygmund theory, which are inherent to noncommutativity and hold trivially in the commutative settings. The analytic one provides a strong doubling condition, which is equivalent to the doubling condition \eqref{doubling} in the setting of metric measure space (see Section \ref{Operator-valued setting}) but is stronger than the doubling condition in the general framework since it can reflect extra covering property.  For the model example mentioned in the previous paragraph, the strong doubling condition states that
 there exist constants $C,n>0$, $D\geq 0$, such that for any $r_1\geq r_2\geq r_3$ and $f\in\mathcal{M}$,
\begin{align*}
\Bigg\|(id\otimes \tau)\bigg(\bigg|(1\otimes f)q_{r_1}(id\otimes \tau)(q_{r_3})^{-\frac{1}{2}}\bigg|^{2}\bigg)\Bigg\|_{\mathcal{M}}\leq C\left(\frac{r_1}{r_2}\right)^{D}\left(\frac{r_1}{r_3}\right)^{n}\Bigg\|(id\otimes \tau)\bigg(\bigg|(1\otimes f)q_{r_2}(id\otimes \tau)(q_{r_2})^{-\frac{1}{2}}\bigg|^{2}\bigg)\Bigg\|_{\mathcal{M}}.
\end{align*}
In particular, with the choices $f=1$ and $r_1=r_2$, this implies the doubling condition:
$$\left\|\Big((id\otimes \tau)(q_{r_3})\Big)^{-\frac{1}{2}}\Big((id\otimes \tau)(q_{r_1})\Big)\Big((id\otimes \tau)(q_{r_3})\Big)^{-\frac{1}{2}}\right\|_{\mathcal{M}}\leq C\left(\frac{r_1}{r_3}\right)^{n}.$$
Instead, the smallest constant $n=n_{\bf q}$ such that the above inequality valid is called the {\it P-metric dimension} of $\mathcal{M}$ which, together with an appropriate choice of $q$, goes back to the homogeneous dimension of $X$ when $\mathcal{M}=L_\infty(X)$. The topological conditions provide certain natural continuity of projections and amplification map.


Next we impose operator conditions on $L$.  In the model example, our operator condition is just an $L_2$-Gaussian estimate, which can be regarded as an `integral-type' replacement of Gaussian upper bound \eqref{GE} and  is stated as follows:
%


There exist constants $C,c>0$ and $m\geq 2$ such that for any $t>0$, $k\geq 1$, $0<r_1\leq r_2$ and $f\in\mathcal{M}\botimes \mathcal{M}$,
\begin{align}\label{preoff}
&\left\|(id \otimes \tau)\Big|(id\otimes S_t)(fb_{k,r_2})q_{r_1} \Big((id\otimes \tau)(q_{r_1})\Big)^{-\frac{1}{2}}\Big|^2\right\|_\mathcal{M}^{\frac{1}{2}}\nonumber\\
&\leq CE_k\left(-c\left(\frac{r_2}{t^{1/m}}\right)^{\frac{m}{m-1}}\right)\Big\|(id \otimes \tau)\Big|fb_{k,r_2} \Big((id\otimes \tau)(q_{\max\{r_1,t^{1/m}\}})\Big)^{-\frac{1}{2}}\Big|^2\Big\|_{\mathcal{M}}^{\frac 12},
\end{align}
where
$$ b_{k,r}=\left\{\begin{array}{ll}q_{2r}, &{\rm if}\ k=1,\\q_{2^{k}r}-q_{2^{k-1}r}, &{\rm if}\ k\geq 2,\end{array}\right.$$
$$ E_{k}(x)=\left\{\begin{array}{ll}1, &{\rm if}\ k=1,\\ \exp(2^{\frac{mk}{m-1}}x), &{\rm if}\ k\geq 2.\end{array}\right.$$

In condition (2), $b_{k,r}$ and $q_r$ play the roles of indicators with respect to the algebraic annuli with radius $2^kr$ and algebraic balls with radius $r$, respectively. In the Euclidean setting, this condition is a simple consequence of \eqref{GE} (see Section \ref{Operator-valued setting} for more details). In the above text, we just formulate our conditions for special amplification algebra model $\mathcal{M}\botimes\mathcal{M}$, but our general conditions will include many more amplification algebras other than this model, which will be specific in Section \ref{aao}. With these more general conditions, our main result can be formulated as follows.
\begin{theorem}\label{main}
Let $(\mathcal{M},\tau)$ be a semifinite von Neumann algebra equipped with a bounded semigroup $\mathcal{S}=(e^{-tL})_{t\geq 0}$ with associated P-metric $\mathcal{Q}$ fulfilling our algebraic, analytic and topological assumptions. Then there exists a constant $C=C(n)>0$ such that for any $f\in \mathcal{A}_\mathcal{M}$,
\begin{align*}
\left\|e^{itL}(I+L)^{-s}f\right\|_{{ BMO}_{\mathcal{S}}(\mathcal{M})}\leq C(1+|t|)^{n/2}\|f\|_{\mathcal{M}},\ {\rm for}\ s\geq \frac{n}{2}.
\end{align*}
If in addition $\mathcal{S}$ is a Markovian semigroup admitting a Markov dilation, then for any $1<p<\infty$,
\begin{align*}
\left\|e^{itL}(I+L)^{-s}f\right\|_{L_{p}(\mathcal{M})}\leq C(1+|t|)^{\sigma_p}\|f\|_{L_p(\mathcal{M})},\ {\rm for}\ s\geq\sigma_p=n\left|\frac12-\frac1p\right|.
\end{align*}
\end{theorem}
\begin{remark}{\rm
Theorem \ref{main} is sharp in the sense that under the imposed conditions, the dependency on $t$ of the upper bound and the minimal regularity of $s$ cannot be improved, which can be seen by the classical example $L=-\Delta$ mentioned at the beginning of the introduction. Note that one cannot expect to show the sharpness of this estimate for all von Neumann algebra $\mathcal{M}$ and all generator $L$ under our framework. A typical counterexample appears when $\mathcal{M}$ and $L$ chosen to be $L_\infty$-space over Heisenberg group and its associated sub-Laplacian operator, respectively. In this case, the threshold dimension should be the topological dimension instead of the homogeneous one \cite{MR4390220}.}
\end{remark}

\subsection{Applications}
By verifying all the algebraic, analytic, topological and operator conditions of the main theorem (with some minor modification), we shall see how our main result applies to a wide variety of concrete von Neumann algebras with specific P-metrics. We list these models in the following.

{\bf Example A. Doubling metric spaces}

In this commutative model $\mathcal{M}=L_\infty(X)$, we shall see that our algebraic theory recover the result obtained by Chen--Duong--Li--Yan (see more comment on Remark \ref{rm12}). In other words, we provide a new proof of $L_p$-boundedness for $(I+L)^{-\sigma_p}e^{itL}$ via $L_\infty\rightarrow BMO_L$ boundedness together with a standard complex interpolation argument (see for example \cite[Section 4]{CDLY}) under the assumption that $L$ is a non-negative self-adjoint operator on $L_2(X)$ satisfying Gaussian upper bound \eqref{GE}.

{\bf Example B. Operator-valued setting}

As a generalization of matrix-valued harmonic analysis, operator-valued (semi-commutative) harmonic analysis has became an active area of research (see for example \cite{HLX,MR4220746,Lai,MR2327840,MR2476951} and the references therein). The operator-valued model is perhaps  the easiest non-commutative model, but it usually require revolutionary arguments to do analysis due to noncommutativity. To illustrate our application on this model, we let $X$ be a doubling metric space with a dimension $n$, a distance function $d$ and a nonnegative Borel doubling measure $\mu$. Besides, let $\mathcal{M}$ be a semifinite von Neumann algebra equipped with a normal semifinite faithful trace $\tau$ and $\mathcal{N}=L_\infty(X)\botimes \mathcal{M}$ be the von Neumann algebra tensor
product equipped with tensor trace $\int\otimes \tau$. Moreover, let $L$ be a non-negative self-adjoint operator on $L_2(X)$.
Then under the assumption that  the heat kernel associated with $L$ is non-negative, satisfies Gaussian upper bound \eqref{GE} and induces a Markovian semigroup admitting a Markov dilation, our algebraic theory can be applied to establish the $L_p$-boundedness of $(id_{\mathcal{N}}+L\otimes id_\mathcal{M})^{-s}e^{it(L\otimes id_\mathcal{M})} $ for $s\geq n\big|{1\over  2}-{1\over  p}\big|$ (see Theorem \ref{operatorthm} for more details and Remark \ref{rmkey} for more comments). There are many concrete generators fulfilling all of our assumptions, such as Neumann-Laplacian operator $\Delta_N$ (resp. $\Delta_{{N_\pm}}$) on $\mathbb{R}^n$ (resp. $\mathbb{R}^n_\pm$) \cite{DDSY}, Dunkl operators $\Delta_{{\rm Dunkl}}$ on $\mathbb{R}^N$ equipped with Dunkl measure \cite{MR3989141}, and Laplace-Beltrami operator $\Delta_g$ on $n$-dimensional complete Riemannian manifold with non-negative Ricci curvature \cite[Section 5]{MR1103113}.

{\bf Example C. Quantum Euclidean spaces}

Quantum Euclidean spaces (also called Moyal spaces) are one of the archetypal algebras of noncommutative geometry, which were first introduced by many authors, such as
Groenewold \cite{MR18562} and Moyal \cite{MR29330}, for the study of quantum mechanics in phase space. Besides, being a kind of geometrical spaces with noncommutating spatial coordinates, these spaces have appeared frequently in several branches of mathematical physics \cite{MR1878801}, including string theory \cite{MR1720697} and noncommutative field theory \cite{MR1670037}.

Recently, many breakthroughs have been obtained in several branches of harmonic analysis over quantum Euclidean space, such as singular integral theory \cite{MR4320770}, function space theory \cite{MR4320770,MR3778570} and quantum differentiability \cite{MR4156216}. To illustrate our application on this space, we let $\mathcal{R}_\Theta$ be a quantum Euclidean space associated with an anti-symmetric real-valued $n\times n$ matrix $\Theta$. Denote $\Delta_\Theta$ be the natural quantum Laplacian operator over it (more details are referred to Section \ref{QES}). Then our algebraic theory can be applied to establish the $L_p(\mathcal{R}_\Theta)$-boundedness of $(id_{\mathcal{R}_\Theta}-\Delta_\Theta)^{-s }e^{-it\Delta_\Theta} $ for $s\geq n\big|{1\over  2}-{1\over  p}\big|$ (see Theorem \ref{qusch}).

{\bf Example D. Matrix algebras}

Denote $L_{B(\ell_2)}$ be the non-negative infinitesimal generator induced from a Markov semigroup of convolution type on $B(\ell_2)$ defined by
$$S_t\Big(\sum_{m,k}a_{m,k}e_{m,k}\Big)=\sum_{m,k}e^{-t|m-k|^2}a_{m,k}e_{m,k},$$
where $e_{m,k}$ is the matrix unit for $B(\ell_2)$. Then our algebraic theory can be applied to establish the $S_p$-boundedness of $(id_{B(\ell_2)}+L_{B(\ell_2)})^{-1/2 }e^{itL_{B(\ell_2)}}$ for $s\geq \big|{1\over  2}-{1\over  p}\big|$ (see Theorem \ref{bl2}), where $S_p$ denotes the Schatten-$p$ class on the Hilbert space $\ell_2$.

{\bf Example E. Group von Neumann algebras}

The interest of Fourier multipliers over group von Neumann algebras dates back to a pioneering work by Haarerup on free groups \cite{MR520930}. Fourier multiplier is an effective
tool to reveal the geometry of these algebras via approximation properties \cite{MR996553,MR784292,MR520930,MR3047470}. It is worthwhile to note that there have been fruitful results about Fourier multiplier $L_p$-theory over group von Neumann algebras in the past decade (see for example \cite{MR4438906,CGPT,MR3679616,MR3283931,MR3776274}). In particular, by investigating  the $L_\infty\rightarrow BMO$ endpoint inequality, Junge--Mei--Parcet \cite{MR3283931} established $L_p$ H\"{o}rmander-Mihlin multiplier theorem with optimal smoothness condition on discrete group von Neumann algebra equipped with finite-dimensional cocycles. To illustrate our application on these algebras, we let $\mathcal{L}(G)$ be the group von Neumann algebra of a discrete group $G$ equipped with $n$-dimensional cocycle.
Then our algebraic theory can be applied to obtain the $L_p(\mathcal{L}(G))$-boundedness of Fourier multiplier operator with multiplier function $e^{it\psi(g)}(1+\psi(g))^{-s}$ for $s\geq n\big|{1\over  2}-{1\over  p}\big|$ (see Theorem \ref{aser}), where $\psi:G\rightarrow \mathbb{R}_+$ is a length function on $G$. This provides an alternative proof for a specific radial Fourier multiplier of \cite[Theorem D]{MR3283931} (see also the first paragraph in the proof of Theorem \ref{aser}).

\subsection{Notation and structure of the paper}
We use $A\lesssim B$ to denote the statement that $A\leq CB$ for some constant $C>0$, and $A\simeq B$ to denote the statement that $A\lesssim B$ and $B\lesssim A$.

This paper is organized as follows. In Section \ref{222}, we first introduce the  concept of P-metric and then illustrate how metrics in homogeneous space fit in. Secondly, we introduce non-commutative $M$-order semigroup BMO space together with the associated metric-type BMO space, and then show an embedding theorem between them. Thirdly, we recall some basic concepts about operator-valued weight. Finally, we formulate the algebraic conditions, analytic conditions and operator conditions in the full generality, and then build the bridge between low-order metric BMO space and the higher one under these assumptions. In Section \ref{333}, we provide the proof of Theorem \ref{main}. In Section \ref{Appi}. we provide several  applications of Theorem \ref{main} by verifying the imposed conditions in different settings.

\bigskip
\section{P-metrics and non-commutative BMO spaces}\label{222}
\setcounter{equation}{0}
\subsection{P-metrics}\label{Hilbert modules}
This subsection will provide more explanation about the P-metric introduced in the introduction.
One shall see below that the notion P-metric is easily understood for commutative model case. We illustrate this in a more general von-Neumann algebra $L_\infty(X)$ by constructing  a P-metric on $L_\infty(X)$. To this end, let $\xi$ be a bounded complex-valued function on $X$. Then it satisfies the following semigroup majorization
\begin{align*}
\int_{X}p_{t}(x,y)|\xi(y)|^2d\mu(y)
&\leq C \sum_{j\geq 0}\frac{\mu(B(x,2^jt^{1/m}))}{\mu(B(x,t^{1/m}))}\exp\big(-c2^{\frac{jm}{m-1}}\big)\fint_{B(x,2^jt^{1/m})}|\xi(y)|^2d\mu(y)
\end{align*}
for some $C,c>0$ independent of $x\in X$ and $t>0$. Here we used the fact that $\exp\big(-c2^{\frac{jm}{m-1}}\big)\simeq 1$ when $j=1,2$.
This means that  $R_{j,t} \xi(x)$ is the average of $\xi\in L_\infty(X)$ over the ball $B(x,2^{j}t)$. Next, note that $B(x,2^jt^{1/m})$ can be covered by at most $c^{jn}$ balls with radius $t^{1/m}$ and with finite overlap. For simplicity, we denote these balls by $B_{j,1}(x,t^{1/m}),\cdots,B_{j,d_{j,n}}(x,t^{1/m})$, where $d_{j,n}\lesssim c^{jn}$. Such a cover is not unique, but we fix a kind of cover. Then for any $j\geq 0$ and $t>0$, the following average  domination inequality holds:
\begin{align*}
\left\|\fint_{B(x,2^jt^{1/m})}|\xi(y)|^2d\mu(y)\right\|_{L_\infty(X)}
&\leq \left\|\sum_{\ell=1}^{d_{j,n}}\frac{\mu(B_{j,\ell}(x,t^{1/m}))}{\mu(B(x,2^jt^{1/m}))}\fint_{B_{j,\ell}(x,t^{1/m})}|\xi(y)|^2d\mu(y)\right\|_{L_\infty(X)}\\
&\leq \left\|\sum_{\ell=1}^{d_{j,n}}\frac{\mu(B_{j,\ell}(x,t^{1/m}))}{\mu(B(x,2^jt^{1/m}))}\right\|_{L_\infty(X)}\left\|\fint_{B(x,t^{1/m})}|\xi(y)|^2d\mu(y)\right\|_{L_\infty(X)}\\
&\leq C\left\|\fint_{B(x,t^{1/m})}|\xi(y)|^2d\mu(y)\right\|_{L_\infty(X)}.
\end{align*}
If we choose $$|\sigma_{j,t}(x)|^2=C\frac{\mu(B(x,2^jt^{1/m}))}{\mu(B(x,t^{1/m}))}\exp\big(-c2^{\frac{jm}{m-1}}),$$ then by the doubling condition \eqref{doubling},
\begin{align*}
k_{\mathcal{Q}}
=C\left\|\sum_{j\geq 0}\frac{\mu(B(x,2^jt^{1/m}))}{\mu(B(x,t^{1/m}))}\exp\big(-c2^{\frac{jm}{m-1}}\big)\right\|_{L_\infty(X)}^{1/2}\leq C\left(\sum_{j\geq 0}2^{jn}\exp\big(-c2^{\frac{jm}{m-1}}\big)\right)^{1/2}<\infty,
\end{align*}
which verifies the conditions (i)-(iii) and then ends the construction of P-metric for $L_\infty(X)$.

\subsection{Semigroup $BMO_\mathcal{S}$}\label{semigroupBMO}
Given $M\in\mathbb{Z}_{+}$ and a noncommutative measure space $(\mathcal{M},\tau)$ together with a B-submarkovian semigroup $\mathcal{S}=(S_t)_{t\geq 0}$ acting on $(\mathcal{M},\tau)$,   we define the semigroup $BMO_{\mathcal{S},M}(\mathcal{M})$ semi-norm for $f\in \mathcal{M}$ as
\begin{align*}
\|f\|_{{ BMO}_{\mathcal{S},M}(\mathcal{M})}:=\max\big\{\|f\|_{{BMO}_{\mathcal{S},M}^r(\mathcal{M})},\|f\|_{{ BMO}_{\mathcal{S},M}^c(\mathcal{M})}\big\},
\end{align*}
where the column BMO semi-norm is given by
\begin{align*}
&\|f\|_{{ BMO}_{\mathcal{S},M}^c(\mathcal{M})}:=\sup_{t> 0}\big\|S_t|(I-S_t)^Mf|^2\big\|_{\mathcal{M}}^{1/2}
\end{align*}
and the row BMO semi-norm is given by $\|f\|_{{ BMO}_{\mathcal{S},M}^r(\mathcal{M})}:=\|f^*\|_{{ BMO}_{\mathcal{S},M}^c(\mathcal{M})}.$
For simplicity, we denote ${BMO}_{\mathcal{S}}^{\dagger}(\mathcal{M}):={ BMO}_{\mathcal{S},1}^{\dagger}(\mathcal{M})$ and ${ BMO}_{\mathcal{S}}(\mathcal{M}):={ BMO}_{\mathcal{S},1}(\mathcal{M})$, where $\dagger\in\{c,r\}$. Here we say that $\|\cdot\|_{BMO_{\mathcal{S},M}(\mathcal{M})}$ and $\|\cdot\|_{{ BMO}_{\mathcal{S},M}^{\dagger}(\mathcal{M})}$ are semi-norms on $\mathcal{M}$ since one can establish the following lemma:
\begin{lemma}\label{HM}
Let $\mathcal{M},\mathcal{N}$ be two von Neumann algebras, and  $T$ be a positive linear operator from $\mathcal{N}$ to $\mathcal{M}$. Then
$\|\cdot\|_Z$ is a semi-norm  on $\mathcal{N}$, where $\|x\|_Z:=\|T(x^*x)\|_\mathcal{M}^{1/2}$.
  \end{lemma}
\begin{proof}Let $\phi$ be a state on $\mathcal{M}$.
   For any $x,y\in\mathcal{N}$, we define $\tilde{\phi}(x,y):=\phi(T(x^*y))$, then one could easily deduce the following  Cauchy-Schwarz inequality:
    \begin{equation}\label{CSI}
      |\tilde{\phi}(x,y)|^2\leq \tilde{\phi}(x,x)\tilde{\phi}(y,y),
    \end{equation}
    which implies that
    $\tilde{\phi}(x+y,x+y)^{1/2}\leq \tilde{\phi}(x,x)^{1/2}+\tilde{\phi}(y,y)^{1/2}.$
    Therefore,
    \begin{align*}
    \|x+y\|_Z&=\sup_{ \phi \text{ is a state}}\phi(T((x+y)^*(x+y)))^{1/2}\leq \|x\|_Z+\|y\|_Z,
    \end{align*}
  which implies  the desired result.
  \end{proof}

   After quotienting out the null space of the semi-norm $\|\cdot\|_{{ BMO}_{\mathcal{S},M}^c(\mathcal{M})}$ on $\mathcal{M}$, we get a normed vector space which is denoted by $\Big(\mathcal{M}, \|\cdot\|_{{ BMO}_{\mathcal{S},M}^c(\mathcal{M})}\Big)$. Now we need to complete this space with respect to some suitable topology as in \cite{MR4437355,MR2885593,MR4178915} via using the concept of Hilbert module (see \cite{MR355613}).

Given a complete positive map $S_t$ from $\mathcal{M}$ to $\mathcal{M}$, one can construct the Hilbert module  $\mathcal{M}\bar{\otimes}_{S_t}\mathcal{M}$ as follows:
let $\langle\cdot, \cdot\rangle_{S_t}$ be  the $\mathcal{M}$-valued inner product on $\mathcal{M}\otimes\mathcal{M}$ which is defined by
  \begin{align*}
  \Big\langle \sum_{i}a_i\otimes b_i,\sum_{j}c_j\otimes d_j\Big\rangle_{S_t}:=\sum_{i,j}b_i^*S_t(a_i^*c_j)d_j
   \end{align*}
 for $a_i,b_i,c_j,d_j \in\mathcal{M}$. Then we consider the semi-norm $\|\cdot\|_{\mathcal{M}\bar{\otimes}_{S_t}\mathcal{M}}$ on $\mathcal{M}\otimes\mathcal{M}$ determined by  the $\mathcal{M}$-valued inner product
 \begin{align*}
  \|x\|_{\mathcal{M}\bar{\otimes}_{S_t}\mathcal{M}}:=\|\langle x,x\rangle_{S_t}^{1/2}\|_\mathcal{M}.
 \end{align*}
 After quotienting out the null space of the semi-norm, we obtain a normed vector space denoted by $\Big(\mathcal{M}\otimes\mathcal{M}, \|\cdot\|_{\mathcal{M}\bar{\otimes}_{S_t}\mathcal{M}}\Big)$, and then  $\mathcal{M}\bar{\otimes}_{S_t}\mathcal{M}$ denotes
 the completion of  $\Big(\mathcal{M}\otimes\mathcal{M}, \|\cdot\|_{\mathcal{M}\bar{\otimes}_{S_t}\mathcal{M}}\Big)$  with respect to the strong operator topology, i.e. the coarsest topology making the maps on $\mathcal{M}\otimes\mathcal{M}$
\begin{equation*}
  x\rightarrow \tau(a\langle x,x\rangle_{S_t})
\end{equation*}
 continuous for all $a\in L_1(\mathcal{M})$.

 Let $u:\Big(\mathcal{M}, \|\cdot\|_{{ BMO}_{\mathcal{S},M}^c(\mathcal{M})}\Big)\rightarrow\bigoplus_{t>0}\mathcal{M}\bar{\otimes}_{S_t}\mathcal{M}$ be the isometric map given by $$u(x):=\left((I-S_t)^Mx)\otimes 1\right)_{t>0}.$$
Then the column BMO space ${ BMO}_{\mathcal{S},M}^{c}(\mathcal{M})$ is defined as the weak-$*$ closure of $u(\mathcal{M})$ in $\bigoplus_{t>0}\mathcal{M}\bar{\otimes}_{S_t}\mathcal{M}$, see \cite{MR4178915} for more details.
 ${BMO}_{\mathcal{S},M}^{r}(\mathcal{M})$ is defined as the space consisting of the elements whose adjoint belongs to ${ BMO}_{\mathcal{S},M}^{c}(\mathcal{M})$. And ${BMO}_{\mathcal{S},M}(\mathcal{M})$ is defined to be the intersection of the row and column BMO spaces.

\subsection{Metric type $BMO_{\mathcal{Q},M}$}
Inspired by \cite{MR4178915},  we introduce a suitable metric type BMO space for von Neumann algebra and relate it with the semigroup type BMO spaces defined in Section \ref{semigroupBMO}. Given a bounded semigroup $\mathcal{S}=(S_t)_{t\geq 0}$ acting on $(\mathcal{M},\tau)$, consider a P-metric $\mathcal{Q}=\{(R_{j,t},\sigma_{j,t}):j\in\mathbb{N},t>0\}$ given in Section \ref{Hilbert modules}. For any $M\in \mathbb{Z}_+$, we define the semigroup $BMO_{\mathcal{Q},M}(\mathcal{M})$ semi-norm for $f\in \mathcal{M}$ as $$\|f\|_{BMO_{\mathcal{Q},M}(\mathcal{M})}=\max\big\{\|f\|_{BMO_{\mathcal{Q},M}^c(\mathcal{M})},\|f\|_{BMO_{\mathcal{Q},M}^r(\mathcal{M})}\big\},$$ where the column BMO semi-norm is given by
\begin{align*}
\|f\|_{BMO_{\mathcal{Q},M}^c(\mathcal{M})}:=\sup_{t> 0}\big\|R_{\sqrt[m]{t}}|(I-S_t)^Mf|^2\big\|_{\mathcal{M}}^{1/2}
\end{align*}
and the row BMO semi-norm is given by $\|f\|_{{BMO}_{\mathcal{Q},M}^r(\mathcal{M})}:=\|f^*\|_{{ BMO}_{\mathcal{Q},M}^c(\mathcal{M})}$. For simplicity, we denote ${ BMO}_{\mathcal{Q}}^{\dagger}(\mathcal{M}):={ BMO}_{\mathcal{Q},1}^{\dagger}(\mathcal{M})$ and ${ BMO}_{\mathcal{Q}}(\mathcal{M}):={ BMO}_{\mathcal{Q},1}(\mathcal{M})$, where $\dagger\in\{c,r\}$.

 With the P-metric, the metric type BMO space for classical model $L_\infty(X)$ equipped with the semigroup $\{S_t\}_{t\geq 0} $, where $L$ is a non-negative self-adjoint operator and its associated heat kernel satisfies $({\rm GE}_m)$, is determined by
\begin{align*}
\|f\|_{BMO_{\mathcal{Q}}}:=\sup_{t\geq 0}\left(\sup_{y\in X}\fint_{B(y,\sqrt[m]{t})}|(I-S_t)f(x)|^2dx\right)^{1/2},
\end{align*}
which induces the $BMO_L(X)$ space considered in \cite{MR2163867,MR2162784,MR2868142}.

Next, recalling that the value of the constant $k_\mathcal{Q}$ in the definition of P-metric, we show that for any $M\in\mathbb{Z}_+$, $BMO_{\mathcal{Q},M}$ embeds in $BMO_{\mathcal{S},M}$.
\begin{lemma}\label{embed}
Let $(\mathcal{M},\tau)$ be a noncommutative measure space equipped with a bounded semigroup $\mathcal{S}=(S_t)_{t\geq 0}$ and a P-metric $\mathcal{Q}=\{(R_{j,t},\sigma_{j,t}):j\in\mathbb{N},t>0\}$. Then for any $M\in\mathbb{Z}_+$ and $f\in\mathcal{M}$,
\begin{align*}
\|f\|_{{BMO}_{\mathcal{S},M}^c}\lesssim k_{\mathcal{Q}}\|f\|_{{ BMO}_{\mathcal{Q},M}^c}.
\end{align*}
\end{lemma}
\begin{proof}
By the definition of P-metric defined in section \ref{Hilbert modules}, one has
\begin{align*}
\|f\|_{{ BMO}_{\mathcal{S},M}^c}
&=\sup_{t>0}\big\|S_t|(I-S_t)^Mf|^2\big\|_{\mathcal{M}}^{1/2}\\
&\leq\sup_{t>0}\Big\|\sum_{j\geq 0}\sigma_{j,t}^*R_{j,\sqrt[m]{t}}(|(I-S_t)^Mf|^2)\sigma_{j,t}\Big\|_{\mathcal{M}}^{1/2}\\
&\leq\sup_{t>0}\Big\|\sum_{j\geq 0}|\sigma_{j,t}|^2\big\|R_{j,\sqrt[m]{t}}(|(I-S_t)^Mf|^2)\big\|_\mathcal{M}\Big\|_{\mathcal{M}}^{1/2}\\
&\leq\sup_{t>0}\Big\|\sum_{j\geq 0}|\sigma_{j,t}|^2\Big\|_{\mathcal{M}}^{1/2}\sup_{t>0}\Big\|R_{\sqrt[m]{t}}(|(I-S_t)^Mf|^2)\Big\|_{\mathcal{M}}^{1/2}\\
&\lesssim k_{\mathcal{Q}}\|f\|_{{ BMO}_{\mathcal{Q},M}^c}.
\end{align*}
This ends the proof of Lemma \ref{embed}.
\end{proof}
\subsection{Operator-valued weights}\label{Operator-valued weights}
In this subsection we collect the definition and basic properties of operator-valued weights from \cite{MR534673,MR549119} (see also \cite{MR4178915} for the summary).

Let $\mathcal{M}$ be a von Neumann algebra, then the extended positive part $\widehat{\mathcal{M}}_+$  is the set consisting of all lower semi-continuous maps $m:\mathcal{M}_{*,+}\rightarrow [0,\infty]$ satisfying the linear law: for $a,b\geq 0$ and $f,g\in\mathcal{M}_{*,+}$, one has
$$m(a f+b g)=a m(f)+b m(g).$$
The extended positive part is closed under addition, scalar multiplication, increasing limits and conjugation map defined by $x\mapsto a^*xa$, for any $a\in\mathcal{M}$. It is direct that $\mathcal{M}_+$ lies in $\widehat{\mathcal{M}}_+$. Recall also that when $\mathcal{M}$ is abelian, then indeed one has $\mathcal{M}\simeq L_\infty(\Omega,\mu)$ for some measure space $(\Omega,\mu)$. Besides, in this case, the extended positive part corresponds to the set of $\mu$-measurable functions on $\Omega$ (module sets of zero measure) with values in $[0,\infty]$.

Operator-valued weights appears as ''unbounded conditional expectations''. The simplest nontrivial example about this concept is probably the partial trace $E_\mathcal{M}={\rm tr}_\mathcal{A}\otimes id_\mathcal{M}$, where $\mathcal{A}$ is a semifinite non-finite von Neumann algebra. Generally, an operator-valued weight is a kind of generalized conditionally expectation from a von Neumann algebra $\mathcal{N}$ to a von Neumann subalgebra $\mathcal{M}$. Specifically, it is a linear map from $\mathcal{N}_+$ to $\widehat{\mathcal{M}_+}$ satisfying
\begin{align}\label{opw}
E_\mathcal{M}(a^* fa)=a^*E_\mathcal{M}(f)a
\end{align}
for all $a\in\mathcal{M}$. $E_\mathcal{M}$ is called normal when it satisfies $\sup\limits_\alpha E_\mathcal{M}(f_\alpha)=E_\mathcal{M}(\sup\limits_\alpha f_\alpha)$ for bounded increasing nets $(f_\alpha)$ in $\mathcal{N}_+$.

Note that when $\mathcal{M}=\mathbb{C}$, the operator-valued weight $E_\mathcal{M}$ becomes an ordinary weight on $\mathcal{N}$. In analogy with ordinary weights, we consider
\begin{align*}
L_\infty^c(\mathcal{N};E_\mathcal{M})=\Big\{f\in\mathcal{N}:\|E_{\mathcal{M}}(f^*f)\|_{\mathcal{M}}<\infty\Big\}
\end{align*}
endowed with the natural (semi-)norm
\begin{align*}
\|f\|_{L_\infty^c(\mathcal{N};E_\mathcal{M})}:=\|E_{\mathcal{M}}(f^*f)\|_{\mathcal{M}}^{1/2}.
\end{align*}
Note that in the case of $E_{\mathcal{M}}=tr_{\mathcal{A}}\otimes id_{\mathcal{M}}$ with $\mathcal{N}=\mathcal{A}\botimes \mathcal{M}$, $L_{\infty}^c(\mathcal{N};E_{\mathcal{M}})$ becomes the Hilbert space valued noncommutative $L_\infty$ spaces $L_2^c(\mathcal{A})\botimes \mathcal{M}$ (see \cite{MR4178915}), which is defined in \cite{MR2265255}.

By identifying $\mathcal{M}$ as a von Neumann subalgebra of $B(L_2(\mathcal{M}))$ and then applying inequality \eqref{CSI}, one has the following equivalent (semi-)norms
\begin{align}\label{keyine}
\|f\|_{L_\infty^c(\mathcal{N}_\rho;E_\rho)}=\sup\limits_{\|b\|_{L_2(\mathcal{M})}=1}\tau(b^*E_\rho(|f|^2)b)^{\frac 12 },
\end{align}
which is an effective tool to deal with the strong operator topology defined in the next subsection.

Let $\mathcal{N}_{E_\mathcal{M}}$ be the linear span of $f_1^*f_2$ with $f_1$, $f_2\in L_\infty^c(\mathcal{N};E_\mathcal{M})$. Then one has the following properties:

(i) $\mathcal{N}_{E_\mathcal{M}}$=span$\{f\in\mathcal{N}_{+}:\|E_{\mathcal{M}}f\|<\infty\}$;

(ii) $L_\infty^c(\mathcal{N};E_\mathcal{M})$ and $\mathcal{N}_{E_\mathcal{M}}$ are two-sided modules over $\mathcal{M}$;

(iii) $E_\mathcal{M}$ has a unique linear extension $E_\mathcal{M}:\mathcal{N}_{E_\mathcal{M}}\rightarrow \mathcal{M}$ satisfying the bimodularity. That is, for any $f\in\mathcal{N}_{E_\mathcal{M}}$ and $a,b\in\mathcal{M}$,
$$E_\mathcal{M}(afb)=aE_\mathcal{M}(f)b.$$

\subsection{Algebraic, analytic, topological and operator conditions}\label{aao}
In what follows, we shall assume that the completely positive unital map $R_r$ from the P-metric $\mathcal{Q}$ defined in Section \ref{Hilbert modules} is of the following form:
\begin{align}\label{hhhh}
&\mathcal{M}\stackrel{\rho_j}{\longrightarrow}\mathcal{N}_\rho\stackrel{E_\rho}{\longrightarrow}\rho_1(\mathcal{M})\simeq \mathcal{M}\nonumber\\
R_{r}f&=E_{\rho}(q_{r})^{-\frac{1}{2}}E_{\rho}(q_{r}\rho_2(f)q_{r})E_{\rho}(q_{r})^{-\frac{1}{2}},
\end{align}
where $\rho_1,\rho_2:\mathcal{M}\rightarrow \mathcal{N}_\rho$ are unital injective $*$-homomorphisms into certain von Neumann algebra $\mathcal{N}_\rho$ and the map $E_\rho:\mathcal{N}_\rho\rightarrow \rho_1(\mathcal{M})$ is a normal operator-valued weight defined in Section \ref{Operator-valued weights}. The elements $q_{r}$ are increasing projections with respect to $r$ in $\mathcal{N}_\rho$ and satisfy some properties determined later on. 
 In particular, we shall assume that $q_{r}$ and $q_{r}\mathcal{N}_\rho q_r$ belong to the domain of $E_\rho$ so that our formula for $R_{r}f$ make sense. Our model provides a quite general form of P-metric which includes the P-metric for the heat semigroup considered before. Indeed, in the model case one could take $\mathcal{N}_\rho=L_\infty(\mathbb{R}^n\times \mathbb{R}^n)$ with $\rho_1(f)(x,y)=f(x)$ and $\rho_2(f)(x,y)=f(y)$. Let $E_\rho$ be the integral in
$\mathbb{R}^n$ with respect to the variable $y$ and set
 $q_{r}=\chi_{B(x,r)}(y).$
Then it is direct to check that we recover from \eqref{hhhh} the $R_{r}$'s for the heat semigroup. Note that the $q_{r}$ reproduce in this case all the Euclidean balls in $\mathbb{R}^n$.

Now we establish some preliminary algebraic, analytic and topological conditions on the P-metric. Let $C_b^\infty$ be  the space of bounded functions mapping $[ 0,\infty)$ to  $\mathbb{C}$ and having bounded derivatives at any order,
 we impose a priori assumption  that for any $F\in C_b^\infty$, the spectral multiplier $F(L):\mathcal{A}_\mathcal{M}\rightarrow \mathcal{A}_\mathcal{M}$ for some weak-$*$ dense subalgebra $\mathcal{A}_\mathcal{M}\subset \mathcal{M}$ such that $\mathcal{A}_\mathcal{M}$ is dense in $L_p(\mathcal{M})$ for any $1<p<\infty$. The above assumption is not a crucial condition as we mentioned in the introduction. This can be easily seen for the classical case $\mathcal{M}=L_\infty(\mathbb{R}^n)$ with $L=-\Delta$ and $\mathcal{A}_\mathcal{M}=\mathcal{S}(\mathbb{R}^n)$. For simplicity, in what follows we denote $V_\rho=L_\infty^c(\mathcal{N}_\rho;E_\rho)+\rho_2(\mathcal{A}_\mathcal{M})$ and $\mathcal{L}(X)$ be the space consisting of all linear operators over a linear space $X$.
  Assume that for any $F\in C_b^\infty$, there exists a homomorphism $\pi$: $\{F(L)\}_{F\in C_b^\infty}\subset B(L_2(\mathcal{M}))\rightarrow \mathcal{L}(V_\rho)$ such that the following intertwining equality holds: for any $F\in C_b^\infty$,
\begin{equation}\label{transfer}
 \pi(F(L))\circ \rho_2=\rho_2\circ F(L)\ {\rm on}\ \mathcal{A}_\mathcal{M}.
\end{equation}


{\bf Algebraic conditions}

(ALi) {\it $\mathcal{Q}$-monotonicity of $E_\rho$}
\begin{align*}
&E_\rho(b_{k,r}|\xi|^2 b_{k,r})\leq E_\rho(|\xi|^2)
\end{align*}
for any $k\geq 1$ and $\xi\in \mathcal{N}_\rho$, where
$$ b_{k,r}=\left\{\begin{array}{ll}q_{2r}, &{\rm if}\ k=1,\\q_{2^{k}r}-q_{2^{k-1}r}, &{\rm if}\ k\geq 2,\end{array}\right.$$
every projection $q_r$ determined by identity in $\eqref{hhhh}$.

(ALii) {\it Right modularity of $\pi(F(L))$}
\begin{align*}
\pi(F(L))(\eta E_\rho(q_r)^{-\frac 12})=\left(\pi(F(L))(\eta)\right)E_\rho(q_r)^{-\frac 12}
\end{align*}
for all $\eta\in L_\infty^c(\mathcal{N}_\rho;E_\rho)$,  $F\in C_b^\infty$ and $r>0$.
\begin{remark}
{\rm We will see below that all the algebraic conditions hold trivially in the classical theory, where the first one essentially says that the integration of a positive function over a ``ball'' or an ``annulus'' is always smaller than over the whole space, while the second condition says that one could take out $|B(x,r)|^{-1/2}$ from the $y$-dependent integral defining $F(L)$. Besides, the first condition also suggests that certain amount of commutativity for the projections is needed. In particular, we will frequently use the following inequality: for any $a\in\mathcal{N}_\rho$, $b\in\mathcal{M}$ and $r>0$,
\begin{align}\label{abcd}
\left\|aq_r b\right\|_{L_\infty^c(\mathcal{N}_{\rho};E_\rho)}\leq \left\|ab\right\|_{L_\infty^c(\mathcal{N}_{\rho};E_\rho)}.
\end{align}
The above inequality can be deduced, by the property of operator-valued weight \eqref{opw} together with condition (ALi), as follows:
$$E_\rho(b^* q_r |a|^2 q_r b)=b^* E_\rho(q_r |a|^2 q_r) b\leq b^* E_\rho(|a|^2) b=E_\rho(b^* |a|^2 b).$$}
\end{remark}

{\bf Topological conditions}

\begin{definition}
Let $\{f_k\}_k\subset L_\infty^c(\mathcal{N}_\rho;E_\rho)$ and $f\in L_\infty^c(\mathcal{N}_\rho;E_\rho)$. We say that $f_k$ converges to $f$ in the strong operator topology if for any $b\in L_2(\mathcal{M})$, we have
\begin{align*}
\lim_{k\rightarrow \infty}\tau\Big(b^*E_\rho((f_k-f)^*(f_k-f))b\Big)=0.
\end{align*}
\end{definition}

(Ti) {\it Continuity of $q_r$}

$q_s$ converges to $1$ in the sense that for any $r>0$, $F\in\mathcal{S}(\mathbb{R})$ and $f\in V_\rho$,
$$\pi(F(L))(fq_s)q_rE_\rho(q_r)^{-\frac 12}\rightarrow \pi(F(L))(f)q_rE_\rho(q_r)^{-\frac 12}\ {\rm in}\ {\rm the}\ {\rm strong}\ {\rm operator } \ {\rm topology}.
$$

(Tii) {\it Continuity of $\pi$}

If $\mathcal{S}(\mathbb{R})\ni F_k\rightarrow F\in\mathcal{S}(\mathbb{R})$ pointwise and $\sup\limits_{k}\|F_k\|_\infty<+\infty$, then for any $r>0$ and $f\in V_\rho$,
$$\pi(F_k(L))(f)q_rE_\rho(q_r)^{-\frac 12}\rightarrow \pi(F(L))(f)q_rE_\rho(q_r)^{-\frac 12}\  {\rm in}\ {\rm the}\ {\rm strong}\ {\rm operator } \ {\rm topology}.$$
%


{\bf Operator conditions}

%

(Oi) {\it Boundedness condition}
$$\|\pi(F(L))\|_{L_\infty^c(\mathcal{N}_{\rho};E_\rho)\rightarrow L_\infty^c(\mathcal{N}_{\rho};E_\rho)}\leq \|F\|_{L_{\infty}(\mathbb{R})}$$
for all  $F\in C_b^\infty$.

%

(Oii) {\it $L_2$-Gaussian estimate}

There exist constants $C,c>0$ and $m\geq 2$ such that for any $t>0$, $k\geq 1$, $r_2\geq r_1$ and $f\in\mathcal{N}_\rho$,
\begin{align}\label{preoff}
&\Big\|\left(\pi(e^{-tL})(fb_{k,r_2})\right)q_{r_1} E_\rho(q_{r_1})^{-\frac{1}{2}}\Big\|_{L_\infty^c(\mathcal{N}_{\rho};E_\rho)}\nonumber\\
&\leq CE_k\left(-c\left(\frac{r_2}{t^{1/m}}\right)^{\frac{m}{m-1}}\right)\Big\|fb_{k,r_2} E_\rho(q_{\max\{r_1,t^{1/m}\}})^{-\frac{1}{2}}\Big\|_{L_\infty^c(\mathcal{N}_{\rho};E_\rho)},
\end{align}
where
$$ E_{k}(x)=\left\{\begin{array}{ll}1, &{\rm if}\ k=1,\\ \exp(2^{\frac{mk}{m-1}}x), &{\rm if}\ k\geq 2.\end{array}\right.$$

%
%
%

\begin{remark}{\rm
We shall see in Section \ref{Appi} that in the lack of good metrics on general von Neumann algebra, the $L_2$-Gaussian condition is a suitable replacement of pointwise Gaussian estimate \eqref{GE}.}
\end{remark}

{\bf Analytic condition}

(ANi)  {\it Strong doubling condition}

There exist constants $C,n>0$, $D\geq 0$, such that for any $r_1\geq r_2\geq r_3$ and $f\in\mathcal{A}_\mathcal{M}$,
\begin{align*}
\Bigg\|E_\rho\bigg(\bigg|\rho_2(f)q_{r_1}E_\rho(q_{r_3})^{-\frac{1}{2}}\bigg|^{2}\bigg)\Bigg\|_{\mathcal{M}}\leq C\left(\frac{r_1}{r_2}\right)^{D}\left(\frac{r_1}{r_3}\right)^{n}\Bigg\|E_\rho\bigg(\bigg|\rho_2(f)q_{r_2}E_\rho(q_{r_2})^{-\frac{1}{2}}\bigg|^{2}\bigg)\Bigg\|_{\mathcal{M}}.
\end{align*}
\begin{remark}{\rm
Due to the lack of concepts of pointwise and position in general metric space, the procedure of covering cannot be proceeded. Instead, in this condition, we apply an energy inequality to suggest certain kind of 'covering phenomenon' over the underlying space. More details could be found in the applications given in the last section. Besides, we would like to mention that by choosing the parameters $f=1$ and $r_1=r_2$, this condition implies the following {\it doubling inequality}:
for any $r_1\geq r_2$, there are constants $C,n>0$ such that
\begin{align}\label{ANi}
E_{\rho}(q_{r_{2}})^{-\frac{1}{2}}E_{\rho}(q_{r_{1}})E_{\rho}(q_{r_{2}})^{-\frac{1}{2}}\leq C\left(\frac{r_1}{r_2}\right)^{n}.
\end{align}
The smallest constant $n=n_{\bf q}$ such that the above inequality valid is called the {\it P-metric dimension} of $\mathcal{M}$.}
\end{remark}
%

$\bullet$ {\bf Note:} Throughout Sections \ref{aao} and \ref{edp}, unless we mention the contrary, we assume all the conditions stated as above.
%
%
%
\begin{lemma}\label{ccccmp}
There exists a constant $C>0$ such that for any $t>0$, $r>0$ and $f\in L_\infty^c(\mathcal{N}_\rho;E_\rho)$,
\begin{align*}
\Big\|\pi(e^{-tL})(f)q_{r} E_\rho(q_{r})^{-\frac{1}{2}}\Big\|_{L_\infty^c(\mathcal{N}_{\rho};E_\rho)}\leq C\Big\|f E_\rho(q_{\max\{r,t^{1/m}\}})^{-\frac{1}{2}}\Big\|_{L_\infty^c(\mathcal{N}_{\rho};E_\rho)}.
\end{align*}
\end{lemma}
\begin{proof}
By equality \eqref{keyine},
\begin{align}\label{fs}
{\rm LHS}&=\sup\limits_{\|b\|_{L_2(\mathcal{M})}=1}\Big|\tau\Big(b^* E_\rho\big(|(\pi(e^{-tL})(f)q_{r} E_\rho(q_{r}\big)^{-\frac{1}{2}})|^2\big)b\Big)\Big|^{1/2}.
\end{align}
Applying Lemma \ref{HM}, we see that for any $N>0$,
\begin{align}\label{wert}
&\Big|\tau\Big(b^* E_\rho\big(|(\pi(e^{-tL})(f)q_{r} E_\rho(q_{r}\big)^{-\frac{1}{2}})|^2\big)b\Big)\Big|^{1/2}\nonumber\\
&\leq \Big|\tau\Big(b^* E_\rho\big(|(\pi(e^{-tL})(f(1-q_{2^{N+1}\max\{r,t^{1/m}\}}))q_{r} E_\rho(q_{r}\big)^{-\frac{1}{2}})|^2\big)b\Big)\Big|^{1/2}\nonumber\\
&+\sum_{k=1}^{N}\Big|\tau\Big(b^* E_\rho\big(|(\pi(e^{-tL})(fb_{k,\max\{r,t^{1/m}\}})q_{r} E_\rho(q_{r}\big)^{-\frac{1}{2}})|^2\big)b\Big)\Big|^{1/2}.
\end{align}
 %
By inequalities \eqref{fs} and \eqref{wert} together with the topological condition (Ti), one can easily conclude that
\begin{align}\label{eeee}
\Big\|\pi(e^{-tL})(f)q_{r} E_\rho(q_{r})^{-\frac{1}{2}}\Big\|_{L_\infty^c(\mathcal{N}_{\rho};E_\rho)}\leq\sum_{k\geq 1}\Big\|\left(\pi(e^{-tL})(fb_{k,\max\{r,t^{1/m}\}})\right)q_{r} E_\rho(q_{r})^{-\frac{1}{2}}\Big\|_{L_\infty^c(\mathcal{N}_{\rho};E_\rho)}.
\end{align}
Applying the operator condition (Oii) and the algebraic condition (ALi), one see that
\begin{align*}
&{\rm RHS}\ {\rm of}\ \eqref{eeee}\\
&\leq C\sum_{k\geq 1}E_k(-c)\Big\|fb_{k,\max\{r,t^{1/m}\}} E_\rho(q_{\max\{r,t^{1/m}\}})^{-\frac{1}{2}}\Big\|_{L_\infty^c(\mathcal{N}_{\rho};E_\rho)}\\
&= C\sum_{k\geq 1}E_k(-c)\Big\|E_\rho(q_{\max\{r,t^{1/m}\}})^{-\frac{1}{2}}E_\rho(b_{k,\max\{r,t^{1/m}\}}|f|^2b_{k,\max\{r,t^{1/m}\}})E_\rho(q_{\max\{r,t^{1/m}\}})^{-\frac{1}{2}}\Big\|_{\mathcal{M}}^{\frac12}\\
&\leq C\sum_{k\geq 1}E_k(-c)\Big\|E_\rho(q_{\max\{r,t^{1/m}\}})^{-\frac{1}{2}}E_\rho(|f|^2)E_\rho(q_{\max\{r,t^{1/m}\}})^{-\frac{1}{2}}\Big\|_{\mathcal{M}}^{\frac12}\\
&\leq C\Big\|f E_\rho(q_{\max\{r,t^{1/m}\}})^{-\frac{1}{2}}\Big\|_{L_\infty^c(\mathcal{N}_{\rho};E_\rho)}.
\end{align*}
This ends the proof of Lemma \ref{ccccmp}.
\end{proof}

We need the following version of Phragmen-Lindel\"{o}f Theorem (see \cite[Lemma 9]{MR1346221}).
\begin{lemma}\label{PL0123}
Assume that $F$ is an analytic function in $\mathbb{C}_{+}:=\{z\in\mathbb{C}:{\rm Rez}>0\}$ satisfying
\begin{align*}
&|F(z)|\leq a_{1}({\rm Re}z)^{-b_{1}},\\
&|F(t)|\leq a_{1}t^{-b_{1}}{\rm exp}(-a_{2}t^{-b_{2}}),
\end{align*}
for some $a_{1}$, $a_{2}>0$, $b_{1}\geq 0$, $b_{2}\in(0,1]$ independent of $z\in\mathbb{C}_{+}$ and $t>0$. Then for any $z\in\mathbb{C}_{+}$,
\begin{align*}
|F(z)|\leq a_{1}2^{b_{1}}({\rm Re}z)^{-b_{1}}{\rm exp}\bigg(-\frac{a_{2}b_{2}}{2}|z|^{-b_{2}-1}{\rm Re}z\bigg).
\end{align*}
\end{lemma}
We now apply Phragmen-Lindel\"{o}f Theorem to pass $L_2$-Gaussian estimate from $\mathbb{R}$ to $\mathbb{C}_+$.
\begin{lemma}\label{PL}
There is a constant $C>0$ such that for any $z\in\mathbb{C}_+$, $k\geq 1$, $0<r_1\leq r_2$ and $f\in\mathcal{N}_\rho$,
\begin{align}\label{cplex}
&\Big\|\left(\pi(e^{-zL})(fb_{k,r_2})\right)q_{r_1} E_\rho(q_{r_1})^{-\frac{1}{2}}\Big\|_{L_\infty^c(\mathcal{N}_{\rho};E_\rho)}\nonumber\\&\leq CE_k\left(-c\left(\frac{r_{2}}{|z|^{1/m}}\right)^{\frac{m}{m-1}}\cos\theta\right)\Big\|fb_{k,r_2} E_\rho(q_{\max\{r_1,({\rm Re}z)^{1/m}\}})^{-\frac{1}{2}}\Big\|_{L_\infty^c(\mathcal{N}_{\rho};E_\rho)},
\end{align}
where $\theta=\arccos \frac{{\rm Re}z}{|z|}$.
\end{lemma}
\begin{proof}
We first consider the proof for the case $k=1$. To this end, noting that the map $\pi$ is a homomorphism, we apply Lemma \ref{ccccmp} and the right modularity condition (ALii) together with the boundedness condition (Oi) to conclude that for any $z\in\mathbb{C}_+$, $r_1\geq 0$ and $f\in L_\infty^c(\mathcal{N}_\rho;E_\rho)$,
\begin{align}\label{diaess}
\Big\|\left(\pi(e^{-zL})(f)\right)q_{r_1} E_\rho(q_{r_1})^{-\frac{1}{2}}\Big\|_{L_\infty^c(\mathcal{N}_{\rho};E_\rho)}
&=\Big\|\left(\pi(e^{-\frac{{\rm Re}z}{2}L})\pi(e^{-(z-\frac{{\rm Re}z}{2})L})(f)\right)q_{r_1} E_\rho(q_{r_1})^{-\frac{1}{2}}\Big\|_{L_\infty^c(\mathcal{N}_{\rho};E_\rho)}\nonumber\\
&\leq C\Big\|\left(\pi(e^{-(z-\frac{{\rm Re}z}{2})L})(f)\right) E_\rho(q_{\max\{r_1,(\frac{{\rm Re}z}{2})^{1/m}\}})^{-\frac{1}{2}}\Big\|_{L_\infty^c(\mathcal{N}_{\rho};E_\rho)}\nonumber\\
&= C\Big\|\pi(e^{-(z-\frac{{\rm Re}z}{2})L})(fE_\rho(q_{\max\{r_1,(\frac{{\rm Re}z}{2})^{1/m}\}})^{-\frac{1}{2}}) \Big\|_{L_\infty^c(\mathcal{N}_{\rho};E_\rho)}\nonumber\\
&\leq C\Big\| fE_\rho(q_{\max\{r_1,(\frac{{\rm Re}z}{2})^{1/m}\}})^{-\frac{1}{2}}\Big\|_{L_\infty^c(\mathcal{N}_{\rho};E_\rho)}\nonumber\\
&\leq C\Big\| fE_\rho(q_{\max\{r_1,({\rm Re}z)^{1/m}\}})^{-\frac{1}{2}}\Big\|_{L_\infty^c(\mathcal{N}_{\rho};E_\rho)},
\end{align}
where in the last inequality we applied the definition of $L_\infty^c(\mathcal{N}_\rho;E_\rho)$ and the doubling inequality \eqref{ANi} to get rid of the harmless multiple $\frac{1}{2}$ from $(\frac{{\rm Re}z}{2})^{1/m}$, which implies the desired estimate for $k=1$.

To obtain the desired estimate for the case $k\geq 2$, we let $\phi$ be a state of $\mathcal{M}$ and $\tilde{\phi}$ be a linear functional from   $L_\infty^c\times L_\infty^c$ to $\mathbb{C}$ given by $\tilde{\phi}(x,y)=\phi(E_\rho(x^*y))$. Let $y\in L_\infty^c(\mathcal{N}_\rho;E_\rho)$ such that $\tilde{\phi}(y,y)=1$. For any  $R>0$, we define  an analytic function over right-half plane $\mathbb{C}_+$ by
\begin{align}
F(z):=e^{-Rz}\tilde{\phi}\left(\left(\pi(e^{-zL})(fb_{k,r_2})\right)q_{r_1} E_\rho(q_{r_1})^{-\frac{1}{2}},y\right).
\end{align}
By inequalities \eqref{CSI} and \eqref{diaess}, we see that for any $z\in\mathbb{C}_{+}$,
\begin{align*}
|F(z)|
&\leq e^{-R{\rm Re}z}\tilde{\phi}\left(\left(\pi(e^{-zL})(fb_{k,r_2})\right)q_{r_1} E_\rho(q_{r_1})^{-\frac{1}{2}},\left(\pi(e^{-zL})(fb_{k,r_2})\right)q_{r_1} E_\rho(q_{r_1})^{-\frac{1}{2}}\right)^{1/2}\\
&\leq e^{-R{\rm Re}z}\Big\|\left(\pi(e^{-zL})(fb_{k,r_2})\right)q_{r_1} E_\rho(q_{r_1})^{-\frac{1}{2}}\Big\|_{L_\infty^c(\mathcal{N}_{\rho};E_\rho)}\\
&\leq Ce^{-R{\rm Re}z}\Big\|fb_{k,r_2}E_\rho(q_{\max\{r_1,({\rm Re}z)^{1/m}\}})^{-\frac{1}{2}}\Big\|_{L_\infty^c(\mathcal{N}_{\rho};E_\rho)},
\end{align*}
which, together with the definition of $L_\infty^c(\mathcal{N}_{\rho};E_\rho)$ and the doubling inequality \eqref{ANi}, implies the following pair of inequalities
\begin{align}\label{comb1}
|F(z)|\leq C \Big\|fb_{k,r_2}E_\rho(q_{r_1})^{-\frac{1}{2}}\Big\|_{L_\infty^c(\mathcal{N}_{\rho};E_\rho)}
\end{align}
and
\begin{align}\label{comb2}
|F(z)|\leq Ce^{-R{\rm Re}z}\Big\|fb_{k,r_2}E_\rho(q_{({\rm Re}z)^{1/m}})^{-\frac{1}{2}}\Big\|_{L_\infty^c(\mathcal{N}_{\rho};E_\rho)}.
\end{align}
To continue, by \eqref{comb2} and the doubling inequality \eqref{ANi}, one has
\begin{align}\label{comb3}
|F(z)|&\leq Ce^{-R{\rm Re}z}\Big\|E_\rho(q_{({\rm Re}z)^{1/m}})^{-\frac{1}{2}}E_\rho(q_{R^{-1/m}})E_\rho(q_{({\rm Re}z)^{1/m}})^{-\frac{1}{2}}\Big\|_\mathcal{M}^{\frac12}\Big\|fb_{k,r_2}E_\rho(q_{R^{-1/m}})^{-\frac{1}{2}}\Big\|_{L_\infty^c(\mathcal{N}_{\rho};E_\rho)}\nonumber\\
&\leq Ce^{-R{\rm Re}z}\bigg(1+\frac{1}{(R{\rm Re}z)^{1/m}}\bigg)^{\frac n2}\Big\|fb_{k,r_2}E_\rho(q_{R^{-1/m}})^{-\frac{1}{2}}\Big\|_{L_\infty^c(\mathcal{N}_{\rho};E_\rho)}\nonumber\\
&\leq C(R{\rm Re}z)^{-\frac{n}{2m}}\Big\|fb_{k,r_2}E_\rho(q_{R^{-1/m}})^{-\frac{1}{2}}\Big\|_{L_\infty^c(\mathcal{N}_{\rho};E_\rho)}.
\end{align}

On the other hand, for any $s>0$,  by inequality \eqref{CSI} and $L_2$-Gaussian estimate (Oii),
 we have
\begin{align}
|F(s)|
&\leq Ce^{-R s}\exp\left(-c\left(\frac{2^kr_2}{s^{1/m}}\right)^{\frac{m}{m-1}}\right)\Big\|fb_{k,r_2}E_\rho(q_{\max\{r_{1},s^{1/m}\}})^{-\frac{1}{2}}\Big\|_{L_\infty^c(\mathcal{N}_\rho;E_\rho)},
\end{align}
which, together with the definition of $L_\infty^c(\mathcal{N}_{\rho};E_\rho)$ and the doubling inequality \eqref{ANi}, implies the following pair of inequalities
\begin{align}\label{pom1}
|F(s)|\leq C\exp\left(-c\left(\frac{2^kr_2}{s^{1/m}}\right)^{\frac{m}{m-1}}\right)\Big\|fb_{k,r_2}E_\rho(q_{r_{1}})^{-\frac{1}{2}}\Big\|_{L_\infty^c(\mathcal{N}_\rho;E_\rho)}
\end{align}
and
\begin{align}\label{pom2}
|F(s)|\leq Ce^{-R s}\exp\left(-c\left(\frac{2^kr_2}{s^{1/m}}\right)^{\frac{m}{m-1}}\right)\Big\|fb_{k,r_2}E_\rho(q_{s^{1/m}})^{-\frac{1}{2}}\Big\|_{L_\infty^c(\mathcal{N}_\rho;E_\rho)}.
\end{align}
To continue, by \eqref{pom2} and the doubling inequality \eqref{ANi}, one has
\begin{align}\label{pom3}
|F(s)|
&\leq C(Rs)^{-\frac{n}{2m}}\exp\left(-c\left(\frac{2^kr_2}{s^{1/m}}\right)^{\frac{m}{m-1}}\right)\Big\|fb_{k,r_2}E_\rho(q_{R^{-1/m}})^{-\frac{1}{2}}\Big\|_{L_\infty^c(\mathcal{N}_\rho;E_\rho)}.
\end{align}

Applying Phragm\'{e}n-Lindel\"{o}f Theorem \ref{PL0123} to inequalities \eqref{comb1} and \eqref{pom1}, we conclude that
\begin{align*}
|F(z)|\leq C\exp\left(-c\left(\frac{2^kr_2}{|z|^{1/m}}\right)^{\frac{m}{m-1}}\cos\theta\right)\Big\|fb_{k,r_2} E_\rho(q_{r_1})^{-\frac{1}{2}}\Big\|_{L_\infty^c(\mathcal{N}_{\rho};E_\rho)}.
\end{align*}
Since $F$ depends on $R$ but the constant $C$ is independent of $R$ and $z$, we choose $R=({\rm Re}z)^{-1}$ to get that
\begin{align}\label{Fz}
&\big|\tilde{\phi}\left(\left(\pi(e^{-zL})(fb_{k,r_2})\right)q_{r_1} E_\rho(q_{r_1})^{-\frac{1}{2}},y\right)\big|\nonumber\\
&\leq C\exp\left(-c\left(\frac{2^kr_2}{|z|^{1/m}}\right)^{\frac{m}{m-1}}\cos\theta\right)\Big\|fb_{k,r_2} E_\rho(q_{r_1})^{-\frac{1}{2}}\Big\|_{L_\infty^c(\mathcal{N}_{\rho};E_\rho)}.
\end{align}

Applying Phragm\'{e}n-Lindel\"{o}f Theorem \ref{PL0123} again to inequalities \eqref{comb3} and \eqref{pom3}, we deduce that
\begin{align} \label{Fz222}
|F(z)|\leq CR^{-\frac{n}{2m}}({\rm Re}z)^{-\frac{n}{2m}}\exp\left(-c\left(\frac{2^kr_2}{|z|^{1/m}}\right)^{\frac{m}{m-1}}\cos\theta\right)\Big\|fb_{k,r_2} E_\rho(q_{R^{-1/m}})^{-\frac{1}{2}}\Big\|_{L_\infty^c(\mathcal{N}_{\rho};E_\rho)}.
\end{align}
Choosing $R=({\rm Re}z)^{-1}$ again, we get that
\begin{align} \label{Fz22222}
&\big|\tilde{\phi}\left(\left(\pi(e^{-zL})(fb_{k,r_2})\right)q_{r_1} E_\rho(q_{r_1})^{-\frac{1}{2}},y\right)\big|\nonumber\\
&\leq C\exp\left(-c\left(\frac{2^kr_2}{|z|^{1/m}}\right)^{\frac{m}{m-1}}\cos\theta\right)\Big\|fb_{k,r_2} E_\rho(q_{({\rm Re}z)^{1/m}})^{-\frac{1}{2}}\Big\|_{L_\infty^c(\mathcal{N}_{\rho};E_\rho)}.
\end{align}

To continue, recall from {inequality \eqref{CSI}} that for $x\in L_\infty^c(\mathcal{N}_{\rho};E_\rho)$,
\begin{align*}
   \|x\|_{L_\infty^c(\mathcal{N_\rho},E_\rho)}=\|E_{\rho}(x^*x)\|_{\mathcal{M}}^{1/2}
  &=\sup_{\phi \text{ is a state}} \phi(E_{\rho}(x^*x))^{1/2}\\
  &=\sup_{\phi \text{ is a state}} \tilde{\phi}(x,x)^{1/2}\\
  &=\sup_{\phi \text{ is a state}}\sup_{\tilde{\phi}(y,y)=1} |\tilde{\phi}(x,y)|.
 \end{align*}
Substituting $x$ with $\left(\pi(e^{-zL})(fb_{k,r_2})\right)q_{r_1} E_\rho(q_{r_1})^{-\frac{1}{2}}$ into the above inequality and then using inequalities \eqref{Fz} and \eqref{Fz22222}, we deduce that
\begin{align}
&\Big\|\left(\pi(e^{-zL})(fb_{k,r_2})\right)q_{r_1} E_\rho(q_{r_1})^{-\frac{1}{2}}\Big\|_{L_\infty^c(\mathcal{N}_{\rho};E_\rho)}\nonumber\\&\leq C\exp\left(-c\left(\frac{2^kr_2}{|z|^{1/m}}\right)^{\frac{m}{m-1}}\cos\theta\right)\Big\|fb_{k,r_2} E_\rho(q_{r_1})^{-\frac{1}{2}}\Big\|_{L_\infty^c(\mathcal{N}_{\rho};E_\rho)}
\end{align}
and that
\begin{align}
&\Big\|\left(\pi(e^{-zL})(fb_{k,r_2})\right)q_{r_1} E_\rho(q_{r_1})^{-\frac{1}{2}}\Big\|_{L_\infty^c(\mathcal{N}_{\rho};E_\rho)}\nonumber\\&\leq C\exp\left(-c\left(\frac{2^kr_2}{|z|^{1/m}}\right)^{\frac{m}{m-1}}\cos\theta\right)\Big\|fb_{k,r_2} E_\rho(q_{({\rm Re}z)^{1/m}})^{-\frac{1}{2}}\Big\|_{L_\infty^c(\mathcal{N}_{\rho};E_\rho)},
\end{align}
which implies the estimate \eqref{cplex}. This ends the proof of Lemma \ref{PL}.
\end{proof}

Next we define a Besov type norm of $f$ by
$$\|F\|_{B^N}:=\int_{-\infty}^{\infty}|\hat{F}(\tau)|(1+|\tau|)^Nd\tau,$$
where $\hat{F}$ denotes the Fourier transform of $f$. It can be verified easily that
$$\|FG\|_{B^N}\leq \|F\|_{B^N}\|G\|_{B^N}.$$
Besides, for any $R>0$, we denote the $R$-dilation of a given function $F$ by
$\delta_R F(x):=F(Rx)$.
\begin{lemma}\label{PLPL}
There is a constant $C>0$ such that for any $R>0$, $k\geq 2$, $0<r_1\leq r_2$ and $f\in\mathcal{N}_\rho$,
\begin{align}\label{parti}
&\Big\|\left(\pi(F(L))(fq_{2r_2})\right)q_{r_1} E_\rho(q_{r_1})^{-\frac{1}{2}}\Big\|_{L_\infty^c(\mathcal{N}_{\rho};E_\rho)}\nonumber\\&\leq C\|\delta_RF\|_{B^0}\Big\|fq_{2r_2} E_\rho(q_{\max\{r_1,R^{-1/m}\}})^{-\frac{1}{2}}\Big\|_{L_\infty^c(\mathcal{N}_{\rho};E_\rho)}
\end{align}
and
\begin{align}\label{parti}
&\Big\|\left(\pi(F(L))(fb_{k,r_2})\right)q_{r_1} E_\rho(q_{r_1})^{-\frac{1}{2}}\Big\|_{L_\infty^c(\mathcal{N}_{\rho};E_\rho)}\nonumber\\&\leq C(\sqrt[m]{R}2^k r_2)^{-N}\|\delta_RF\|_{B^N}\Big\|fb_{k,r_2} E_\rho(q_{\max\{r_1,R^{-1/m}\}})^{-\frac{1}{2}}\Big\|_{L_\infty^c(\mathcal{N}_{\rho};E_\rho)}
\end{align}
for all smoothing functions $F$ such that supp$F\subset [-R,R]$.
\end{lemma}
\begin{proof}
Let $G(\lambda)=F(R\lambda)e^\lambda$. We apply the Fourier inversion formula
$$F(L)=G(L/R)e^{-L/R}=\frac{1}{2\pi}\int_{\mathbb{R}}e^{(i\tau -1)R^{-1}L}\hat{G}(\tau)d\tau$$
and the linearity of the amplification map $\pi$ together with the topological condition (Tii)  to get that
\begin{align}\label{convergeT}
\pi(F(L))(f)q_{r_1}E_\rho(q_{r_1})^{-\frac 12}=\frac{1}{2\pi}\int_{\mathbb{R}}\pi(e^{(i\tau -1)R^{-1}L})(f)q_{r_1}E_\rho(q_{r_1})^{-\frac 12}\hat{G}(\tau)d\tau,
\end{align}
where the improper integral converges in the strong operator topology.

To continue, we apply Lemma \ref{PL} with the choice $z=(1-i\tau)R^{-1}\in\mathbb{C}_+$ to deduce that  there exist constants $C,c>0$ such that for any $\tau\in\mathbb{R}$, $k\geq 1$, $R>0$, $0<r_1\leq r_2$ and $f\in\mathcal{N}_\rho$,
\begin{align}\label{parti}
&\Big\|\left(\pi(e^{(i\tau-1)R^{-1}L})(fb_{k,r_2})\right)q_{r_1} E_\rho(q_{r_1})^{-\frac{1}{2}}\Big\|_{L_\infty^c(\mathcal{N}_{\rho};E_\rho)}\nonumber\\&\leq CE_k\left(-c\left(\frac{\sqrt[m]{R}r_2}{\sqrt{1+\tau^2}}\right)^{\frac{m}{m-1}}\right)\Big\|fb_{k,r_2} E_\rho(q_{\max\{r_1,R^{-1/m}\}})^{-\frac{1}{2}}\Big\|_{L_\infty^c(\mathcal{N}_{\rho};E_\rho)}.
\end{align}
This, in combination with equalities \eqref{keyine} and \eqref{convergeT}, yields
\begin{align}\label{vvv1}
&\Big\|\left(\pi(F(L))(fb_{k,r_2})\right)q_{r_1} E_\rho(q_{r_1})^{-\frac{1}{2}}\Big\|_{L_\infty^c(\mathcal{N}_{\rho};E_\rho)}\nonumber\\
&\leq C\int_{\mathbb{R}}\Big\|\left(\pi(e^{(i\tau -1)R^{-1}L})(fb_{k,r_2})\right)q_{r_1} E_\rho(q_{r_1})^{-\frac{1}{2}}\Big\|_{L_\infty^c(\mathcal{N}_{\rho};E_\rho)}|\hat{G}(\tau)|d\tau\\
&\leq C\int_{\mathbb{R}}E_k\left(-c\left(\frac{\sqrt[m]{R}r_2}{\sqrt{1+\tau^2}}\right)^{\frac{m}{m-1}}\right)|\hat{G}(\tau)|d\tau \Big\|fb_{k,r_2} E_\rho(q_{\max\{r_1,R^{-1/m}\}})^{-\frac{1}{2}}\Big\|_{L_\infty^c(\mathcal{N}_{\rho};E_\rho)}.\nonumber
\end{align}
Note that the last term is dominated by $$C\|G\|_{B^0}\Big\|fq_{2r_2} E_\rho(q_{\max\{r_1,R^{-1/m}\}})^{-\frac{1}{2}}\Big\|_{L_\infty^c(\mathcal{N}_{\rho};E_\rho)}$$ if $k=1$ and $$C(\sqrt[m]{R}2^k r_2)^{-N}\|G\|_{B^N}\Big\|fb_{k,r_2} E_\rho(q_{\max\{r_1,R^{-1/m}\}})^{-\frac{1}{2}}\Big\|_{L_\infty^c(\mathcal{N}_{\rho};E_\rho)}$$ if $k\geq 2$.
To continue, note that supp$F\subset [-R,R]$ and so supp$F(R\cdot)\subset[-1,1]$. Hence taking a function $\psi\in C_c^\infty$ such that supp$\psi\subset[-2,2]$ and $\psi(\lambda)=1$ for $\lambda\in [-1,1]$, we have
\begin{align*}
G(\lambda)=F(R\lambda)e^\lambda=F(R\lambda)\psi(\lambda)e^\lambda,
\end{align*}
which implies that
\begin{align}\label{BNORM0}
\|G\|_{B^N}\leq C\|\delta_RF\|_{B^N}\|\psi(\lambda)e^\lambda\|_{B^N}\leq C\|\delta_RF\|_{B^N}.
\end{align}
This finishes the proof of Lemma \ref{PLPL}.
\end{proof}

\begin{proposition}\label{emb}
For any $M\in\mathbb{Z}_+$ and for all $f\in \mathcal{A}_\mathcal{M}$,
\begin{align}\label{pree}
\|f\|_{{ BMO}_{\mathcal{Q},1}^c}\simeq\|f\|_{{ BMO}_{\mathcal{Q},M}^c},
\end{align}
where the implicit constant depends on $M$.


\end{proposition}
\begin{proof}
We first show that for any $f\in \mathcal{A}_\mathcal{M}$,
\begin{align}\label{pree}
\|f\|_{{ BMO}_{\mathcal{Q},2}^c}\lesssim \|f\|_{{ BMO}_{\mathcal{Q},1}^c}.
\end{align}
To show this, we apply Kadison-Schwarz's inequality and Lemma \ref{embed} to deduce that
\begin{align*}
{\rm LHS}&\leq \sup\limits_{t>0}\big\|R_{\sqrt[m]{t}}|(I-e^{-tL})f|^2\big\|_\mathcal{M}^{1/2}+\sup\limits_{t>0}\big\|R_{\sqrt[m]{t}}|e^{-tL}(I-e^{-tL})f|^2\big\|_\mathcal{M}^{1/2}\\
&\leq \sup\limits_{t>0}\big\|R_{\sqrt[m]{t}}|(I-e^{-tL})f|^2\big\|_\mathcal{M}^{1/2}+\sup\limits_{t>0}\big\|e^{-tL}(I-e^{-tL})f\big\|_\mathcal{M}\\
&\lesssim \sup\limits_{t>0}\big\|R_{\sqrt[m]{t}}|(I-e^{-tL})f|^2\big\|_\mathcal{M}^{1/2}+\sup\limits_{t>0}\big\|e^{-tL}|(I-e^{-tL})f|^2\big\|_\mathcal{M}^{1/2}\\
&\lesssim \sup\limits_{t>0}\big\|R_{\sqrt[m]{t}}|(I-e^{-tL})f|^2\big\|_\mathcal{M}^{1/2}.
\end{align*}
This ends the proof of inequality \eqref{pree}. By induction on $M\geq 1$, one has
\begin{align*}
\|f\|_{{ BMO}_{\mathcal{Q},M}^c}\lesssim \|f\|_{{ BMO}_{\mathcal{Q},1}^c},
\end{align*}
where the implicit constant depends on $M$.

It remains to show the converse direction of the above inequality.
To begin with, for any $M\geq 2$, we observe that
\begin{align*}
I&=(I-e^{-tL})^{M-1}+(I-(I-e^{-tL})^{M-1})\\
&=(I-e^{-tL})^{M-1}+\sum_{k=1}^{M-1}(-1)^{k}\frac{(M-1)!}{k!(M-1-k)!}e^{-ktL}.
\end{align*}
Then
\begin{align*}
\left\|R_{\sqrt[m]{t}}|(I-e^{-tL})f|^2\right\|_{\mathcal{M}}^{1/2}
&\leq\left\|R_{\sqrt[m]{t}}|(I-e^{-tL})^Mf|^2\right\|_{\mathcal{M}}^{1/2}\\
&+\sum_{k=1}^{M-1}\frac{(M-1)!}{k!(M-1-k)!}\left\|R_{\sqrt[m]{t}}|e^{-ktL}(I-e^{-tL})f|^2\right\|_{\mathcal{M}}^{1/2}.
\end{align*}

This means that it suffices to show that for any $1\leq k\leq M-1$,
\begin{align}\label{sf}
\sup\limits_{t>0}\left\|R_{\sqrt[m]{t}}|e^{-ktL}(I-e^{-tL})f|^2\right\|_{\mathcal{M}}^{1/2}
\lesssim \sup\limits_{t>0}\left\|R_{\sqrt[m]{t}}|(I-e^{-tL})^Mf|^2\right\|_{\mathcal{M}}^{1/2}.
\end{align}
To show this, we first choose $\Phi$ be a smooth cut-off function supported in $[1/2,2]$ such that
\begin{align*}
\int_0^{\infty}\Phi(s\lambda)(I-e^{-s\lambda})^M\frac{ds}{s}=1,\ {\rm for}\ \lambda> 0.
\end{align*}
This, together with the topological condition (Tii), implies that
\begin{align}\label{convergeTT}
&\pi(e^{-ktL}(I-e^{-tL}))(\rho_2(f))q_{\sqrt[m]{t}}E_\rho(q_{\sqrt[m]{t}})^{-\frac 12}\nonumber\\&=\int_0^{\infty}\pi\Big(e^{-ktL}(I-e^{-tL})\Phi(sL)(I-e^{-sL})^M\Big)(\rho_2(f))q_{\sqrt[m]{t}}E_\rho(q_{\sqrt[m]{t}})^{-\frac 12}\frac{ds}{s},
\end{align}
where the improper integral converges in the strong operator topology.
This, together with the intertwining equality \eqref{transfer} and formula \eqref{keyine}, implies that
\begin{align}\label{djdj}
&\ \left\|R_{\sqrt[m]{t}}|e^{-ktL}(I-e^{-tL})f|^2\right\|_{\mathcal{M}}^{1/2}\nonumber\\
&=\left\|\pi(e^{-ktL}(I-e^{-tL}))(\rho_2(f))q_{\sqrt[m]{t}}E_\rho(q_{\sqrt[m]{t}})^{-\frac{1}{2}}\right\|_{L_\infty^c(\mathcal{N}_{\rho};E_\rho)}\nonumber\\
&\leq\int_{0}^{\infty} \left\|\pi\left(e^{-ktL}(I-e^{-tL})\Phi(sL)(I-e^{-sL})^M\right)(\rho_2(f))q_{\sqrt[m]{t}}E_\rho(q_{\sqrt[m]{t}})^{-\frac{1}{2}}\right\|_{L_\infty^c(\mathcal{N}_{\rho};E_\rho)}\frac{ds}{s}.
\end{align}

To continue, let $\psi\in C_c^\infty(\mathbb{R})$ supported in $\psi\in[1/8,4]$ and $\psi(\lambda)=1$ for $\lambda\in[1/2,2]$.
For simplicity, we denote $$F_{k,t,s}(\lambda):=e^{-kt\lambda}(I-e^{-t\lambda})\Phi(s\lambda),$$ which satisfies for any $N\geq 0$, there is a constant $C=C_N>0$ such that
\begin{align}\label{doBesov}
\|\delta_{s^{-1}}F_{k,t,s}\|_{B^N}
&\leq \left\|e^{-\frac{kt\lambda}{s}}(1-e^{-\frac{t\lambda}{s}})\Phi(\lambda)\right\|_{C^{N+2}}\nonumber\\
&\leq \left\|e^{-\frac{kt\lambda}{s}}\Phi(\lambda)\right\|_{C^{N+2}}\left\|(1-e^{-\frac{t\lambda}{s}})\psi(\lambda)\right\|_{C^{N+2}}\nonumber\\
&\leq C_N\min\left\{\Big(\frac{t}{s}\Big)^{-N},\Big(\frac{t}{s}\Big)\right\},
\end{align}
where $C^N$ is the space consisting of all $N$-order differentiable functions, endowed with the norm:
$$\|f\|_{C^N}:=\sum_{k=1}^N\left\|\frac{d^k}{dx^k}f\right\|_{L_\infty}.$$
By inequality \eqref{djdj},
\begin{align*}
\left\|R_{\sqrt[m]{t}}|e^{-ktL}(I-e^{-tL})f|^2\right\|_{\mathcal{M}}^{1/2}
&\leq \int_0^t\left\| \left(\pi(F_{k,t,s}(L))\rho_2((I-e^{-sL})^Mf)\right)q_{\sqrt[m]{t}}E_\rho(q_{\sqrt[m]{t}})^{-\frac{1}{2}}\right\|_{L_\infty^c(\mathcal{N}_{\rho};E_\rho)}\frac{ds}{s}\\
&+ \int_{t}^\infty\left\| \left(\pi(F_{k,t,s}(L))\rho_2((I-e^{-sL})^Mf)\right)q_{\sqrt[m]{t}}E_\rho(q_{\sqrt[m]{t}})^{-\frac{1}{2}}\right\|_{L_\infty^c(\mathcal{N}_{\rho};E_\rho)}\frac{ds}{s}\nonumber\\
&=:{\rm I}+{\rm II}.
\end{align*}

To estimate term ${\rm I}$,  we follow the estimate of \eqref{eeee} and then apply Lemma \ref{PLPL} to conclude that
\begin{align}
{\rm I}&\leq \int_{0}^{t}\sum_{j\geq 1} \left\| \left(\pi(F_{k,t,s}(L))(\rho_2((I-e^{-sL})^Mf)b_{j,\sqrt[m]{t}})\right)q_{\sqrt[m]{t}}E_\rho(q_{\sqrt[m]{t}})^{-\frac{1}{2}}\right\|_{L_\infty^c(\mathcal{N}_{\rho};E_\rho)}\frac{ds}{s}\nonumber\\
&\leq C\int_{0}^{t} \sum_{j\geq 1}2^{-jN}\|\delta_{s^{-1}}F_{k,t,s}\|_{B^{N}} \left\| \rho_2((I-e^{-sL})^Mf)b_{j,\sqrt[m]{t}}E_\rho(q_{\sqrt[m]{t}})^{-\frac{1}{2}}\right\|_{L_\infty^c(\mathcal{N}_{\rho};E_\rho)}\frac{ds}{s}\nonumber\\
&\leq C\int_{0}^{t} \sum_{j\geq 1}2^{-jN}\Big(\frac{t}{s}\Big)^{-N} \left\| \rho_2((I-e^{-sL})^Mf)b_{j,\sqrt[m]{t}}E_\rho(q_{\sqrt[m]{t}})^{-\frac{1}{2}}\right\|_{L_\infty^c(\mathcal{N}_{\rho};E_\rho)}\frac{ds}{s}\nonumber
.\end{align}
This, together with the strong doubling condition (ANi), implies that
\begin{align*}
{\rm I}
&\leq C\int_{0}^{t}\sum_{j\geq 1}2^{-jN}\Big(\frac{t}{s}\Big)^{-N}\Big(\frac{2^j\sqrt[m]{t}}{\sqrt[m]{s}}\Big)^{\frac{D}{2}}2^{\frac{jn}{2}}\left\| \rho_2((I-e^{-sL})^Mf)q_{\sqrt[m]{s}}E_\rho(q_{\sqrt[m]{s}})^{-\frac{1}{2}}\right\|_{L_\infty^c(\mathcal{N}_{\rho};E_\rho)}\frac{ds}{s}\\
&\leq C\int_{0}^{t}\Big(\frac{s}{t}\Big)^{N-\frac{D}{2m}}\left\| \rho_2((I-e^{-sL})^Mf)q_{\sqrt[m]{s}}E_\rho(q_{\sqrt[m]{s}})^{-\frac{1}{2}}\right\|_{L_\infty^c(\mathcal{N}_{\rho};E_\rho)}\frac{ds}{s}\\
&\leq C\sup\limits_{s>0}\left\|R_{\sqrt[m]{s}}|(I-e^{-sL})^Mf|^2\right\|_{\mathcal{M}}^{1/2},
\end{align*}
provided we choose $N$ be a sufficiently large constant.

To estimate term ${\rm II}$,  we follow the estimate of \eqref{eeee} and then apply Lemma \ref{PLPL} to get that
\begin{align}
{\rm II}
&\leq \int_{t}^{\infty}\sum_{j\geq 1} \left\| \left(\pi(F_{k,t,s}(L))(\rho_2((I-e^{-sL})^Mf)b_{j,\sqrt[m]{s}})\right)q_{\sqrt[m]{t}}E_\rho(q_{\sqrt[m]{t}})^{-\frac{1}{2}}\right\|_{L_\infty^c(\mathcal{N}_{\rho};E_\rho)}\frac{ds}{s}\nonumber\\
&\leq C\int_{t}^{\infty} \sum_{j\geq 1}2^{-jN}\|\delta_{s^{-1}}F_{k,t,s}\|_{B^{N}} \left\| \rho_2((I-e^{-sL})^Mf)b_{j,\sqrt[m]{s}}E_\rho(q_{\sqrt[m]{s}})^{-\frac{1}{2}}\right\|_{L_\infty^c(\mathcal{N}_{\rho};E_\rho)}\frac{ds}{s}\nonumber\\
&\leq C\int_{t}^{\infty} \sum_{j\geq 1}2^{-jN}\left(\frac{t}{s}\right) \left\| \rho_2((I-e^{-sL})^Mf)b_{j,\sqrt[m]{s}}E_\rho(q_{\sqrt[m]{s}})^{-\frac{1}{2}}\right\|_{L_\infty^c(\mathcal{N}_{\rho};E_\rho)}\frac{ds}{s}\nonumber.
\end{align}
This, together with the strong doubling condition (ANi), implies that
\begin{align*}
{\rm II}
&\leq C\int_{t}^{\infty} \sum_{j\geq 1}2^{-j(N-\frac{n}{2}-\frac{D}{2})}\left(\frac{t}{s}\right) \left\| \rho_2((I-e^{-sL})^Mf)q_{\sqrt[m]{s}}E_\rho(q_{\sqrt[m]{s}})^{-\frac{1}{2}}\right\|_{L_\infty^c(\mathcal{N}_{\rho};E_\rho)}\frac{ds}{s}\\
&\leq C\int_{t}^{\infty} \left(\frac{t}{s}\right) \left\| \rho_2((I-e^{-sL})^Mf)q_{\sqrt[m]{s}}E_\rho(q_{\sqrt[m]{s}})^{-\frac{1}{2}}\right\|_{L_\infty^c(\mathcal{N}_{\rho};E_\rho)}\frac{ds}{s}
\\&\leq C\sup\limits_{s>0}\left\|R_{\sqrt[m]{s}}|(I-e^{-sL})^Mf|^2\right\|_{\mathcal{M}}^{1/2},
\end{align*}
provided we choose $N$ be a sufficiently large constant.

Combining the estimates of terms I and II, we complete the proof of Proposition \ref{emb}.
\end{proof}
\bigskip

\section{Proof of Theorem \ref{main}}\label{333}
\setcounter{equation}{0}
\subsection{$L_\infty\rightarrow BMO$ boundedness}\label{edp}
In this subsection, we will provide a proof for the first statement of Theorem \ref{main}. For future use in the complex interpolation, we will show a stronger argument under the same condition. That is, we will show that there is a constant $C>0$ independent of $t>0$ and $y\in\mathbb{R}$ such that
\begin{align}\label{strongar}
\left\|e^{itL}(I+L)^{-s(1+iy)}f\right\|_{{ BMO}_{\mathcal{S}}(\mathcal{M})}\leq C(1+|y|)^{s+3}(1+|t|)^{n/2}\|f\|_{\mathcal{M}},\ {\rm for}\ s\geq n/2.
\end{align}
Then the first statement of Theorem \ref{main} follows directly by choosing $y=0$.

To begin with, we will apply a simple argument to reduce the proof to the column BMO estimate. To this end, we assume that $e^{itL}(I+L)^{-s(1+iy)}:\mathcal{A}_\mathcal{M}\rightarrow BMO_{\mathcal{S}}^c(\mathcal{M})$ is bounded by $C(1+|y|)^{s+1}(1+|t|)^{n/2}$ with a constant $C>0$ independent of $t\in\mathbb{R}$.
Note that $$(e^{itL}(I+L)^{-s(1+iy)}f)^*=e^{-itL}(I+L)^{-s(1-iy)}f^*,$$
then
 \begin{align*}
\left\|e^{itL}(I+L)^{-s(1+iy)}f\right\|_{{ BMO}_{\mathcal{S}}^r(\mathcal{M})}&=\left\|e^{-itL}(I+L)^{-s(1-iy)}f^*\right\|_{{ BMO}_{\mathcal{S}}^c(\mathcal{M})}\\&\leq C(1+|y|)^{s+1}(1+|t|)^{n/2}\|f^*\|_{\mathcal{M}}\\&\leq C(1+|y|)^{s+1}(1+|t|)^{n/2}\|f\|_{\mathcal{M}}.
\end{align*}

This, together with Lemma \ref{emb}, implies that it suffices to establish the column BMO estimate of $e^{itL}(I+L)^{-s(1+iy)}$ by showing that for a given large positive constant $M$, there is a constant $C>0$ independent of $t$ such that for any $f\in \mathcal{A}_{\mathcal{M}}$,
\begin{align}\label{keygoal}
\|F(L)f\|_{{ BMO}_{\mathcal{Q},M}^c(\mathcal{M})}\leq C(1+|y|)^{s+1}(1+|t|)^{n/2}\|f\|_{\mathcal{M}},
\end{align}
where we denote $F(u)=e^{itu}(1+u)^{-s(1+iy)}$ for simplicity.

To show this, recalling that $\rho_2$ is a  $*$-homomorphism from $\mathcal{M}$ to $\mathcal{N}_\rho$, we apply  equality \eqref{transfer}  and the property of operator-valued weight to get that
\begin{align*}
\left\|R_{\sqrt[m]{\lambda}}|(I-e^{-\lambda L})^MF(L)f|^2\right\|_\mathcal{M}^{1/2}
&=\left\| \left(\pi((I-e^{-\lambda L})^MF(L))(\rho_2(f))\right)q_{\sqrt[m]{\lambda}}E_\rho(q_{\sqrt[m]{\lambda}})^{-\frac{1}{2}}\right\|_{L_\infty^c(\mathcal{N}_{\rho};E_\rho)}\\
&\leq \left\| \left(\pi((I-e^{-\lambda L})^MF(L))(\rho_2(f)q_{4(1+|t|)\sqrt[m]{\lambda}})\right)q_{\sqrt[m]{\lambda}}E_\rho(q_{\sqrt[m]{\lambda}})^{-\frac{1}{2}}\right\|_{L_\infty^c(\mathcal{N}_{\rho};E_\rho)}\\
&+\left\| \left(\pi((I-e^{-\lambda L})^MF(L))(\rho_2(f)(1-q_{4(1+|t|)\sqrt[m]{\lambda}}))\right)q_{\sqrt[m]{\lambda}}E_\rho(q_{\sqrt[m]{\lambda}})^{-\frac{1}{2}}\right\|_{L_\infty^c(\mathcal{N}_{\rho};E_\rho)}\\
&=:{\rm I}+{\rm II}.
\end{align*}

To estimate term I, we apply the inequality \eqref{abcd} and condition (ALii) to get that
\begin{align*}
{\rm I}&=\left\| \left(\pi((I-e^{-\lambda L})^MF(L))(\rho_2(f)q_{4(1+|t|)\sqrt[m]{\lambda}})\right)q_{\sqrt[m]{\lambda}}E_\rho(q_{\sqrt[m]{\lambda}})^{-\frac{1}{2}}\right\|_{L_\infty^c(\mathcal{N}_{\rho};E_\rho)}\nonumber\\
&\leq\left\| \left(\pi((I-e^{-\lambda L})^MF(L))(\rho_2(f)q_{4(1+|t|)\sqrt[m]{\lambda}})\right)E_\rho(q_{\sqrt[m]{\lambda}})^{-\frac{1}{2}}\right\|_{L_\infty^c(\mathcal{N}_{\rho};E_\rho)}\nonumber\\
&=\left\| \pi((I-e^{-\lambda L})^MF(L))(\rho_2(f)q_{4(1+|t|)\sqrt[m]{\lambda}}E_\rho(q_{\sqrt[m]{\lambda}})^{-\frac{1}{2}})\right\|_{L_\infty^c(\mathcal{N}_{\rho};E_\rho)}
.\end{align*}
Next we apply conditions (Oi) and (ANi) to get that
\begin{align*}
{\rm I}
&\leq \left\|\rho_2(f)q_{4(1+|t|)\sqrt[m]{\lambda}}E_\rho(q_{\sqrt[m]{\lambda}})^{-\frac{1}{2}}\right\|_{L_\infty^c(\mathcal{N}_{\rho};E_\rho)}\nonumber\\
&\leq \left\|E_{\rho}(q_{\sqrt[m]{\lambda}})^{-\frac{1}{2}}E_{\rho}(q_{4(1+|t|)\sqrt[m]{\lambda}})E_{\rho}(q_{\sqrt[m]{\lambda}})^{-\frac{1}{2}}\right\|_{\mathcal{M}}^{1/2}\|f\|_{\mathcal{M}}\nonumber\\
&\leq C(1+|t|)^{n/2}\|f\|_{\mathcal{M}}.
\end{align*}

To estimate term II, we let $\phi$ be a non-negative $C_c^\infty$ function on $\mathbb{R}$ such that supp$\phi\subset (1/4,1)$ and
$$\sum_{\ell\in\mathbb{Z}}\phi(2^{-\ell}u)=1,\ \ {\rm for}\ {\rm any}\ u>0.$$
For simplicity we set $\phi_\ell(u)=\phi(2^{-\ell}u)$. Pick $\nu_0\in\mathbb{Z}_{+}$ such that
$$\left\{\begin{array}{ll}4\leq 2^{\nu_0-\ell (m-1)/m}\sqrt[m]{\lambda}<8, &{\rm if}\ 2^{-\ell(m-1)/m}\sqrt[m]{\lambda}<1,\\ \nu_0=2, &{\rm if}\ 2^{-\ell(m-1)/m}\sqrt[m]{\lambda}\geq 1.\end{array}\right.$$
Observe that if $2^{\ell(m-1)/m}< \sqrt[m]{\lambda}$, then one has
$$q_{2^{\nu_0}(1+|t|)\sqrt[m]\lambda}=q_{ 4(1+|t|)\sqrt[m]\lambda}q_{2^{\nu_0}(1+|t|)\sqrt[m]\lambda}.$$
This, together with a similar procedure as estimating \eqref{djdj} and \eqref{eeee} to deal with the convergence of the infinity sum, implies that
\begin{align}\label{vvv3}
{\rm II}
&\leq \sum_{\ell\in\mathbb{Z}}\left\|\left(\pi ((I-e^{-\lambda L})^MF(L)\phi_\ell(L))(\rho_2(f)(1-q_{4(1+|t|)\sqrt[m]{\lambda}}))\right)q_{\sqrt[m]{\lambda}}E_\rho(q_{\sqrt[m]{\lambda}})^{-\frac{1}{2}}\right\|_{L_\infty^c(\mathcal{N}_{\rho};E_\rho)}\\
&\leq \sum_{\ell:2^{\ell(m-1)/m}> \sqrt[m]{\lambda}}\bigg\| \left(\pi((I-e^{-\lambda L})^MF(L)\phi_\ell(L))(\rho_2(f)((1-q_{4(1+|t|)\sqrt[m]{\lambda}})q_{2^{\nu_0}(1+|t|)\sqrt[m]{\lambda}}))\right)\times\nonumber\\
&\hspace{10.3cm}\times q_{\sqrt[m]{\lambda}}E_\rho(q_{\sqrt[m]{\lambda}})^{-\frac{1}{2}}\bigg\|_{L_\infty^c(\mathcal{N}_{\rho};E_\rho)}\nonumber\\
&+ \sum_{\ell\in\mathbb{Z}}\sum_{\nu\geq \nu_0}\left\| \left(\pi((I-e^{-\lambda L})^MF(L)\phi_\ell(L))(\rho_2(f)b_{\nu,(1+|t|)\sqrt[m]{\lambda}})\right)q_{\sqrt[m]{\lambda}}E_\rho(q_{\sqrt[m]{\lambda}})^{-\frac{1}{2}}\right\|_{L_\infty^c(\mathcal{N}_{\rho};E_\rho)}\nonumber\\
&=:{\rm (a)}+{\rm (b)}.\nonumber
\end{align}

For term (a), we observe that the function $F_\ell(u):=(1-e^{-\lambda u})^MF(u)\phi_{\ell}(u)$ satisfies
$$\|F_\ell\|_\infty\leq C\min\{1,(2^{\ell}\lambda)^{M}\}\min\{1,2^{-\ell s}\}\leq C\min\{1,(2^{\ell}\lambda)^{M}\}2^{-\ell n/2}.$$
This, in combination with inequality \eqref{abcd} and the conditions (ALii) and (Oi), yields
 \begin{align}\label{substi}
{\rm (a)}
&\leq \sum_{\ell:2^{\ell(m-1)/m}> \sqrt[m]{\lambda}}\left\|\left(\pi ((I-e^{-\lambda L})^MF(L)\phi_\ell(L))(\rho_2(f)((1-q_{4(1+|t|)\sqrt[m]{\lambda}})q_{2^{\nu_0}(1+|t|)\sqrt[m]\lambda}))\right)E_\rho(q_{\sqrt[m]{\lambda}})^{-\frac{1}{2}}\right\|_{L_\infty^c(\mathcal{N}_{\rho};E_\rho)}\nonumber\\
&= \sum_{\ell:2^{\ell(m-1)/m}> \sqrt[m]{\lambda}}\left\|\pi ((I-e^{-\lambda L})^MF(L)\phi_\ell(L))(\rho_2(f)((1-q_{4(1+|t|)\sqrt[m]{\lambda}})q_{2^{\nu_0}(1+|t|)\sqrt[m]\lambda})E_\rho(q_{\sqrt[m]{\lambda}})^{-\frac{1}{2}})\right\|_{L_\infty^c(\mathcal{N}_{\rho};E_\rho)}\nonumber\\
&\leq C\sum_{\ell:2^{\ell(m-1)/m}> \sqrt[m]{\lambda}}\min\{1,(2^{\ell}\lambda)^{M}\}2^{-\ell n/2}\left\| \rho_2(f)((1-q_{4(1+|t|)\sqrt[m]{\lambda}})q_{2^{\nu_0}(1+|t|)\sqrt[m]\lambda})E_\rho(q_{\sqrt[m]{\lambda}})^{-\frac{1}{2}}\right\|_{L_\infty^c(\mathcal{N}_{\rho};E_\rho)}\nonumber\\
&\leq C\sum_{\ell:2^{\ell(m-1)/m}> \sqrt[m]{\lambda}}\min\{1,(2^{\ell}\lambda)^{M}\}2^{-\ell n/2}\left\|E_\rho(q_{\sqrt[m]{\lambda}})^{-\frac{1}{2}}E_{\rho}((1-q_{4(1+|t|)\sqrt[m]{\lambda}})q_{2^{\nu_0}(1+|t|)\sqrt[m]\lambda})E_\rho(q_{\sqrt[m]{\lambda}})^{-\frac{1}{2}}\right\|_{\mathcal{M}}^{1/2}\|f\|_{\mathcal{M}}.
\end{align}
To continue, by doubling inequality \eqref{ANi} and the definition of $\nu_0$, we see that if $2^{\ell(m-1)/m}> \sqrt[m]{\lambda}$, then
\begin{align*}
&\left\|E_\rho(q_{\sqrt[m]{\lambda}})^{-\frac{1}{2}}E_{\rho}((1-q_{4(1+|t|)\sqrt[m]{\lambda}})q_{2^{\nu_0}(1+|t|)\sqrt[m]\lambda})E_\rho(q_{\sqrt[m]{\lambda}})^{-\frac{1}{2}}\right\|_{\mathcal{M}}^{1/2}\\
&\leq C \left\|E_\rho(q_{\sqrt[m]{\lambda}})^{-\frac{1}{2}}E_{\rho}(q_{2^{\nu_0}(1+|t|)\sqrt[m]\lambda})E_\rho(q_{\sqrt[m]{\lambda}})^{-\frac{1}{2}}\right\|_{\mathcal{M}}^{1/2}\\
&\leq C(1+|t|)^{n/2}2^{\frac{\nu_0 n}{2}}\\
&\leq C(1+|t|)^{n/2}(2^{\ell(m-1)/m}(\sqrt[m]{\lambda})^{-1})^{n/2}.
\end{align*}
Substituting the above inequality into \eqref{substi}, we conclude that
 \begin{align*}
{\rm (a)}
&\leq C(1+|t|)^{n/2}\sum_{\ell\in\mathbb{Z}}\min\{1,(2^{\ell}\lambda)^{M}\}(2^{\ell}\lambda)^{-\frac{n}{2m}}\|f\|_{\mathcal{M}}.
\\&\leq C(1+|t|)^{n/2}\|f\|_{\mathcal{M}}.
\end{align*}

For term (b) we applying Lemma \ref{PLPL} to conclude that
\begin{align}\label{besti}
{\rm (b)}&\leq \sum_{\ell\in\mathbb{Z}}\sum_{\nu\geq \nu_0}(2^{\ell/m}2^\nu(1+|t|)\sqrt[m]{\lambda})^{-N}\|\delta_{2^\ell}(F_\ell)\|_{B^N}\left\| \rho_2(f)b_{\nu,(1+|t|)\sqrt[m]{\lambda}}E_\rho(q_{\sqrt[m]{\lambda}})^{-\frac{1}{2}}\right\|_{L_\infty^c(\mathcal{N}_{\rho};E_\rho)}.
\end{align}
To continue, we claim that
\begin{align}\label{Bnorm}
\|\delta_{2^\ell}(F_\ell)\|_{B^N}\leq C\max\{1,2^{(N-n/2)\ell}\}(1+|y|)^{s+1}(1+|t|)^N\min\{1,(2^{\ell}\lambda)^{M}\}.
\end{align}
To show this, we let $\psi\in C_c^\infty(\mathbb{R})$ supported in $\psi\in[1/8,2]$ and $\psi(u)=1$ for $u\in[1/4,1]$. We have that for any $\ell\in\mathbb{Z}$,
\begin{align*}
\|\delta_{2^\ell}(F_\ell)\|_{B^N}
&=\|\phi(\cdot)(1-e^{2^{\ell}\lambda\cdot})^MF(2^{\ell}\cdot)\|_{B^N}\\
&\leq \|\psi(\cdot)(1-e^{2^{\ell}\lambda\cdot})^M(1+2^\ell \cdot)^{-isy}\|_{B^N}\|\phi(\cdot)G(2^{\ell}\cdot)\|_{B^N}\\
&\leq \|\psi(\cdot)(1-e^{2^{\ell}\lambda\cdot})^M(1+2^\ell \cdot)^{-isy}\|_{C^{N+2}}\|\phi(\cdot)G(2^{\ell}\cdot)\|_{B^N}\\
&\leq C(1+|y|)^{N+2}\min\{1,(2^{\ell}\lambda)^{M}\}\|\phi(\cdot)G(2^{\ell}\cdot)\|_{B^N},
\end{align*}
where we denote $G(u):=(1+u)^{-s}e^{itu}$.

Let us estimate $\|\phi(\cdot)G(2^{\ell}\cdot)\|_{B^N}$. It follows from the Fourier transform $\mathcal{F}(\phi(\cdot) G(2^\ell\cdot))$ of $\phi(\cdot)G(2^\ell\cdot)$ that
\begin{align*}
\mathcal{F}(\phi(\cdot) G(2^\ell\cdot))(\tau)=\int_{\mathbb{R}}\phi(u)\frac{e^{i(2^{\ell}t-\tau)u}}{(1+2^{\ell}u)^{s}}du.
\end{align*}
Using integration by parts, we get that for any $a\in\mathbb{N}$,
\begin{align*}
|\mathcal{F}(\phi(\cdot) G(2^\ell\cdot))(\tau)|&\leq C_N\min\{1,2^{-\ell s}\}(1+|2^\ell t-\tau|)^{-a}\\
&\leq C_N\min\{1,2^{-\ell n/2}\}(1+|2^\ell t-\tau|)^{-a},\end{align*}
which means that
\begin{align*}
\|\phi(u)G(2^\ell u)\|_{B^N}
&\leq C\min\{1,2^{-\ell n/2}\}\int_{\mathbb{R}}(1+|2^{\ell}t-\tau|)^{-s}(1+|2^{\ell}t-\tau|)^Nd\tau\\
&\leq C\max\{1,2^{(N-n/2)\ell}\}(1+|t|)^N.
\end{align*}
Combining these estimates together, we end the proof of \eqref{Bnorm}.

Now, choosing $N=s+\frac{1}{2}$  such that $N>n/2$, we substitute \eqref{Bnorm} into \eqref{besti} and then apply the doubling inequality \eqref{ANi} to see that
\begin{align}
{\rm (b)}
&\leq C(1+|y|)^{s+3}\sum_{\ell\in\mathbb{Z}}\sum_{\nu\geq \nu_0}(2^{\ell/m}2^\nu\sqrt[m]{\lambda})^{-N}\max\{1,2^{(N-s)\ell}\}\min\{1,(2^{\ell}\lambda)^{M}\}\nonumber\\
&\hspace{6.0cm}\times\left\| \rho_2(f)b_{\nu,(1+|t|)\sqrt[m]{\lambda}}E_\rho(q_{\sqrt[m]{\lambda}})^{-\frac{1}{2}}\right\|_{L_\infty^c(\mathcal{N}_{\rho};E_\rho)}\nonumber\\
&\leq C(1+|y|)^{s+3}\sum_{\ell\in\mathbb{Z}}\sum_{\nu\geq \nu_0}(2^{\ell/m}2^\nu\sqrt[m]{\lambda})^{-N}\max\{1,2^{(N-s)\ell}\}\min\{1,(2^{\ell}\lambda)^{M}\}\nonumber\\
&\hspace{6.0cm}\times\left\| E_\rho(q_{\sqrt[m]{\lambda}})^{-\frac{1}{2}} E_\rho(q_{2^\nu(1+|t|)\sqrt[m]{\lambda}})E_\rho(q_{\sqrt[m]{\lambda}})^{-\frac{1}{2}}\right\|_{\mathcal{M}}^{1/2}\|f\|_\mathcal{M}\nonumber\\
&\leq C(1+|y|)^{s+3}(1+|t|)^{n/2}\sum_{\ell\in\mathbb{Z}}\sum_{\nu\geq \nu_0}2^{-\nu(N-n/2)}(2^{\ell/m}\sqrt[m]{\lambda})^{-N}\max\{1,2^{(N-s)\ell}\}\min\{1,(2^{\ell}\lambda)^{M}\}\|f\|_\mathcal{M}.
\end{align}
From the definition of $\nu_0\geq 0$, we deduce that
\begin{align}
{\rm (b)}
&\leq C(1+|y|)^{s+3}(1+|t|)^{n/2}\sum_{\ell>0}(2^{-\ell(m-1)/m}\sqrt[m]{\lambda})^{N-n/2}(2^{\ell/m}\sqrt[m]{\lambda})^{-N}2^{(N-s)\ell}\min\{1,(2^{\ell}\lambda)^M\}\|f\|_\mathcal{M}\nonumber\\
&+C(1+|y|)^{s+2}(1+|t|)^{n/2}\sum_{\ell\leq0}(2^{\ell}\lambda)^{-\frac{N}{m}}\min\{1,(2^{\ell}\lambda)^M\}\|f\|_\mathcal{M}\nonumber\\
&\leq C(1+|y|)^{s+3}(1+|t|)^{n/2}\sum_{\ell>0}(2^{\ell}\lambda)^{-\frac{n}{2m}}\min\{1,(2^{\ell}\lambda)^M\}\|f\|_\mathcal{M}\nonumber\\
&+C(1+|y|)^{s+3}(1+|t|)^{n/2}\sum_{\ell\leq0}(2^{\ell}\lambda)^{-\frac{N}{m}}\min\{1,(2^{\ell}\lambda)^M\}\|f\|_\mathcal{M}\nonumber\\
&\leq C(1+|y|)^{s+3}(1+|t|)^{n/2}\|f\|_{\mathcal{M}}.
\end{align}
Combining the estimates for (a) and (b), we deduce that
$${\rm II}\leq C(1+|y|)^{s+3}(1+|t|)^{n/2}\|f\|_{\mathcal{M}}.$$
Therefore, the proof of the first statement of Theorem \ref{main} is completed.
\subsection{$L_p$ boundedness}\label{lpbod}
Before providing the proof for the remaining part of our main theorem, we first recall the complex interpolation from \cite[Section 4]{Bergh}. An interpolation pair $(X_0,X_1)$ are two Banach spaces that can be continuously embedded into a topological vector space $X$. Consider the sum space $X_0+X_1$ and intersection space $X_0\cap X_1$, which are Banach spaces equipped with the norm
\begin{align*}
   &\|x\|_{X_0+X_1}:=\inf\{\|x_0\|_{X_0}+\|x_1\|_{X_1}:x=x_0+x_1, x_0\in X_1, x_1\in X_1\} \\
&\|x\|_{X_0\cap X_1}:=\max\{\|x\|_{X_0},\|x\|_{X_1}\}.
\end{align*}
Let $S=\{z\in\mathbb{C}:0<\Re z<1\}$ denote the open strip in the complex plane and $\bar{S}$ be its closure $\bar{S}=\{z\in\mathbb{C}:0\leq\Re z\leq1\}.$ Let $\mathcal{F}(X_0,X_1)$ denote the family of all functions $f$ with values in $X_0+X_1$ and satisfying:
\begin{enumerate}
  \item $f$ is bounded and continuous on $\bar{S}$ and analytic on $S$;
  \item $f(it)\in X_0$ and $f(1+it)\in X_1$ for all $t\in\mathbb{R}$;
  \item $\|f(it)\|_{X_0},\|f(1+it)\|_{X_1}$ tend to $0$ as $|t|\rightarrow\infty$.
\end{enumerate}
Then $\mathcal{F}(X_0,X_1)$ is a Banach space equipped with  the norm given by
\begin{equation*}
  \|f\|_{\mathcal{F}(X_0,X_1)}:=\max\Big\{\sup_{t\in\mathbb{R}}\|f(it)\|_{X_0}, \sup_{t\in\mathbb{R}}\|f(1+it)\|_{X_1}\Big\}.
\end{equation*}
For $0\leq\theta\leq1$, we define a subspace $[X_0,X_1]_\theta$ of $X_0+X_1$, which consists of all $x\in X_0+X_1$ such that $f(\theta)=x$ for some $f\in \mathcal{F}(X_0,X_1)$, and $[X_0,X_1]_\theta$ is a Banach space with the norm
\begin{equation*}
  \|x\|_{[X_0,X_1]_\theta}=\inf\{\|f\|_\mathcal{F}:f(\theta)=x, f\in \mathcal{F}(X_0,X_1)\}.
\end{equation*}
Let $\mathcal{D}(X_0,X_1)$ be a subspace of $ \mathcal{F}(X_0,X_1)$, which consists of the functions of the form
$f(z)=\sum_{\ell=1}^m g_\ell(z)x_\ell$ with each $x_\ell\in X_0\cap X_1$ and with each $g_\ell(z)$ being a scalar function such that it is  continuous on $\bar{S}$, analytic on $S$ and vanishes at infinity.
We will use the following lemma which was established in \cite{MR253054}:
\begin{lemma}\label{anlem}
 For any $0< \theta<1$, $X_0\cap X_1$ is dense in $[X_0,X_1]_\theta$. Besides, for any $x\in X_0\cap X_1$,
\begin{equation*}
  \|x\|_{[X_0,X_1]_\theta}=\inf\{\|f\|_{\mathcal{F}(X_0,X_1)}:f(\theta)=x, f\in \mathcal{D}(X_0,X_1)\}.
\end{equation*}
\end{lemma}

Throughout the sequel in this subsection, we assume that the semigroup $(S_t)_{t\geq0}$ is a  Markov semigroup admitting a Markov dilation (see \cite{MR2885593,MR4178915} for the specific definition). We shall see that these assumptions are only used to guarantee the $L_p$--interpolation theory available (see Lemma \ref{interpolation}), but it is still a challenging problem to minimize the assumptions on $(S_t)_{t\geq0}$ such that the $L_p$--interpolation theory works.

Note that for $1< p\leq\infty$,
\begin{equation*}
  \ker (A_p):=\{f\in dom_p(A), Af=0\}=\{f\in L_p(\mathcal{M}), S_tf=f, \forall t>0 \}.
\end{equation*}
Besides, under the Markovian assumption of $(S_t)_{t\geq0}$, it is easy to show that $(S_t)_{t\geq0}$ is trace preserving and that the null space of ${ BMO}_{\mathcal{S}}^{\dagger}(\mathcal{M})$ is the fixed point space $\ker A_\infty$. Now let $L_p^0(\mathcal{M})$ be the complemented subspace of $L_p(\mathcal{M})$ given by
\begin{equation*}
  L_p^0(\mathcal{M}):=L_p(\mathcal{M})/\ker (A_p)=\{f\in L_p(\mathcal{M}), \lim_{t\rightarrow\infty}S_tf=0 \}.
\end{equation*}
Here the limit is taken with respect to the $L_p(\mathcal{M})$--norm for $1<p<\infty$ and is taken with respect to the weak-$*$ topology for $p=\infty$. Recall from \cite{MR2885593,MR2276775} that there is a canonical decomposition on $L_p(\mathcal{M})$ for $1<p<\infty$:
\begin{equation}\label{decom}
  L_p(\mathcal{M})= \ker (A_p)\oplus L_p^0(\mathcal{M}).
\end{equation}
For any $1<p<\infty$, let $P_p$ be the contractive positive projection from $L_p(\mathcal{M})$ onto $\ker (A_p)$. By mean ergodic theorem, for any $x\in L_p(\mathcal{M})$, $P_p(x)$ can be written as the $L_p$--limit of ergodic mean given by
$$M_s(x):=\frac{1}{s}\int_0^sS_t(x)dt.$$
Note that $P_p$ and $P_q$ coincide on $\ker (A_p)\cap \ker (A_q)$ for two disparate $p$, $q$, which allows us to denote the $P_p$'s by $P$.

The following interpolation result was established in \cite{MR2885593}.
\begin{lemma} \label{interpolation}
 Assume that $(S_t)_{t\geq0}$ is a Markov semigroup admitting a Markov dilation,  for any $2< p<\infty$, we have
 \begin{equation*}
   [ L_2^0(\mathcal{M}),{ BMO}_{\mathcal{S}}(\mathcal{M})]_{1-2/p}= L_p^0(\mathcal{M}).
 \end{equation*}
\end{lemma}
Now we provide a proof for the second statement of our main theorem.

\begin{proof}[The proof of $L_p$ boundedness:]
By duality and density argument, it suffices to show that when $2<p<\infty$, we have
 \begin{align*}
\left\|e^{itL}(I+L)^{-s(1-2/p)}f\right\|_{L_{p}(\mathcal{M})}\leq C(1+|t|)^{\sigma_p}\|f\|_{L_p(\mathcal{M})},\ {\rm for}\ s\geq \frac{n}{2}
\end{align*}
and all $f\in\mathcal{A}_{\mathcal{M}}$.
Note that if $x\in \ker{A_p}$, we have $e^{itL}(I+L)^{-s(1-2/p)}x=x$. Let $P$ be the projection from $L_p(\mathcal{M})$ to $\ker{A_p}$, then for any $f\in \mathcal{A}_{\mathcal{M}}$, we have
$$P(e^{itL}(I+L)^{-s(1-2/p)}f)=e^{itL}(I+L)^{-s(1-2/p)}P(f)=P(f).$$ Hence,
\begin{align*}
  \left\|e^{itL}(I+L)^{-s(1-2/p)}f\right\|_{L_{p}(\mathcal{M})} &\leq\left\|P(e^{itL}(I+L)^{-s(1-2/p)}f)\right\|_{L_{p}(\mathcal{M})}+\left\|(I-P)e^{itL}(I+L)^{-s(1-2/p)}(f)\right\|_{L_{p}(\mathcal{M})} \\
 &=\|P(f)\|_{L_{p}(\mathcal{M})}+\left\|(I-P)e^{itL}(I+L)^{-s(1-2/p)}(f)\right\|_{L_{p}^0(\mathcal{M})} \\
 &\leq \|f\|_{L_{p}(\mathcal{M})}+\left\|(I-P)e^{itL}(I+L)^{-s(1-2/p)}(f)\right\|_{L_{p}^0(\mathcal{M})}.
\end{align*}

Then it suffices to show that
\begin{equation}\label{L_p}
  \left\|(I-P)e^{itL}(I+L)^{-s(1-2/p)}(f)\right\|_{L_{p}^0(\mathcal{M})}
  \leq C(1+|t|)^{\sigma_p}\|f\|_{L_p(\mathcal{M})} ,\ {\rm for}\ s\geq \frac{n}{2}.
\end{equation}
We will use a complex interpolation argument to show the estimate \eqref{L_p}. To this end, suppose that $f\in\mathcal{A}_{\mathcal{M}}$ and $\|f\|_{L_p(\mathcal{M})}=1$,
then by Lemma \ref{anlem}, for any $\epsilon>0$,
there exists an $F\in \mathcal{D}(L_2(\mathcal{M}), \overline{\mathcal{A}_\mathcal{M}})$, such that
$F(1-2/p)=f$ and
\begin{equation*}
  \|F\|_{\mathcal{F}(L_2(\mathcal{M}), \overline{\mathcal{A}_\mathcal{M}})}\leq \|f\|_{[L_2(\mathcal{M}), \overline{\mathcal{A}_\mathcal{M}}]_{1-2/p}}+\epsilon\leq \|f\|_{L_p(\mathcal{M})}+\epsilon=1+\epsilon,
\end{equation*}
where we denote the closure of $\mathcal{A}_\mathcal{M}$ under $L_\infty$ norm by
$\overline{\mathcal{A}_\mathcal{M}}$, and where in the second inequality we used the embedding
$$L_p(\mathcal{M})\cap\mathcal{A}_\mathcal{M}\subset [L_2(\mathcal{M}), \overline{\mathcal{A}_\mathcal{M}}]_{1-2/p}.$$
To continue, we define
 \begin{equation*}
   T(z):=e^{z^2}(1+|t|)^{-nz/2}(I-P)e^{itL}(I+L)^{-sz}(F(z)).
   \end{equation*}

Recall that for any $x\in L_\infty(\mathcal{M})$ with $Px\in \ker A_\infty$, one has
$\|Px\|_{{ BMO}_{\mathcal{S}}(\mathcal{M})}=0$. Therefore, by inequality \eqref{strongar}, when $z=1+iy, y\in\mathbb{R}$,  we have
\begin{align*}\label{BMO}
 \left\|T(1+iy)\right\|_{{ BMO}_{\mathcal{S}}(\mathcal{M})}&\leq e^{1-y^2}(1+|t|)^{-n/2}\left\|e^{itL}(I+L)^{-s(1+iy)}(F(1+iy))\right\|_{{ BMO}_{\mathcal{S}}(\mathcal{M})}\\
 &\leq C\|F(1+iy)\|_{L_\infty(\mathcal{M})}.
\end{align*}
Besides, when $z=iy, y\in\mathbb{R}$, we have
\begin{align*}
 \|T(iy)\|_{L_2^0(\mathcal{M})}&=\left\|e^{-y^2}(1+|t|)^{-iny/2}(I-P)e^{itL}(I+L)^{-iys}(F(iy))\right\|_{L_2^0(\mathcal{M})} \\ &\leq C \left\|(I-P)e^{itL}(I+L)^{-iys}(F(iy))\right\|_{L_2^0(\mathcal{M})}\\&\leq C \left\|e^{itL}(I+L)^{-iys}(F(iy))\right\|_{L_2(\mathcal{M})}\\
&\leq C\|F(iy)\|_{L_2(\mathcal{M})}.
\end{align*}
 Hence, $T\in\mathcal{F}(L_2^0(\mathcal{M}),{{ BMO}_{\mathcal{S}}(\mathcal{M})})$ and
 \begin{equation}\label{inter}
   \|T\|_{\mathcal{F}(L_2^0(\mathcal{M}),{{ BMO}_{\mathcal{S}}(\mathcal{M})})}\leq C\|F\|_{\mathcal{F}(L_2(\mathcal{M}), \overline{\mathcal{A}_\mathcal{M}})}\leq C(1+\varepsilon).
 \end{equation}
 Note that
 \begin{equation*}
(I-P)e^{itL}(I+L)^{-s(1-2/p)}(F(1-2/p))=e^{-(1-2/p)^2}(1+|t|)^{\sigma_p}T(1-2/p).
\end{equation*}
This, in combination with Lemma \ref{interpolation}, implies
 \begin{align*}
   \left\|(I-P)e^{itL}(I+L)^{-s(1-2/p)}(f)\right\|_{L_{p}^0(\mathcal{M})}&= \left\|(I-P)e^{itL}(I+L)^{-s(1-2/p)}(f)\right\|_{[ L_2^0(\mathcal{M}),{ BMO}_{\mathcal{S}}(\mathcal{M})]_{1-2/p}}\\
   &\leq e^{-(1-2/p)^2}(1+|t|)^{\sigma_p}\|T\|_{\mathcal{F}(L_2^0(\mathcal{M}),{BMO}_{\mathcal{S}}(\mathcal{M}))}\\&\leq C(1+|t|)^{\sigma_p}(1+\epsilon).
 \end{align*}
Letting $\epsilon\rightarrow 0$, we get \eqref{L_p}. Therefore, the proof of the second statement of our main theorem is completed.
\end{proof}
\bigskip

\section{Applications}\label{Appi}
\setcounter{equation}{0}

\subsection{Operator-valued setting}\label{Operator-valued setting}
In this subsection we apply our algebraic approach to investigate Schr\"{o}dinger groups in the operator-valued setting, which opens a door to study Schr\"{o}dinger groups over quantum Euclidean space, matrix algebra and group von Neumann algebra.

Let $(X,d,\mu)$ be a doubling metric space such that the doubling condition \eqref{doubling} holds. It follows that there exist $C>0$ and $0\leq D\leq n$ such that
\begin{align}\label{dis}
\mu(B(y,r))\leq C\left(1+\frac{d(x,y)}{r}\right)^{D}\mu(B(x,r))
\end{align}
independent of $x,y\in X$ and $r>0$. Indeed, equality \eqref{dis} holds trivially in Euclidean setting with $D=0$ and holds directly in the abstract setting with $D=n$ due to the triangle inequality for the metric $d$ in combination with the doubling condition \eqref{doubling}.

Now suppose that $L$ is a non-negative self-adjoint operator on $L_2(X)$ and its associated heat kernel satisfies ${\rm (GE_m)}$. The estimate \eqref{GE} holds for many (sub-)elliptic differential operators of order $m$ (see for example \cite{MR4150932, MR1103113,MR1715407,MR2569498,MR2124040,MR1218884} and the references therein).

Let $\mathcal{M}$ be a semifinite von Neumann algebra equipped with a normal semifinite faithful trace $\tau$. Then $\hat{\mathcal{S}}=(e^{-tL}\otimes id_\mathcal{M})$   is called a semicommutative heat semigroup. Denote by $\hat{L}=L\otimes id_{\mathcal{M}}$ the generator of $\hat{S}$. Next, consider the algebra of essentially  bounded operator-valued functions $f:X\rightarrow \mathcal{M}$ equipped with the trace
$${\rm Trace}(f)=\int_X \tau(f(x))d\mu(x).$$
Observe that its weak-$*$ closure $\mathcal{N}=L_\infty(X)\botimes \mathcal{M}$ is a von Neumann algebra. Now we let $\mathcal{Q}=\{(R_{j,t},\sigma_{j,t}):j\in\mathbb{N},t>0\}$ be a P-metric associated to the original bounded semigroup $\mathcal{S}=\{e^{-tL}\}_{t\geq 0}$ on $L_\infty(X)$, which is constructed explicitly in Section \ref{Hilbert modules}. We assume in this subsection that $p_t(x,y)\geq 0$ for any $x,y\in X$ to ensure the positivity assumption of the semigroup $\mathcal{S}$, then $\mathcal{Q}$ satisfies an operator-valued generalization of semigroup majorization and average domination condition. Therefore, $\mathcal{Q}$ extends to a P-metric $\mathcal{Q}_\mathcal{N}$ in $\mathcal{N}$ by tensorizing with $id_\mathcal{M}$ and ${\bf 1}_\mathcal{M}$, which can be expressed explicitly as
$$\mathcal{Q}_\mathcal{N}=\{(R_{j,t}\otimes id_{\mathcal{M}},\sigma_{j,t}\otimes {\bf 1}_{\mathcal{M}}):j\in\mathbb{N},t>0\}.$$

Before providing our application, we first establish two auxiliary lemmas. For simplicity, we set $f_x(y):=f(x,y)$ for any $x,y\in X$.
\begin{lemma}\label{keytran}
Let $T$ be a linear bounded operator on $L_2(X)$, then for any $f\in L_2^c(X)\botimes \mathcal{N}$,
\begin{align}\label{prin}
\left\|\left(\int_X|(T\otimes id_{\mathcal{M}})(f_x)(y)|^2d\mu(y)\right)^{1/2}\right\|_{\mathcal{N}}\leq \|T\|_{L_2(X)\rightarrow L_2(X)}\left\|\left(\int_X|f(x,y)|^2d\mu(y)\right)^{1/2}\right\|_{\mathcal{N}}.
\end{align}
\end{lemma}
\begin{proof}
The lemma may be well-known to experts, but we are unable to locate a precise reference. For completeness, we provide a proof here.

To begin with, using the embedding $\mathcal{N}\subset B(\ell_2(\mathcal{N}))$, we obtain
\begin{align}\label{jjjjk}
&\left\|\left(\int_X|(T\otimes id_{\mathcal{M}})(f_x)(y)|^2d\mu(y)\right)^{1/2}\right\|_\mathcal{N}\nonumber\\
&=\sup\limits_{\|h\|_{L_2(\mathcal{N})}\leq 1}\left(\int_X\int_X\tau\left( |((T\otimes id_{\mathcal{M}})(f_x)(y))h(x)|^2\right) d\mu(y)d\mu(x)\right)^{\frac{1}{2}}\nonumber\\
&=\sup\limits_{\|h\|_{L_2(\mathcal{N})}\leq 1}\left(\int_X\int_X\tau\left( |(T\otimes id_{\mathcal{M}})(f_xh(x))(y))|^2\right) d\mu(y)d\mu(x)\right)^{\frac{1}{2}}\nonumber\\
&\leq \|T\otimes id_{\mathcal{M}}\|_{L_2(\mathcal{N})\rightarrow L_2(\mathcal{N})}\sup\limits_{\|h\|_{L_2(\mathcal{N})}\leq 1}\left(\int_X\int_X\tau\left( |f(x,y)h(x)|^2\right) d\mu(y)d\mu(x)\right)^{\frac{1}{2}}.
\end{align}
To continue, recall the fact that every bounded linear operator $T$ on a scalar-valued $L^2$-space can be extended to a bounded operator on an Hilbert space valued $L^2$-space, i.e.
\begin{align}\label{remm}
\|T\otimes id_{\mathcal{M}}\|_{L_2(\mathcal{N})\rightarrow L_2(\mathcal{N})}\leq \|T\|_{L_2(X)\rightarrow L_2(X)}.
\end{align}
Combining the inequalities \eqref{jjjjk} and \eqref{remm}, we finish the proof of Lemma \ref{keytran}.
%
%
\end{proof}
\begin{lemma}\label{apprx}
Let $T\in B(L_2(X))$ and $\{T_k\}_k\subset B(L_2(X))$ such that $\sup\limits_k \|T_k\|_{B(L_2(X))}<+\infty$ and $T_k\rightarrow T$ in the strong operator topology of $B(L_2(X))$. Then for any  $f\in L_2^c(X)\botimes \mathcal{N}$ and $x\in X$,
\begin{enumerate}
  \item $\int_{X}|(T\otimes id_{\mathcal{M}})(f_x\chi_{B(x,s)^c})(y)|^2d\mu(y)\rightarrow 0,\ {\rm as}\ s\rightarrow \infty$,
      \smallskip

  \item $\int_{X}|((T_k-T)\otimes id_\mathcal{M})(f_x)(y)|^2d\mu(y)\rightarrow 0,\ {\rm as}\ k\rightarrow \infty,$
\end{enumerate}
in the strong operator topology of $B(L_2(\mathcal{M}))$.
\end{lemma}
\begin{proof}
To show (1), from the definition of strong operator topology of $B(L_2(\mathcal{M}))$, it suffices to verify that for any $b\in L_2(\mathcal{M})$,
\begin{align*}
\lim_{s\rightarrow\infty}\tau\left(\int_{X}|(T\otimes id_{\mathcal{M}})(f_x\chi_{B(x,s)^c}b)(y)|^2d\mu(y)\right)=0.
\end{align*}
However, this can be verified by domination convergence theorem together with inequality \eqref{remm}.

To show (2), we first note that the conditions imply that $T_k\otimes id_\mathcal{M}$ converges to $T\otimes id_\mathcal{M}$ in the strong operator topology of $B(L_2(\mathcal{N}))$. Therefore, for any $b\in L_2(\mathcal{M})$,
\begin{align*}
\lim_{k\rightarrow \infty}\tau\left(\int_X |((T_k-T)\otimes id_\mathcal{M})(f_xb)(y)|^2d\mu(y)\right)=0.
\end{align*}
This shows (2) and finishes the proof of Lemma \ref{apprx}.
\end{proof}
Then our main theorem applied to the operator-valued setting can be formulated as follows.
%
\begin{theorem}\label{operatorthm}
Assume that $(X, d, \mu)$ is  an  $n$-dimensional doubling metric space.  Suppose that $L$ is a non-negative self-adjoint operator on $L_2(X)$ and its associated heat kernel $p_t(x,y)$
satisfies the Gaussian upper bound \eqref{GE} and $p_t(x,y)\geq 0$ for any $x,y\in X$.
Then  there exists a  constant $C=C(n, m)>0$ independent of $t$ such that for any $f\in L_\infty(\mathcal{N})\cap \big(L_2^c(X)\bar{\otimes}L_\infty(\mathcal{M})\big),$
\begin{eqnarray} \label{e1.6dd}
  \left\| (id_{\mathcal{N}}+\hat{L})^{-s }e^{it\hat{L}} f\right\|_{BMO_{\hat{\mathcal{S}}}(\mathcal{N})} \leq C (1+|t|)^{n/2} \|f\|_{\mathcal{N}}, \ {\rm for}\ s\geq \frac{n}{2}.
\end{eqnarray}
Furthermore, if in additionally $\mathcal{S}=\{e^{-tL}\}$ admits a Markov dilation, then for any $1<p<\infty$, there exists a  constant $C=C(n,m,p)>0$ independent of $t$ such that
\begin{eqnarray} \label{e1.555dd}
 \left\| (id_{\mathcal{N}}+\hat{L})^{-s } e^{it\hat{L}}f \right\|_{L_p(\mathcal{N})} \leq C (1+|t|)^{\sigma_p} \|f\|_{L_p(\mathcal{N})}, \  {\rm for}\ s\geq \sigma_p=n\Big|{1\over  2}-{1\over  p}\Big|.
\end{eqnarray}
\end{theorem}
\begin{proof}
We shall see that all the conditions imposed on Theorem \ref{main} hold for the special case $X=\mathbb{R}^n$ and $L=-\Delta$. Then making some minor modifications with higher-level viewpoints toward parts of these conditions, we could pass the result for this special model to a more general one as we formulate in Theorem \ref{operatorthm}.


To begin with, we note that the assumption $p_t(x,y)\geq0$ for any $x,y\in X$ implies that $e^{-t\hat{L}}$ is a positive map.
Next,
consider two maps $\omega_1,\omega_2$ defined by $\omega_1(f)(x,y)=f(x)$ and $\omega_2(f)(x,y)=f(y)$, for any $x,y\in X$ and any Borel measurable function $f$. Define $\rho_j=\omega_j\otimes id_{\mathcal{M}}$, $j=1,2$. Then $\rho_1,\rho_2: \mathcal{N}\rightarrow\mathcal{N}_\rho=L_\infty(X\times X)\botimes \mathcal{M}$  are injective $*$-homomorphisms. Besides, $\rho_1,\rho_2$ are maps from $L_2^c(X)\bar{\otimes}L_\infty(\mathcal{M})$ to $L_2^c(X)\bar{\otimes}L_\infty(\mathcal{N})$.

Now for any bounded Borel measurable function $F$ on $\mathbb{R}$, define amplified spectral multiplier  $\pi(F(\hat{L}))$ by $\pi(F(\hat{L})):=id_{L_\infty(X)}\otimes F(\hat{L})$,  then it is direct to see that $\pi(F(\hat{L})):L_2^c(X)\botimes L_\infty(\mathcal{N})\rightarrow L_2^c(X)\botimes L_\infty(\mathcal{N})$ satisfies $\pi(F(\hat{L}))\circ\rho_2=\rho_2\circ F(\hat{L})$ on $\mathcal{A}_\mathcal{N}=L_\infty(\mathcal{N})\cap\big( L_2^c(X)\bar{\otimes}L_\infty(\mathcal{M})\big)$. 
 Next, choose  $E_{\rho}=id_{L_\infty(X)}\otimes\int_X\otimes id_\mathcal{M}$ and $q_r$  be the projection $\chi_{B(x,r)}\otimes {\bf 1}_\mathcal{M}$.
The priori assumption trivially holds for the most classical model $L=-\Delta$  and $X=\mathbb{R}^n$ with $\mathcal{A}_\mathcal{N}$ chosen to be a smaller subalgebra $\mathcal{S}(\mathbb{R}^{n})\otimes \mathcal{S}_{\mathcal{M}}\subset L_\infty(\mathcal{N})\cap\big( L_2^c(\mathbb{R}^n)\bar{\otimes}L_\infty(\mathcal{M})\big)$. Here  $\mathcal{S}_{\mathcal{M}}$ is the linear span of all positive elements with finite support in $\mathcal{M}$ (See \cite{PX}). For the general setting formulated in our theorem, we observe that the priori assumption can be relaxed to the following more flexible condition: for any bounded Borel measurable function $F:[0,\infty)\rightarrow \mathbb{C}$, the spectral multiplier $F(\hat{L})$ is bounded on $L_2^c(X)\bar{\otimes}L_\infty(\mathcal{M})$. This condition is a direct consequence of Lemma \ref{keytran} with $f$ being chosen to be independent of $x-$variable. Under this condition, $f\mapsto\rho_2(F(\hat{L})f)$ and $f\mapsto R_{r}(F(\hat{L})f)$ are well-defined operators in $L_2^c(X)\bar{\otimes}L_\infty(\mathcal{M})$ and therefore, the proof of Theorem \ref{main} follows exactly as it was shown there.

The algebraic conditions (ALi) and (ALii) are direct consequences of commutativity. Now we verify the analytic condition (ANi), which can be stated precisely in this setting as
\begin{align}\label{iop}
\left\|\frac{1}{\mu(B(x,r_3))}\int_{B(x,r_1)}|f(y)|^2d\mu(y)\right\|_{\mathcal{N}}
\leq C\left(\frac{r_1}{r_2}\right)^{D}\left(\frac{r_1}{r_3}\right)^{n}\left\|\frac{1}{\mu(B(x,r_2))}\int_{B(x,r_2)}|f(y)|^2d\mu(y)\right\|_{\mathcal{N}}
\end{align}
for any $r_1\geq r_2\geq r_3$ and $f\in \mathcal{N}$. To verify this, we note that $B(x,r_1)$ can be covered by at most $c\big(\frac{r_1}{r_2}\big)^n$ balls with radius $r_2$. We denote these balls by $B(x_1,r_2)$, $\cdots$, $B(x_{d_{n}},r_2)$, where $d_n\lesssim \big(\frac{r_1}{r_2}\big)^n$.
By inequality \eqref{dis}, for any $1\leq \ell\leq d_n$,
\begin{align*}
\mu(B(x_{\ell},r_2))\leq C\mu(B(x,r_2))\left(1+\frac{d(x_\ell,x)}{r_2}\right)^D\leq C\mu(B(x,r_2))\left(\frac{r_1}{r_2}\right)^D.
\end{align*}
Combining these facts together yields inequality \eqref{iop}.

To verify the operator conditions, we first observe that the boundedness condition (Oi) is equivalent to
\begin{align}\label{veroi}
\left\|\left(\int_X|F(\hat{L})f_x(y)|^2d\mu(y)\right)^{1/2}\right\|_\mathcal{N}
&\leq \|F\|_{\infty}\left\|\left(\int_X|f(x,y)|^2d\mu(y)\right)^{1/2}\right\|_\mathcal{N},
\end{align}
but this is an easy consequence of spectral theorem (see for example \cite{MR617913}) together with inequality \eqref{prin}.

The boundedness condition (Oii) is equivalent to for any $t>0$, $k\geq 1$, $0<r_1\leq r_2$ and $f\in L_\infty(X)\bar{\otimes}L_\infty(\mathcal{N})$,
\begin{align} \label{veroii}
&\left\|\left(\fint_{B(x,r_1)}|e^{-t\hat{L}}(\chi_{U_k(B(x,r_2))}f_x)(y)|^2d\mu(y)\right)^{\frac{1}{2}}\right\|_{\mathcal{N}}\nonumber\\
&\leq CE_k\left(-c\left(\frac{r_2}{t^{1/m}}\right)^{\frac{m}{m-1}}\right)\left\|\left(\frac{1}{\mu(B(x,\max\{r_1,t^{1/m}\}))}\int_{U_k(B(x,r_2))}|f(x,y)|^2d\mu(y)\right)^{\frac{1}{2}}\right\|_{\mathcal{N}}
\end{align}
for some constants  $C,c>0$ and $m\geq 2$, where we denote $U_1(B):=2B$ and $U_k(B):=2^{k}B \backslash 2^{k-1}B$ for $k\geq 2$.

{\bf Case 1: }if $r_1\geq t^{1/m}$, then the estimate \eqref{veroii} in the case of $k=1$ is a direct consequence of spectral theorem (see for example \cite{MR617913}). For the case $k\geq 2$, we apply the assumption imposed on $L$ to get that
\begin{align}\label{rhso}
&\left(\fint_{B(x,r_1)}|e^{-t\hat{L}}(\chi_{U_k(B(x,r_2))}f_x)(y)|^2d\mu(y)\right)^{\frac{1}{2}}\nonumber\\
&=  \left(\fint_{B(x,r_1)}\left|\int_{U_k(B(x,r_{2}))}p_t(y,w)f(x,w)d\mu(w)\right|^2d\mu(y)\right)^{\frac{1}{2}}\nonumber\\
&\leq  \left(\fint_{B(x,r_1)}\int_{U_k(B(x,r_{2}))}|p_t(y,w)|^2d\mu(w)d\mu(y)\right)^{\frac{1}{2}}\left(\int_{U_k(B(x,r_{2}))}|f(x,w)|^2d\mu(w)\right)^{\frac{1}{2}}\nonumber\\
&\leq C \left(\fint_{B(x,r_1)}\int_{U_k(B(x,r_{2}))}\left|\frac{1}{\mu(B(y,t^{1/m}))}\exp\left(-c\left(\frac{2^kr_2}{t^{1/m}}\right)^{\frac{m}{m-1}}\right)\right|^2d\mu(w)d\mu(y)\right)^{\frac{1}{2}}\left(\int_{U_k(B(x,r_{2}))}|f(x,w)|^2d\mu(w)\right)^{\frac{1}{2}}.
\end{align}
To continue, by inequality \eqref{dis},
\begin{align}
\frac{1}{\mu(B(y,t^{1/m}))}\leq \frac{C}{\mu(B(x,t^{1/m}))}\bigg(1+\frac{d(x,y)}{t^{1/m}}\bigg)^D\leq \frac{C}{\mu(B(x,t^{1/m}))}\bigg(1+\frac{2^kr_2}{t^{1/m}}\bigg)^D.
\end{align}
Note that the non-doubling factor $\big(1+\frac{2^kr_2}{t^{1/m}}\big)^D$ can be absorbed in the exponential term of the right-hand side of \eqref{rhso}. Therefore, one has
\begin{align}\label{still}
&\left(\fint_{B(x,r_1)}|e^{-t\hat{L}}(\chi_{U_k(B(x,r_2))}f_x)(y)|^2d\mu(y)\right)^{\frac{1}{2}}\nonumber\\
&\leq C\frac{\mu(U_k(B(x,r_2)))^{1/2}}{\mu(B(x,t^{1/m}))}\exp\left(-c\left(\frac{2^kr_2}{t^{1/m}}\right)^{\frac{m}{m-1}}\right)\left(\int_{U_k(B(x,r_2))}|f(x,w)|^2d\mu(w)\right)^{\frac12}\nonumber\\
&\leq C\frac{1}{\mu(B(x,t^{1/m}))^{1/2}}\bigg(1+\frac{2^kr_2}{t^{1/m}}\bigg)^{\frac{n}{2}}\exp\left(-c\left(\frac{2^kr_2}{t^{1/m}}\right)^{\frac{m}{m-1}}\right)\left(\int_{U_k(B(x,r_2))}|f(x,w)|^2d\mu(w)\right)^{\frac12}\nonumber\\
&\leq C\frac{1}{\mu(B(x,t^{1/m}))^{1/2}}\exp\left(-c\left(\frac{2^kr_2}{t^{1/m}}\right)^{\frac{m}{m-1}}\right)\left(\int_{U_k(B(x,r_2))}|f(x,w)|^2d\mu(w)\right)^{\frac12}.
\end{align}
By the doubling condition \eqref{doubling},
\begin{align*}
\frac{1}{\mu(B(x,t^{1/m}))^{1/2}}\leq \frac{C}{\mu(B(x,r_1))^{1/2}}\bigg(1+\frac{r_1}{t^{1/m}}\bigg)^{\frac n2}\leq \frac{C}{\mu(B(x,r_1))^{1/2}}\bigg(1+\frac{2^kr_2}{t^{1/m}}\bigg)^{\frac n2}.
\end{align*}
Noting that the extra term $\big(1+\frac{2^kr_2}{t^{1/m}}\big)^{\frac n2}$ can be absorbed in the exponential in the right-hand side of \eqref{still}, we obtain \eqref{veroii} in this case.

{\bf Case 2: }if $r_1\leq t^{1/m}$, then we note that inequality \eqref{still} still holds when $k\geq 2$ with the same proof. Therefore, it remains to provide a proof for the case $k=1$. To this end, we divide our proof into two subcases.

\ \ \ \ {\bf Subcase 1:} if $2r_2\leq t^{1/m}$, then we apply the assumption imposed on $L$ to get that
\begin{align}\label{rhso22}
&\left(\fint_{B(x,r_1)}|e^{-t\hat{L}}(\chi_{B(x,2r_2)}f_x)(y)|^2d\mu(y)\right)^{\frac{1}{2}}\nonumber\\
&\leq C \left(\fint_{B(x,r_1)}\left(\int_{B(x,2r_{2})}\frac{1}{\mu(B(y,t^{1/m}))^2}d\mu(w)\right)d\mu(y)\right)^{\frac{1}{2}}\left(\int_{B(x,2r_2)}|f(x,w)|^2d\mu(w)\right)^{\frac{1}{2}}.
\end{align}
By inequality \eqref{dis} and using the condition $r_1 \leq t^{1/m}$, we see that
\begin{align}\label{biandis}
\frac{1}{\mu(B(y,t^{1/m}))}\leq \frac{C}{\mu(B(x,t^{1/m}))}\bigg(1+\frac{d(x,y)}{t^{1/m}}\bigg)^D\leq \frac{C}{\mu(B(x,t^{1/m}))}\bigg(1+\frac{r_1}{t^{1/m}}\bigg)^D\leq \frac{C}{\mu(B(x,t^{1/m}))}.
\end{align}
Then by H\"{o}lder's inequality,  the right-hand side of \eqref{rhso22} can be estimated as follows.
\begin{align*}
{\rm RHS}\ {\rm of}\ \eqref{rhso22}&\leq C\frac{\mu(B(x,2r_2))^{1/2}}{\mu(B(x,t^{1/m}))^{1/2}}\left(\frac{1}{\mu(B(x, t^{1/m}))}\int_{B(x,2r_2)}|f(x,w)|^2d\mu(w)\right)^{\frac{1}{2}}\\
&\leq C\left(\frac{1}{\mu(B(x, t^{1/m}))}\int_{B(x,2r_2)}|f(x,w)|^2d\mu(w)\right)^{\frac{1}{2}}.
\end{align*}

\ \ \ \ {\bf Subcase 2:} if $2r_2\geq t^{1/m}$, then we choose $k_0\in\mathbb{Z}_+$ such that
$$2^{k_0-1}t^{1/m}\leq 2r_2< 2^{k_0}t^{1/m}$$
and apply the assumption imposed on $L$ together with inequality \eqref{biandis} to get that
\begin{align*}
&\left(\fint_{B(x,r_1)}|e^{-t\hat{L}}(\chi_{B(x,2r_2)}f_x)(y)|^2d\mu(y)\right)^{\frac{1}{2}}\nonumber\\
&= C \left(\fint_{B(x,r_1)}\left|\int_{B(x,2r_{2})}p_t(y,w)f(x,w)d\mu(w)\right|^2d\mu(y)\right)^{\frac{1}{2}}\nonumber\\
&\leq C \left(\fint_{B(x,r_1)}\int_{B(x,2r_{2})}|p_t(y,w)|^2d\mu(w)d\mu(y)\right)^{\frac{1}{2}}\left(\int_{B(x,2r_{2})}|f(x,w)|^2d\mu(w)\right)^{\frac{1}{2}}\nonumber\\
&\leq C \left(\fint_{B(x,r_1)}\int_{B(x,2r_2)}\left|\frac{1}{\mu(B(y,t^{1/m}))}\exp\left(-c\left(\frac{d(y,w)}{t^{1/m}}\right)^{\frac{m}{m-1}}\right)\right|^2d\mu(w)d\mu(y)\right)^{\frac{1}{2}}\left(\int_{B(x,2r_{2})}|f(x,w)|^2d\mu(w)\right)^{\frac{1}{2}}\nonumber\\
&\leq C\left(\sum_{j=1}^{k_0} \fint_{B(x,r_1)}\int_{U_j(B(x,t^{1/m}))\cap B(x,2r_2)}\left|\frac{E_j(-c)}{\mu(B(x,t^{1/m}))}\right|^2d\mu(w)d\mu(y)\right)^{\frac{1}{2}}\left(\int_{B(x,2r_{2})}|f(x,w)|^2d\mu(w)\right)^{\frac{1}{2}}\\
&\leq C\left(\sum_{j=1}^{k_0} E_j(-c)\frac{\mu(B(x,2^jt^{1/m}))^{1/2}}{\mu(B(x,t^{1/m}))^{1/2}}\right)\left(\frac{1}{\mu(B(x,t^{1/m}))}\int_{B(x,2r_{2})}|f(x,w)|^2d\mu(w)\right)^{\frac 12}
\\&\leq C\left(\frac{1}{\mu(B(x,t^{1/m}))}\int_{ B(x,2r_2)}|f(x,w)|^2d\mu(w)\right)^{\frac 12}.
\end{align*}
Combining the estimates in two cases together, we deduce inequality \eqref{veroii}.

Finally, we turn to the topological conditions (Ti) and (Tii). Note that in the proof of Theorem \ref{main}, the assumption (Ti) is only used to ensure the validity of triangle inequality: for any $F\in\mathcal{S}(\mathbb{R})$ and $s>0$,
\begin{align*}
\left\|\fint_{B(x,r)}|F(\hat{L})f_x(y)|^2d\mu(y)\right\|_\mathcal{N}^{\frac{1}{2}}
\leq\sum_{k\geq 1}\left\|\fint_{B(x,r)}|F(\hat{L})(\chi_{U_k(B(x,s))}f_x)(y)|^2d\mu(y)\right\|_\mathcal{N}^{\frac{1}{2}},
\end{align*}
while the assumption (Tii) is only used to ensure the validity of \eqref{vvv1}, \eqref{djdj} and \eqref{vvv3}.
However, these inequalities can be verified easily by using the tensor separation nature of $\mathcal{N}$, identifying  $\mathcal{M}$ as a von Neumann subalgebra of $B(L_2(\mathcal{M}))$ and then applying Lemma \ref{apprx}.

This ends the proof of Theorem \ref{operatorthm}.
\end{proof}

\begin{remark}\label{rm12}
{\rm Together with the available $L_p$-interpolation theorem at hand, Theorem \ref{operatorthm} can be regarded as an extension of two previous results. Firstly, note that the positivity assumption on the heat kernel is only used to ensure that $BMO_{\mathcal{S}}$ and $BMO_{\mathcal{Q}}$ are (semi-)normed spaces and $\mathcal{Q}_\mathcal{N}$  is a P-metric in $\mathcal{N}$. However, in the case of $\mathcal{M}=\mathbb{C}$, these facts hold trivially without this assumption, so Theorem \ref{operatorthm} goes back to the classical result due to \cite{MR4150932}, which showed the sharp endpoint  $L_p$ boundedness of Schr\"{o}dinger groups under \eqref{GE}.
Secondly, in the case of $X=\mathbb{R}^n$ and $L=-\Delta$, the above theorem goes back to the $n$-dimensional analogous of  \cite[Theorem 6.4]{MR2327840}.}
\end{remark}
\begin{remark}\label{rmkey}{\rm
Inspired by the classical result in \cite{MR4150932}, one may expect that the regularity assumption can be removed from the second statement of Theorem \ref{operatorthm}. To remove this, it requires an interpolation theorem for the operator-valued BMO space associated with non-negative self-adjoint operator satisfying \eqref{GE}, which will be addressed in the forthcoming paper \cite{FHW}. At this stage we just point out that there are still a large number of concrete examples fulfilling all the assumptions imposed on Theorem \ref{operatorthm}. This can be seen easily from the fact that on commutative or semicommutative von Neumann algebras, a Markov semigroup is regular if it is diffusion (see \cite[Remark 1.4]{MR4178915}). In particular, Theorem \ref{operatorthm} holds for the following generators $L$:

(1)  Neumann-Laplacian operator $\Delta_N$ on $\mathbb{R}^n$ (see for example \cite{DDSY});

(2)  Neumann-Laplacian operator $\Delta_{N_\pm}$ on $\mathbb{R}_\pm^n$ (see for example \cite{DDSY});

(3)  Dunkl-Laplacian operator $\Delta_{{\rm Dunkl}}$ on $\mathbb{R}^N$ equipped with Dunkl measure $d\omega$ (see for example \cite{MR3989141});

(4)  Laplace-Beltrami operator $\Delta_g$ on $n$-dimensional complete Riemannian manifold $(M,g)$ with non-negative Ricci curvature (see for example \cite[Section 5]{MR1103113}).

%
%
}
\end{remark}

%

\subsection{Quantum Euclidean space}\label{QES}
In this subsection we apply our algebraic approach to study Schr\"{o}dinger groups in quantum Euclidean space.

At the beginning of this subsection, we recall some concepts and definitions in quantum Euclidean space. There are several alternative definitions of quantum Euclidean space (see \cite{GJMarxiv,MR4320770,MR4178915,MR4156216}). We adopt the one from \cite[Section 1]{MR4320770} (see also \cite[Section 4.3]{MR4178915}). To begin with, let $\Theta$ be an anti-symmetric $\mathbb{R}$-valued $n\times n$ matrix with $n\geq 1$.  $A_\Theta$ is defined as the universal $C^*$-algebra generated by a family $\{u_j(s)\}_{j=1}^n$ of strongly continuous one-parameter unitary groups in $s\in\mathbb{R}$ which satisfies the following $\Theta$-commutation relations
$$u_j(s)u_k(t)=e^{2\pi i\Theta_{jk}st}u_k(t)u_j(s).$$
In the case of $\Theta=0$, by Stone's theorem, one can take $u_j(s)=\exp(2\pi is\langle e_j,\cdot\rangle)$ and $A_\Theta$ be the space consisting of all bounded continuous functions on $\mathbb{R}^n$. In general, for any $\xi=(\xi_1,\cdots,\xi_n)\in\mathbb{R}^n$, we define the unitaries $\lambda_\Theta(\xi)=u_1(\xi_1)u_2(\xi_2)\cdots u_n(\xi_n)$. Next, $E_\Theta$ is defined as the closure in $A_\Theta$ of $\lambda_\Theta(L_1(\mathbb{R}^n))$ with
$$f=\int_{\mathbb{R}^{n}}\hat{f}_{\Theta}(\xi)\lambda_\Theta(\xi)d\xi.$$
Note that if $\Theta=0$, then $E_\Theta=\mathcal{C}_0(\mathbb{R}^n)$. Define
\begin{align*}
\tau_\Theta(f)=\tau_{\Theta}\left(\int_{\mathbb{R}^{n}}\hat{f}_{\Theta}(\xi)\lambda_\Theta(\xi)d\xi\right)=\hat{f}_{\Theta}(0),
\end{align*}
where $\hat{f}_\Theta$ is an integrable and smooth function mapping $\mathbb{R}^n$ to $\mathbb{C}$. $\tau_\Theta$ extends to a normal semifinite faithful trace on $E_\Theta$. Let $\mathcal{R}_\Theta$ be the von Neumann algebra generated by $E_\Theta$ (that is, the double commutant of $E_\Theta$) in the GNS representation of $\tau_\Theta$. $\mathcal{R}_\Theta$ is called a quantum Euclidean space associated with $\Theta$. Note that if $\Theta=0$, then $\mathcal{R}_\Theta=L_\infty(\mathbb{R}^n)$.  Consider a normal injective $*$-homomorphism $$\sigma_\Theta:\mathcal{R}_\Theta\ni \lambda_\Theta(\xi)\mapsto \exp_\xi\otimes \lambda_\Theta(\xi)\in L_\infty(\mathbb{R}^n)\botimes \mathcal{R}_\Theta,$$
where $\exp_\xi$ stands for the Fourier character $\exp(2\pi i\langle \xi,\cdot\rangle)$. Now we set
$\mathcal{S}_\Theta=\big\{f\in\mathcal{R}_\Theta:\check{f}_{\Theta}\in\mathcal{S}(\mathbb{R}^n)\big\}$
be the quantum Schwartz class, which is a weak-$*$ dense subalgebra of $\mathcal{R}_\Theta$ and dense in $L_p(\mathcal{M})$ for all $p>0$ (\cite[Proposition 1.13]{MR4320770}).

Our first goal is to modify the argument in \cite{MR4178915} to construct a natural P-metric for quantum Euclidean spaces. To begin with, we recall the heat semigroup $\{H_t\}_{t\geq 0}=\{e^{t\Delta}\}_{t\geq 0}$ on $\mathbb{R}^n$ acting on $f:\mathbb{R}^n\rightarrow \mathbb{C}$, which can be expressed as follows
\begin{align}\label{ht}
H_tf(x)=\int_{\mathbb{R}^n}\hat{f}(\xi)e^{-t|\xi|^2}\exp_{\xi}(x)d\xi.
\end{align}
This induces a regular semigroup $\mathcal{S}=\{S_{\Theta,t}\}_{t\geq 0}$ on $\mathcal{R}_\Theta$ determined by
\begin{align}\label{stheta}
\sigma_{\Theta}\circ S_{\Theta,t}=(H_t\otimes id_{\mathcal{R}_{\Theta}})\circ\sigma_\Theta.
\end{align}
$S_{\Theta,t}$ gives a Markov semigroup on $\mathcal{R}_{\Theta}$ which formally acts as
\begin{align*}
S_{\Theta,t}(f)=\int_{\mathbb{R}^{n}}\hat{f}_{\Theta}(\xi)e^{-t|\xi|^2}\lambda_{\Theta}(\xi)d\xi.
\end{align*}

Recall that $S_{\Theta,t}$ admits an infinitesimal negative generator
$$\Delta_{\Theta}f=\lim_{t\rightarrow 0}\frac{S_{\Theta,t}(f)-f}{t}=-4\pi^2\int_{\mathbb{R}^n}\hat{f}_{\Theta}(\xi)|\xi|^2\lambda_{\Theta}(\xi)d\xi.$$
Alternatively, one may define $\Delta_{\Theta}$ via second-order  partial $\Theta$-derivatives. To illustrate this, for $1\leq j\leq n$, define the $j$-th first-order  partial $\Theta$-derivative $\partial_{\Theta,j}$ as the linear extension of the map
$$\partial_{\Theta,j}(\lambda_\Theta(\xi))=2\pi i\xi_j\lambda_\Theta(\xi)$$
over the quantum Schwartz class $\mathcal{S}_\Theta$. Then it is easy to see that $\Delta_\Theta=\sum_{j=1}^n\partial_{\Theta,j}^2$.

Recall from Section \ref{Operator-valued setting} that the semicommutative extension $H_t\otimes id_{\mathcal{R}_{\Theta}}$ satisfies the following semigroup majorization
\begin{align}\label{majorization}
(H_{t}\otimes id_{\mathcal{R}_{\Theta}})(|\xi|^2)\leq \sum_{j\geq 0}2^{jn}\exp(-c2^{2j})(R_{j,\sqrt{t}}\otimes id_{\mathcal{R}_{\Theta}})(|\xi|^2),\ {\rm for}\ {\rm any}\ \xi\in L_\infty(\mathbb{R}^n)\botimes\mathcal{R}_\Theta,
\end{align}
and the metric integrability condition as well as the average domination condition. Then one can produce a P-metric on $\mathcal{R}_{\Theta}$. To this end, let $B(r):=B(0,r)$ and consider the projections $q_{r}=\chi_{B(r)}\otimes {\bf 1}_{\mathcal{R}_{\Theta}}$. Define the completely positive unital map
\begin{align*}
&R_{\Theta,j,r}(f)=\frac{1}{|B(2^jr)|}\int_{B(2^jr)}\sigma_{\Theta}(f)(x)dx=\frac{1}{|B(2^jr)|}\int_{\mathbb{R}^n}\hat{\chi}_{B(2^jr)}(\xi)\hat{f}_{\Theta}(\xi)\lambda_{\Theta}(\xi)d\xi.
\end{align*}
Then $R_{\Theta,r}:=R_{\Theta,0,r}$ is of the form \eqref{hhhh} with $\rho_1 ={\bf 1}\otimes \cdot$, $\rho_2=\sigma_\Theta$, $E_\rho=\int\otimes id_{\mathcal{R}_\Theta}$ and with  $q_r$ being the projection $\chi_{B(r)}\otimes {\bf 1}_{\mathcal{R}_\Theta}$. One could easily check that
\begin{align}\label{tran}
\sigma_{\Theta}\circ R_{\Theta,j,r}=(R_{j,r}\otimes id_{\mathcal{R}_{\Theta}})\circ \sigma_{\Theta}
\end{align}
and  $R_{\Theta,j,r}f=f$ for any $f$ in the null space of $BMO_{\mathcal{S}}^c(\mathcal{R}_{\Theta})$-seminorm. Note that by substituting $\xi$ with $\sigma_\Theta(f)$ in \eqref{majorization}  for some $f\in\mathcal{R}_\Theta$  and then using the intertwining identities \eqref{stheta} and \eqref{tran}, we conclude that the semigroup majorization
\begin{align}\label{sdd}
S_{\Theta,t}(|\xi|^2)\leq C\sum_{j\geq 0}2^{jn}\exp(-c2^{2j})R_{\Theta,j,\sqrt{t}}(|\xi|^2)
\end{align}
holds for all $\xi\in\mathcal{R}_{\Theta}$.

Now we claim that for any $x_0\in\mathbb{R}^n$ and $r>0$, one has
\begin{align}\label{aver}
\left\|\fint_{B(x_0,r)}|\sigma_{\Theta}(f)(x)|^2dx\right\|_{\mathcal{R}_{\Theta}}\leq \left\|\fint_{B(0,r)}|\sigma_{\Theta}(f)(x)|^2dx\right\|_{\mathcal{R}_{\Theta}}.
\end{align}
Indeed, since $\sigma_\Theta$ is an injective $*$-homomorphism: $\mathcal{R}_\Theta\rightarrow L_\infty(\mathbb{R}^n)\botimes\mathcal{R}_{\Theta}$, one has
\begin{align*}
{\rm LHS}&=\left\|\sigma_{\Theta}\left(\fint_{B(0,r)}|\sigma_{\Theta}(f)(x)|^2dx\right)(x_0)\right\|_{\mathcal{R}_{\Theta}}\\
&\leq\left\|\sigma_{\Theta}\left(\fint_{B(0,r)}|\sigma_{\Theta}(f)(x)|^2dx\right)\right\|_{L_\infty(\mathbb{R}^n)\botimes\mathcal{R}_{\Theta}}\leq\left\|\fint_{B(0,r)}|\sigma_{\Theta}(f)(x)|^2dx\right\|_{\mathcal{R}_{\Theta}}.
\end{align*}
Next, note that $B(2^j\sqrt{t})$ can be covered by at most $2^{jn}$ balls with radius $\sqrt{t}$ and with finite overlap. For simplicity, we denote these balls by $B_{j,1}(\sqrt{t}),\cdots,B_{j,d_{j,n}}(\sqrt{t})$, where $d_{j,n}\lesssim 2^{jn}$. Such a cover is not unique, but we choose a kind of cover such that $B_{j,\ell}(y,r)=B_{j,\ell}(r)+y$. Then for any $j\geq 0$ and $t>0$, the following average  domination inequality holds:
\begin{align*}
\left\|\fint_{B(0,2^j\sqrt{t})}|\sigma_\Theta(f)(y)|^2dy\right\|_{\mathcal{R}_\Theta}
&\leq 2^{-jn}\sum_{\ell=1}^{d_{j,n}}\left\|\fint_{B_{j,\ell}(\sqrt{t})}|\sigma_\Theta(f)(y)|^2dy\right\|_{\mathcal{R}_\Theta}\\
&\leq C\left\|\fint_{B(\sqrt{t})}|\sigma_\Theta(f)(y)|^2dy\right\|_{\mathcal{R}_\Theta}.
\end{align*}

Therefore,
$$\mathcal{Q}_\Theta=\left\{\left(R_{\Theta,j,t},2^{jn/2}\exp(-c2^{2j})\cdot {\bf 1}_{\mathcal{R}_\Theta}\right):j\in\mathbb{N},t>0\right\}$$
is a P-metric for $\mathcal{R}_\Theta$.

Theorem \ref{main} applied to quantum Euclidean space can be formulated as follows.
\begin{theorem}\label{qusch}
For any anti-symmetric $\mathbb{R}$-valued $n\times n$ matrix $\Theta$, there exists a  constant $C=C(n)>0$ independent of $t$ and $\Theta$ such that
\begin{eqnarray*}
  \left\| (id_{\mathcal{R}_\Theta}-\Delta_\Theta)^{-s }e^{-it\Delta_\Theta} f\right\|_{BMO_{\mathcal{S}}(\mathcal{R}_{\Theta})} \leq C (1+|t|)^{n/2} \|f\|_{\mathcal{R}_\Theta}, \ {\rm for}\ s\geq \frac{n}{2}.
\end{eqnarray*}
Furthermore, for any $1<p<\infty$, there exists a  constant $C=C(n,p)>0$ independent of $t$ and $\Theta$ such that
\begin{eqnarray*}
 \left\| (id_{\mathcal{R}_\Theta}-\Delta_\Theta)^{-s }e^{-it\Delta_\Theta}  f\right\|_{L_p(\mathcal{R}_\Theta)} \leq C (1+|t|)^{\sigma_p} \|f\|_{L_p(\mathcal{R}_\Theta)},\ {\rm for}
  \ s\geq \sigma_p=n\Big|{1\over  2}-{1\over  p}\Big|.
\end{eqnarray*}
\end{theorem}
\begin{proof}
Since $\sigma_\Theta$ is an injective $*$-homomorphism: $\mathcal{R}_\Theta\rightarrow L_\infty(\mathbb{R}^n)\botimes \mathcal{R}_\Theta$, the theorem may be shown by $n$-dimensional analogous in \cite[Theorem 6.4]{MR2327840} together with a transference technique. But here we will provide a different proof via our algebraic approach.

We will verify the algebraic, analytic, topological and operator conditions imposed on Theorem \ref{main}, then the assertion follows directly.  See also \cite[Remark 1.4]{MR4178915} for the comment about the regularity assumption on the semigroup. To begin with, we choose $\mathcal{M}=\mathcal{R}_\Theta$, $\mathcal{N}_\rho= L_\infty(\mathbb{R}^n)\botimes \mathcal{R}_\Theta$ in Theorem \ref{main} and then consider two $*$-homomorphisms $\rho_1,\rho_2: \mathcal{R}_\Theta\rightarrow \mathcal{N}_\rho= L_\infty(\mathbb{R}^n)\botimes \mathcal{R}_\Theta$ defined by $\rho_1(f)={\bf 1}\otimes f$ and $\rho_2(f)=\sigma_\Theta(f)$, respectively. Besides, for any function $F\in C_b^\infty$, define amplified spectral multiplier $\pi(F(-\Delta_\Theta))$ by $\pi(F(-\Delta_\Theta)):=F(-\Delta\otimes id_{\mathcal{R}_\Theta})$, then it is direct to see that $\pi(F(-\Delta_\Theta)):L_2^c(\mathbb{R}^n)\botimes\mathcal{R}_\Theta\rightarrow L_2^c(\mathbb{R}^n)\botimes\mathcal{R}_\Theta$ satisfying $$\pi(F(-\Delta_\Theta))\circ\rho_2=\rho_2\circ F(-\Delta_\Theta)\ {\rm on}\ \mathcal{S}_\Theta.$$
 Next, choose $E_{\rho}=\int\otimes id_{\mathcal{R}_\Theta}$ and  $q_r$ be the projections $\chi_{B(r)}\otimes id_{\mathcal{R}_\Theta}$, then the algebraic conditions (ALi)-(ALii) follows from commutativity directly.

Now we verify the analytic condition (ANi) by showing that
\begin{align*}
\left\|\frac{1}{|B(r_3)|}\int_{B(r_1)}|\sigma_\Theta(f)(y)|^2dy\right\|_{\mathcal{R}_\Theta}
\leq C\left(\frac{r_1}{r_3}\right)^{n}\left\|\frac{1}{|B(r_2)|}\int_{B(r_2)}|\sigma_\Theta(f)(y)|^2dy\right\|_{\mathcal{R}_\Theta}
\end{align*}
for any $r_1\geq r_2\geq r_3$ and $f\in \mathcal{R}_\Theta$. To verify this, we recall that $B(r_1)$ can be covered by at most $c\big(\frac{r_1}{r_2}\big)^n$ balls. We denote these balls by $B(x_1,r_2)$, $\cdots$, $B(x_{d_{n}},r_2)$, where $d_n\lesssim \big(\frac{r_1}{r_2}\big)^n$. Then by \eqref{aver},
\begin{align}\label{yinyin}
\left\|\frac{1}{|B(r_3)|}\int_{B(r_1)}|\sigma_\Theta(f)(y)|^2dy\right\|_{\mathcal{R}_\Theta}
&\leq C \sum_{\ell=1}^{d_n}\left(\frac{r_2}{r_3}\right)^n\left\|\frac{1}{|B(x_\ell,r_2)|}\int_{B(x_\ell,r_2)}|\sigma_\Theta(f)(y)|^2dy\right\|_{\mathcal{R}_\Theta}\nonumber\\
&\leq C\left(\frac{r_1}{r_3}\right)^{n}\left\|\frac{1}{|B(r_2)|}\int_{B(r_2)}|\sigma_\Theta(f)(y)|^2dy\right\|_{\mathcal{R}_\Theta}.
\end{align}

Observe that the conditions (Oi), (Oii) in this setting are special cases of \eqref{veroi} and \eqref{veroii}, respectively, in which we choose  $\mathcal{M}=\mathcal{R}_\Theta$, $X=\mathbb{R}^n$ and $L=-\Delta$. Note that $\sigma_\Theta(\mathcal{S}_\Theta)\subset L_2^c(\mathbb{R}^n)\botimes \mathcal{R}_\Theta$ (\cite[Proposition B.3]{MR4320770}), so the topological conditions (Ti) and (Tii) are direct consequences of Lemma \ref{apprx} with $T$, $\mathcal{M}$ and $X$ chosen to be $F(-\Delta)$, $\mathcal{R}_\Theta$ and $\mathbb{R}^n$, respectively. Recalling that $S_{\Theta,t}$ is a regular Markov semigroup (\cite{MR4178915}), we end the proof of Theorem \ref{qusch}.
\end{proof}
\subsection{Matrix algebra}
In this subsection we apply our algebraic approach to investigate Schr\"{o}dinger groups in matrix algebra $B(\ell_2)$.

Given $A=\sum_{m,k}a_{m,k}e_{m,k}\in B(\ell_2)$, define a Markov semigroup of convolution type on $B(\ell_2)$ by
$$S_t(A)=\sum_{m,k}e^{-t|m-k|^2}a_{m,k}e_{m,k}.$$
Recall that $S_{t}$ admits an infinitesimal non-negative generator
$$L_{B(\ell_2)}A=-\lim_{t\rightarrow 0}\frac{S_{t}(A)-A}{t}=\sum_{m,k}|m-k|^2a_{m,k}e_{m,k}.$$
Consider an injective $*$-homomorphism $u: B(\ell_2)\rightarrow L_\infty(\mathbb{R})\botimes B(\ell_2)$ defined by
$$u(A)(y)=\sum_{m,k}a_{m,k}e^{2\pi i(m-k)y}e_{m,k}.$$
Let $\{H_t\}_{t\geq 0}=\{e^{t\Delta}\}_{t\geq 0}$ be the heat semigroup on $\mathbb{R}$ defined in \eqref{ht}. Then it is direct to see that $\mathcal{S}=(S_t)_{t\geq 0}$ is the regular semigroup associated to $\{H_t\}_{t\geq 0}$ by transference
\begin{align}\label{inter0}
u\circ S_t=(H_t\otimes id_{B(\ell_2)})\circ u.
\end{align}
For any $j\geq 0$ and $t>0$, define a completely positive unital map $\mathcal{R}_{j,t}$ on $B(\ell_2)$ by
$$\mathcal{R}_{j,t}(A):=\fint_{B(2^jt)}u(A)(y)dy.$$
Then $\mathcal{R}_{t}:=\mathcal{R}_{0,t}$ is of the form \eqref{hhhh} with $\rho_1 ={\bf 1}\otimes \cdot$, $\rho_2=u$, $E_\rho=\int\otimes id_{B(\ell_2)}$ and with  $q_r$ being the projection $\chi_{B(r)}\otimes {\bf 1}_{B(\ell_2)}$.
One could easily check that
\begin{align}\label{interw}
&u\circ \mathcal{R}_{j,t}=(R_{j,t}\otimes id_{B(\ell_2)})\circ u.
\end{align}
With this equality, one can follow the approach as estimating \eqref{sdd} and \eqref{aver} to conclude the semigroup majorization
\begin{align*}
S_{t}(|\xi|^2)\leq \sum_{j\geq 0}2^{j}\exp(-c2^{2j})\mathcal{R}_{j,\sqrt{t}}(|\xi|^2)
\end{align*}
for all $\xi\in B(\ell_2)$ and the average domination condition: for any $x_0\in\mathbb{R}$ and $r>0$, one has
\begin{align*}
\left\|\fint_{B(x_0,r)}|u(A)(x)|^2dx\right\|_{B(\ell_2)}\leq \left\|\fint_{B(r)}|u(A)(x)|^2dx\right\|_{B(\ell_2)}.
\end{align*}
Therefore,
$$\mathcal{Q}_{B(\ell_2)}=\left\{\left(\mathcal{R}_{j,t},2^{j/2}\exp(-c2^{2j})\otimes {\bf 1}_{B(\ell_2)}\right):j\in\mathbb{N},t>0\right\}$$
is a P-metric for $B(\ell_2)$.

Denote by $S_p$ the Schatten-$p$ class on the Hilbert space $\ell_2$, i.e. $S_p=L_p(B(\ell_2))$. Then Theorem \ref{main} applied to matrix algebra can be formulated as follows.
\begin{theorem}\label{bl2}
There exists a  constant $C>0$ independent of $t$ such that
\begin{eqnarray}
  \left\| (id_{B(\ell_2)}+L_{B(\ell_2)})^{-s}e^{itL_{B(\ell_2)}} A\right\|_{BMO_{\mathcal{S}}^c(B(\ell_2))} \leq C (1+|t|)^{1/2} \|A\|_{B(\ell_2)}, \ {\rm for}\ s\geq\frac{1}{2}.
\end{eqnarray}
Furthermore, for any $1<p<\infty$, there exists a  constant $C_p>0$ independent of $t$ such that
\begin{eqnarray}
 \left\| (id_{B(\ell_2)}+L_{B(\ell_2)})^{-s }e^{itL_{B(\ell_2)}}  A\right\|_{S_p} \leq C (1+|t|)^{|1/2-1/p|} \|A\|_{S_p}, \ {\rm for}
  \ s\geq \Big|{1\over  2}-{1\over  p}\Big|.
\end{eqnarray}
\end{theorem}
\begin{proof}
The proof of this theorem is similar to that of Theorem \ref{qusch}. We will modify the proof by providing an algebraic skeleton for $B(\ell_2)$.  To begin with, we choose $\mathcal{M}=B(\ell_2)$, $\mathcal{N}_\rho= L_\infty(\mathbb{R})\botimes B(\ell_2)$ in Theorem \ref{main} and then consider two $*$-homomorphisms $\rho_1,\rho_2: B(\ell_2)\rightarrow \mathcal{N}_\rho= L_\infty(\mathbb{R})\botimes B(\ell_2)$ defined by $\rho_1(A)={\bf 1}_{B(\ell_2)}\otimes A$ and $\rho_2(A)=u(A)$, respectively. Besides, for any function $F\in C_b^\infty$, define amplified spectral multiplier $\pi(F(L_{B(\ell_2)}))$ by $\pi(F(L_{B(\ell_2)})):=F(-\Delta\otimes id_{B(\ell_2)})$, then it is direct to see that $\pi(F(L_{B(\ell_2)})):V_\rho\rightarrow V_\rho$ satisfying $$\pi(F(L_{B(\ell_2)}))\circ\rho_2=\rho_2\circ F(L_{B(\ell_2)})\ {\rm on}\ \mathcal{A}_{{B(\ell_2)}},$$
where $$\mathcal{A}_{{B(\ell_2)}}=\left\{x\in B(\ell_2):\ x=\sum_{m\in I,k\in J}a_{m,k}e_{m,k}, \ \text {for  some finite index sets}\ I,\ J \right\} .$$ Next, choose $E_{\rho}=\int\otimes id_{B(\ell_2)}$ and  $q_r$ be the projection $\chi_{B(r)}\otimes id_{B(\ell_2)}$, then the algebraic conditions (ALi)-(ALii) follow from commutativity directly and one can follow a similar procedure as verifying \eqref{yinyin} to verify the condition (ANi). Observe that the conditions (Oi), (Oii) in this setting are special cases of \eqref{veroi} and \eqref{veroii}, respectively, in which we choose  $\mathcal{M}=B(\ell_2)$, $X=\mathbb{R}$ and $L=-\Delta$.

Note that $u(\mathcal{A}_{B(\ell_2)})\nsubseteq L_2^c(\mathbb{R}^n)\botimes B(\ell_2)$. Therefore, to verify the topological conditions (Ti) and (Tii), by Lemma \ref{apprx} with $T$, $\mathcal{M}$, $X$ and $x$ chosen to be $F(-\Delta)$, $B(\ell_2)$, $\mathbb{R}$ and $0$, respectively, it remains to verify that for any $A=\sum_{m,k}a_{m,k}e_{m,k}\in\mathcal{A}_{B(\ell_2)}$, any $b\in S_2$, and any Schwartz functions $F_j$, $F$ satisfying $F_j\rightarrow F$ pointwise with $\sup\limits_{j}\|F_j\|_\infty<+\infty$,
\begin{align}\label{linmm}
\lim_{s\rightarrow \infty}{\rm Tr}\left(b^*\int_{B(0,r)}\Big|(F(-\Delta)\otimes id_{B(\ell_2)})\Big(\sum_{m,k}a_{m,k}e^{i(m-k)\cdot}\chi_{B(0,s)^c}(\cdot)e_{m,k}\Big)(y)\Big|^2dyb\right)=0
\end{align}
and
\begin{align}\label{linmmmm}
\lim_{j\rightarrow \infty}{\rm Tr}\left(b^*\int_{B(0,r)}\Big|((F_j(-\Delta)-F(-\Delta))\otimes id_{B(\ell_2)})\Big(\sum_{m,k}a_{m,k}e^{i(m-k)\cdot}e_{m,k}\Big)(y)\Big|^2dyb\right)=0.
\end{align}
To show \eqref{linmm}, we apply domination convergence theorem to see that
\begin{align*}
&{\rm Tr}\left(b^*\int_{B(0,r)}\Big|(F(-\Delta)\otimes id_{B(\ell_2)})\Big(e^{i(m-k)\cdot}\chi_{B(0,s)^c}(\cdot)e_{m,k}\Big)(y)\Big|^2dyb\right)\\
&=\int_{B(0,r)}\Big|F(-\Delta)\Big(e^{i(m-k)\cdot}\chi_{B(0,s)^c}(\cdot)\Big)(y)\Big|^2{\rm Tr}(b^*e_{k,k}b)dy\\
&\leq\int_{B(0,r)}\left(\int_{B(0,s)^c}|\mathcal{F}^{-1}(F(|\cdot|^2))(y-z)|dz\right)^2dy\|b\|_{S_2}^2\rightarrow 0,\ {\rm as}\ s\rightarrow \infty,
\end{align*}
where we used the notation $\mathcal{F}^{-1}$ to denote Fourier inverse transform of a suitable function. To show \eqref{linmmmm}, a direct calculation yields
\begin{align*}
&{\rm Tr}\left(b^*\int_{B(0,r)}\Big|((F_j(-\Delta)-F(-\Delta))\otimes id_{B(\ell_2)})\Big(e^{i(m-k)\cdot}e_{m,k}\Big)(y)\Big|^2dyb\right)\\
&=r|(F_j-F)(|m-k|^2)|^2\tau(b^*e_{k,k}b)\rightarrow 0,\ {\rm as}\ j\rightarrow \infty.
\end{align*}
Therefore, all the conditions are verified. To obtain the $L_p$ boundedness assertion, it suffices to note that the semigroup $\mathcal{S}_t$ is regular (see \cite{MR4178915}). This finishes the proof of  Theorem \ref{bl2}.
\end{proof}
\subsection{Group von Neumann algebra}
In this subsection we apply our algebraic approach to investigate Schr\"{o}dinger groups in group von Neumann algebra equipped with finite-dimensional cocycles.

At the beginning of this subsection, we collect some basic concepts and properties about Fourier multipliers on group von Neumann algebra from \cite{MR3283931}. Let $G$ be a discrete group and $\lambda:G\rightarrow B(\ell_2(G))$ be its left regular representation defined by $\lambda(g)\delta_h=\delta_{gh}$, where each $\delta_g$ takes value 1 at $g$ and zero elsewhere such that $(\delta_g)_{g\in G}$ forms a canonical basis of $\ell_2(G)$. Write $\mathcal{L}(G)$ for its group von Neumann algebra, which is defined to be the closure of $\lambda(C_c(G))$ in the weak operator topology of $B(\ell_2(G))$. Let $\tau_G$ be a standard normalized trace  on $\mathcal{L}(G)$ defined by
$$\tau_G(f)=\langle \delta_e,f\delta_e \rangle$$
for  $f\in\mathcal{L}(G)$, where $e$ denotes the identity of $G$. Note that any $f\in\mathcal{L}(G)$ admits a Fourier series expansion
$$\sum_{g\in G}\hat{f}(g)\lambda(g),\ {\rm with}\ \hat{f}(g)=\tau_G(f\lambda(g^{-1}))\ {\rm such}\ {\rm that}\ \tau_G(f)=\hat{f}(e).$$
Denote $L_p(\hat{G})=L_p(\mathcal{L}(G),\tau_G)$ be the $L_p$ space over the noncommutative measure space $(\mathcal{L}(G),\tau_G)$, which is equipped with the norm
\begin{align*}
\|f\|_{L_p(\hat{G})}=\bigg\|\sum_{g\in G} \hat{f}(g)\lambda(g)\bigg\|_p=\bigg(\tau_G\bigg[\Big|\sum_{g\in G} \hat{f}(g)\lambda(g)\Big|^p\bigg]\bigg)^{\frac{1}{p}}.
\end{align*}
For any $m:G\rightarrow \mathbb{R}$, a Fourier multiplier on $\mathcal{L}(G)$ is defined by
$$T_m:\sum_{g\in G}\hat{f}(g)\lambda(g)\mapsto \sum_{g\in G}m(g)\hat{f}(g)\lambda(g).$$
Note that in the case of $G=\mathbb{Z}^n$, this corresponds to the Fourier mutipliers on the $n$-torus.

Let $(\mathcal{H},a,b)$ be an affine representation of $G$, which is an orthogonal representation $a: G\rightarrow O(\mathcal{H})$ over a real Hilbert space $\mathcal{H}$ together with a mapping $b:G\rightarrow \mathcal{H}$ satisfying the cocycle law $$b(gh)=a_g (b(h))+b(g).$$
Given an affine representation $(\mathcal{H},a,b)$, the function $\psi(g)=\langle b(g),b(g)\rangle_\mathcal{H}$ is called a length if it satisfies the following three properties:

(1) Vanishing property: $\psi(e)=0$;

(2) Symmetry: $\psi(g)=\psi(g^{-1})$, for any $g\in G$;

(3) Conditionally negative property: $\sum_{g,h\in G}\bar{\beta}_g\beta_h \psi(g^{-1}h)\leq 0$ if $\sum_{g\in G}\beta_g=0$.

Conversely, it follows from Schoenberg's theorem \cite{MR1503439} that any length $\psi$ determines a precise affine representation $(\mathcal{H}_\psi,a_\psi,b_\psi)$ such that $\psi(g)=\langle b_\psi(g),b_\psi(g)\rangle_{\mathcal{H}_\psi}.$ Throughout this subsection, we assume that dim$\mathcal{H}_\psi=n$. Note that under this assumption, one can find a linear isometric $T:\mathcal{H}_\psi\rightarrow \mathbb{R}^n$, an orthogonal representation $\tilde{a}: G\rightarrow O(\mathbb{R}^n)$ defined by $\tilde{a}_g(x):=T\alpha_g(T^{-1}x)$ and a map $\tilde{b}:G\rightarrow \mathbb{R}^n$ defined by $\tilde{b}(g)=T(b(g))$, which satisfies the cocycle law $\tilde{b}(gh)=\tilde{a}_g (\tilde{b}(h))+\tilde{b}(g)$. Therefore, $(\mathbb{R}^n,\tilde{a},\tilde{b})$ is also an affine representation of $G$ and it satisfies $\psi(g)=\langle \tilde{b}(g),\tilde{b}(g)\rangle_{\mathbb{R}^n}$, for any $g\in G$. In what follows, we will abuse the notation $(\mathbb{R}^n,a,b)$ to denote an affine representation of $G$ for simplicity.

Let $\psi:G\rightarrow \mathbb{R}_+$ be any length on some discrete group $G$, then the infinitesimal generator of the semigroup $S_{\psi,t}:\lambda(g)\mapsto \exp(-t\psi(g))\lambda(g)$ is the map determined by $L_\psi(\lambda(g))=\psi(g)\lambda(g)$. Write $\mathcal{S}_\psi=\{S_{\psi,t}\}_{t\geq 0}$. Note that any $\psi$-radial Fourier multipliers $T_{m\circ \psi}$ fall in the category of spectral operators of the form $m(L_\psi)$.

Define an action $\alpha:G\rightarrow {\rm Aut}(L_\infty(\mathbb{R}^n))$ by $\alpha_g(e^{2\pi i\langle \xi,\cdot \rangle})=e^{2\pi i \langle a_g(\xi),\cdot\rangle}$, which is clearly trace preserving. Then consider two $*$-representations
$$\rho:L_\infty(\mathbb{R}^n)\ni f\mapsto \sum_{h\in G}\alpha_{h^{-1}}(f)\otimes e_{h,h}\in L_\infty(\mathbb{R}^n)\botimes B(\ell_2(G)),$$
$$\Lambda:G\ni g\mapsto\sum_{h\in G}{\bf 1}\otimes e_{gh,h}\in L_\infty(\mathbb{R}^n)\botimes B(\ell_2(G)),$$
where $e_{g,h}$ is the matrix unit for $B(\ell_2(G))$.  Now we recall from \cite[Section 2.1]{MR3283931} (see also \cite{MR1943006}) the concept of crossed product algebra $L_\infty(\mathbb{R}^n)\rtimes_\alpha G$, which is defined as the weak operator closure in $L_\infty(\mathbb{R}^n)\botimes B(\ell_2(G))$ of the $*$-algebra generated by $\rho(L_\infty(\mathbb{R}^n))$ and $\Lambda(G)$.  The crossed product algebra $L_\infty(\mathbb{R}^n)\rtimes_\alpha G$ can be embedded into $L_\infty(\mathbb{R}^n)\botimes B(\ell_2(G))$ via the map $\rho\rtimes \Lambda$. To illustrate this, we formally write a generic element of $L_\infty(\mathbb{R}^n)\rtimes_\alpha G$ as $\sum_{g\in G}f_g\rtimes_\alpha \lambda(g)$, with $f_g\in L_\infty(\mathbb{R}^n)$, and then observe that
\begin{align*}
(\rho\rtimes \Lambda)\left(\sum_{g\in G}f_g\rtimes_\alpha \lambda(g)\right)
&=\sum_{g\in G}\rho(f_g)\Lambda(g)\\
&=\sum_{g\in G}\left(\sum_{h,\gamma\in G}(\alpha_{h^{-1}}(f_g)\otimes e_{h,h})({\bf 1}\otimes e_{g\gamma,\gamma})\right)\\
&=\sum_{g\in G}\left(\sum_{h\in G}\alpha_{h^{-1}}(f_g)\otimes e_{h,g^{-1}h}\right)\\
&=\sum_{g\in G}\left(\sum_{h\in G}\alpha_{(gh)^{-1}}(f_g)\otimes e_{gh,h}\right).
\end{align*}

 Consider a normal $*$-homomorphism given by $$\pi_1:\mathcal{L}(G)\ni \lambda(g)\mapsto e^{2\pi i \langle b(g),\cdot \rangle}\rtimes_\alpha\lambda(g)\in L_\infty(\mathbb{R}^n)\rtimes_\alpha G.$$ Then it is direct that $\eta_2:=(\rho\rtimes \Lambda)\circ \pi_1$ is a normal injective $*$-homomorphism: $\mathcal{L}(G)\rightarrow L_\infty(\mathbb{R}^n)\botimes B(\ell_2(G))$, which can be expressed as
$$\eta_2(\lambda(g))=\sum_{h\in G}e^{2\pi i\langle  a_{(gh)^{-1}}b(g),\cdot\rangle}\otimes e_{gh,h}.$$
For any $j\geq 0$ and $t>0$, we define a completely positive unital map  $\mathscr{R}_{j,t}$ on $\mathcal{L}(G)$ by
$$\mathscr{R}_{j,t}(f):=\fint_{B(2^jt)}\eta_2(f)(y)dy,\ f\in\mathcal{L}(G).$$
Since $a_g\in O(\mathbb{R}^n)$, it can be verified that for any $\lambda(g)\in\mathcal{L}(G)$,
$$\mathscr{R}_{j,t}(\lambda(g))=C_{g,2^jt}\sum_{h\in G}{\bf 1}\otimes e_{gh,h}\in \eta_1(\mathcal{L}(G))\simeq \mathcal{L}(G),$$
where $C_{g,r}:=\fint_{B(r)}e^{2\pi i \langle b(g),y\rangle}dy$ and $\eta_1$ is an injective $*$-homomorphism defined by
$$\eta_1:\mathcal{L}(G)\ni\lambda(g)\mapsto \sum_{h\in G}{\bf 1}\otimes e_{gh,h}\in B(\ell_2(G)).$$
Then $\mathscr{R}_{t}:=\mathscr{R}_{0,t}$ is of the form \eqref{hhhh} with $\rho_j=\eta_j$, $j=1,2$, $E_\rho=\int\otimes id_{B(\ell_2(G))}$ and with  $q_r$ being the projection $\chi_{B(r)}\otimes {\bf 1}_{B(\ell_2(G))}$.
To avoid confusion, we distinguish $\mathcal{L}(G)$ with $\eta_1(\mathcal{L}(G))$ by  considering
$$\tilde{\mathscr{R}}_{j,t}(\lambda(g)):=C_{g,2^jt}\lambda(g),\ {\rm for}\ {\rm any}\ \lambda(g)\in\mathcal{L}(G).$$

Now we establish two useful intertwining inequalities.
\begin{lemma}\label{ide}
For any $j\geq0$ and $t>0$, one has the following equality.
\begin{align*}
&\eta_2\circ \tilde{\mathscr{R}}_{j,t}=(R_{j,t}\otimes id_{B(\ell_2(G))})\circ \eta_2.
\end{align*}

\end{lemma}
\begin{proof}
%
To begin with, observe that for any $\lambda(g)\in\mathcal{L}(G)$,
\begin{align*}
(R_{j,t}\otimes id_{B(\ell_2(G))})\circ \eta_2(\lambda(g))(x)
&=\fint_{B(x,2^jt)}\sum_{h\in G}e^{2\pi i\langle  a_{(gh)^{-1}}b(g),y\rangle}\otimes e_{gh,h}dy\\
&=\sum_{h\in G}\fint_{B(2^jt)}e^{2\pi i\langle  a_{(gh)^{-1}}b(g),y\rangle}dye^{2\pi i\langle  a_{(gh)^{-1}}b(g),x\rangle}\otimes e_{gh,h}\\
&=\fint_{B(2^jt)}e^{2\pi i\langle  b(g),y\rangle}dy\eta_2(\lambda(g))(x)\\
&=\eta_2\circ \tilde{\mathscr{R}}_{j,t}(\lambda(g)).
\end{align*}
This completes the proof of Lemma \ref{ide}.
\end{proof}
Denote by $\mathcal{A}_G$  the algebra of trigonometric polynomials in $\mathcal{L}(G)$. That is,
$$\mathcal{A}_G:=\left\{\sum_{g\in A}c_g\lambda(g):\ c_g\in\mathbb{C},A\subset G\ {\rm finite}\right\}.$$
\begin{lemma}\label{verfff}
For any bounded Borel function $F$, one has the equality:
\begin{align*}
\eta_2\circ F(L_\psi)=(F(-\Delta\otimes id_{B(\ell_2(G))}))\circ\eta_2\ \ \ \ {\rm on}\ \mathcal{A}_G.
\end{align*}
\end{lemma}
\begin{proof}
Note that for any $\lambda(g)\in\mathcal{L}(G)$,
\begin{align*}
(F(-\Delta\otimes id_{B(\ell_2(G))}))\circ\eta_2(\lambda(g))
&=(F(-\Delta)\otimes id_{B(\ell_2(G))})\left(\sum_{h\in G}e^{2\pi i\langle  a_{(gh)^{-1}}b(g),\cdot\rangle}\otimes e_{gh,h}\right)\\
&=\left(\sum_{h\in G}F(|a_{(gh)^{-1}}b(g)|^2)e^{2\pi i\langle  a_{(gh)^{-1}}b(g),\cdot\rangle}\otimes e_{gh,h}\right)\\
&=\left(\sum_{h\in G}F(\psi(g))e^{2\pi i\langle  a_{(gh)^{-1}}b(g),\cdot\rangle}\otimes e_{gh,h}\right)\\
&=\eta_2\circ F(L_\psi)(\lambda(g)).
\end{align*}
This ends the proof of Lemma \ref{verfff}.
\end{proof}
Observe that the semicommutative extension $H_t\otimes id_{B(\ell_2(G))}$ satisfies ${\rm for}\ {\rm any}\ \xi\in L_\infty(\mathbb{R}^n)\botimes B(\ell_2(G))$,
\begin{align*}
(H_{t}\otimes id_{B(\ell_2(G))})(|\xi|^2)\leq C\sum_{j\geq 0}2^{jn}\exp(-c2^{2j})(R_{j,\sqrt{t}}\otimes id_{B(\ell_2(G))})(|\xi|^2).
\end{align*}
This, in combination with Lemmas \ref{ide} and \ref{verfff}, implies that for any $\xi\in \mathcal{L}(G)$,
\begin{align*}
S_{\psi,t}(|\xi|^2)\leq C\sum_{j\geq 0}2^{jn}\exp(-c2^{2j})\tilde{\mathscr{R}}_{j,\sqrt{t}}(|\xi|^2).
\end{align*}
Recalling that $B(2^j\sqrt{t})$ is covered by the balls $B_{j,1}(\sqrt{t})$, ..., $B_{j,d_{j,n}}(\sqrt{t})$, we deduce that
\begin{align}\label{deccc}
\mathscr{R}_{j,\sqrt{t}}(|\xi|^2)\leq C2^{-jn}\sum_{\ell=1}^{d_{j,n}}\fint_{B_{j,\ell}(\sqrt{t})}\eta_2(|\xi|^2)(y)dy
\end{align}
for all $\xi\in \mathcal{L}(G)$, where $d_{j,n}\lesssim 2^{jn}$.  Now we verify the
average domination condition.
\begin{lemma}\label{avercon}
For any $x_0\in\mathbb{R}^n$, $r>0$ and $f\in \mathcal{A}_{G}$, one has
\begin{align*}
\left\|\int_{B(x_0,r)}|\eta_2(f)(x)|^2dx\right\|_{B(\ell_2(G))}\leq \left\|\int_{B(r)}|\eta_2(f)(x)|^2dx\right\|_{B(\ell_2(G))}.
\end{align*}
\end{lemma}
\begin{proof}
Note that for any $f=\sum_{g\in A}c_g\lambda(g)\in \mathcal{A}_G$,
\begin{align*}
\left\|\int_{B(x_0,r)}|\eta_2(f)(x)|^2dx\right\|_{B(\ell_2(G))}
&=\left\|\int_{B(0,r)}\Big|\sum_{g\in A}\sum_{h\in G}c_ge^{2\pi i\langle  a_{(gh)^{-1}}b(g),x+x_0\rangle}\otimes e_{gh,h}\Big|^2dx\right\|_{B(\ell_2(G))}\\
&=\left\|\int_{B(0,r)}\Big|\sum_{g\in A}\sum_{h\in G}c_{gh^{-1}}e^{2\pi i\langle  a_{g^{-1}}b(gh^{-1}),x+x_0\rangle}\otimes e_{g,h}\Big|^2dx\right\|_{B(\ell_2(G))}={\rm RHS}.
\end{align*}
To continue, consider the map $$\zeta:B(\ell_2(G))\ni\sum_{g,h\in G}a_{g,h}\otimes e_{g,h}\mapsto \sum_{g,h\in G}e^{2\pi i\langle a_{g^{-1}}b(gh^{-1}),\cdot\rangle}a_{g,h}\otimes e_{g,h}\in L_\infty(\mathbb{R}^n)\botimes B(\ell_2(G)) .$$
Recalling that $(\mathbb{R}^n,a,b)$ satisfies the cocycle law $b(gh)=a_g (b(h))+b(g)$, one can deduce that $\zeta$ is an injective $*$-homomorphism: $B(\ell_2(G))\rightarrow L_\infty(\mathbb{R}^n)\botimes B(\ell_2(G))$. Therefore,
\begin{align*}
{\rm RHS}
&=\left\|\zeta\Bigg(\int_{B(0,r)}\Big|\sum_{g\in A}\sum_{h\in G}c_{gh^{-1}}e^{2\pi i\langle  a_{g^{-1}}b(gh^{-1}),x\rangle}\otimes e_{g,h}\Big|^2dx\Bigg)(x_0)\right\|_{B(\ell_2(G))}\\
&\leq \left\|\zeta\Bigg(\int_{B(0,r)}\Big|\sum_{g\in A}\sum_{h\in G}c_{gh^{-1}}e^{2\pi i\langle  a_{g^{-1}}b(gh^{-1}),x\rangle}\otimes e_{g,h}\Big|^2dx\Bigg)\right\|_{L_\infty(\mathbb{R}^n)\botimes B(\ell_2(G))}\\
&\leq \left\|\int_{B(0,r)}\Big|\sum_{g\in A}\sum_{h\in G}c_{gh^{-1}}e^{2\pi i\langle  a_{g^{-1}}b(gh^{-1}),x\rangle}\otimes e_{g,h}\Big|^2dx\right\|_{B(\ell_2(G))}=\left\|\int_{B(r)}|\eta_2(f)(x)|^2dx\right\|_{B(\ell_2(G))}.
\end{align*}
This ends the proof of Lemma \ref{avercon}.
\end{proof}
Now we apply inequality \eqref{deccc} and Lemma \ref{avercon} to deduce that
\begin{align*}
\big\|\tilde{\mathscr{R}}_{j,t}(|\xi|^2)\big\|_{\mathcal{L}(G)}
&=\big\|\mathscr{R}_{j,t}(|\xi|^2)\big\|_{B(\ell_2(G))}\\
&\leq C2^{-jn}\sum_{\ell=1}^{d_{j,n}}\bigg\|\fint_{B_{j,\ell}(t)}\eta_2(|\xi|^2)(y)dy\bigg\|_{B(\ell_2(G))}\\
&=C2^{-jn}\sum_{\ell=1}^{d_{j,n}}\bigg\|\fint_{B_{j,\ell}(t)}|\eta_2(\xi)(y)|^2dy\bigg\|_{B(\ell_2(G))}\\
&\leq C\big\|\mathscr{R}_{\sqrt{t}}(|\xi|^2)\big\|_{B(\ell_2(G))}=C\big\|\tilde{\mathscr{R}}_{\sqrt{t}}(|\xi|^2)\big\|_{\mathcal{L}(G)}.
\end{align*}
Therefore,
$$\mathcal{Q}_{\mathcal{L}(G)}=\left\{\left(\tilde{\mathscr{R}}_{j,t},2^{jn/2}\exp(-c2^{2j})\otimes {\bf 1}_{\mathcal{L}(G)}\right):j\in\mathbb{N},t>0\right\}$$
is a P-metric for $\mathcal{L}(G)$.

Theorem \ref{main} applied to group von Neumann algebra can be formulated as follows.
\begin{theorem}\label{aser}
Let $G$ be a discrete group equipped with any length $\psi:G\rightarrow \mathbb{R}_+$ satisfying dim$\mathcal{H}_\psi=n<\infty$. For any $s\geq 0$, let $T_{m_s}:\sum_g \hat{f}(g)\lambda(g)\mapsto \sum_g m_s(\psi(g))\hat{f}(g)\lambda(g)$, where $m_s(\psi(g))=e^{it\psi(g)}(1+\psi(g))^{-s}$. Then there exists a  constant $C=C(n)>0$ independent of $t$ such that
\begin{eqnarray}
  \left\|T_{m_{s}} (f)\right\|_{BMO_{\mathcal{S}_\psi}(\mathcal{L}(G))} \leq C (1+|t|)^{n/2} \|f\|_{\mathcal{L}(G)}, \  {\rm for }\ s\geq \frac{n}{2}.
\end{eqnarray}
Furthermore, for any $1<p<\infty$, there exists a  constant $C=C(n,p)>0$ independent of $t$ such that
\begin{eqnarray}
 \left\| T_{m_{s}}(f)\right\|_{L_p(\hat{G})} \leq C (1+|t|)^{\sigma_p} \|f\|_{L_p(\hat{G})}, \  {\rm for}
  \ s\geq \sigma_p=n\Big|{1\over  2}-{1\over  p}\Big|.
\end{eqnarray}
\end{theorem}
\begin{proof}
The theorem can be shown by a general principle in \cite[Theorem D]{MR3283931}, which establishes a connection between $L_\infty(\mathbb{R}^n)\rightarrow BMO(\mathbb{R}^n)$ boundedness of radial Fourier multiplier over $\mathbb{R}^n$ and $\mathcal{L}(G)\rightarrow BMO(S_\psi)$ boundedness of radial Fourier multiplier over discrete group von Neumann algebra $\mathcal{L}(G)$ with dim$\mathcal{H}_\psi=n$. But here we will provide a different proof via our algebraic approach.


To show the BMO boundedness of $T_{m_s}$, we will modify the proof in Theorem \ref{qusch} by providing an algebraic skeleton for $\mathcal{L}(G)$.  To begin with, we choose $\mathcal{M}=\mathcal{L}(G)$, $\mathcal{N}_\rho= L_\infty(\mathbb{R}^n)\botimes B(\ell_2(G))$ in Theorem \ref{main} and then consider two $*$-homomorphisms $\rho_1,\rho_2: \mathcal{L}(G)\rightarrow \mathcal{N}_\rho= L_\infty(\mathbb{R}^n)\botimes B(\ell_2(G))$ defined by
$\rho_1(\lambda(g))=\sum_{h\in G}{\bf 1}\otimes e_{gh,h}$
 and
 $\rho_2(\lambda(g))=\eta_2(\lambda(g))$,
 respectively. Besides, for any function $F\in C_b^\infty$, define amplified spectral multiplier $\pi(F(L_\psi))$ by $\pi(F(L_\psi)):=F(-\Delta\otimes id_{B(\ell_2(G))})$, then it is direct to see that $\pi(F(L_\psi)):V_{\rho}\rightarrow V_{\rho}$ satisfying $$\pi(F(L_\psi))\circ\rho_2=\rho_2\circ F(L_\psi)\ {\rm on}\ \mathcal{A}_G.$$
 Next, choose $E_{\rho}=\int\otimes id_{B(\ell_2(G))}$, $q_r$ be the projection $\chi_{B(r)}\otimes id_{B(\ell_2(G))}$. then the algebraic conditions (ALi)-(ALii) follows from commutativity and one can follow a similar procedure as verifying \eqref{yinyin} to verify the condition (ANi).  Observe that the conditions (Oi), (Oii) in this setting are special cases of \eqref{veroi} and \eqref{veroii}, respectively, in which we choose  $\mathcal{M}=B(\ell_2(G))$, $X=\mathbb{R}^n$ and $L=-\Delta$.

 Note that $\eta_2(\mathcal{L}(G))\nsubseteq L_2^c(\mathbb{R}^n)\botimes B(\ell_2(G))$. Therefore, to verify the topological conditions (Ti) and (Tii), by Lemma \ref{apprx} with $T$, $\mathcal{M}$, $X$ and $x$ chosen to be $F(-\Delta)$, $B(\ell_2(G))$, $\mathbb{R}^n$ and $0$, respectively, it remains to verify that for any $f=\sum_gc_g\lambda(g)\in\mathcal{A}_G$, any $\kappa\in S_2(\ell_2(G))$, and any Schwartz functions $F_j$, $F$ satisfying $F_j\rightarrow F$ pointwise with $\sup\limits_{j}\|F_j\|_\infty<+\infty$,
\begin{align}\label{linmmsss}
\lim_{s\rightarrow \infty}{\rm Tr}\Bigg(\kappa^*\int_{B(0,r)}\Big|(F(-\Delta)\otimes id_{B(\ell_2(G))})\Big(\sum_{g\in A}\sum_{h\in G}c_ge^{2\pi i\langle  a_{(gh)^{-1}}b(g),\cdot\rangle}\chi_{B(0,s)^c}(\cdot)\otimes e_{gh,h}\Big)(y)\Big|^2dy\kappa\Bigg)=0
\end{align}
and
\begin{align}\label{linmmsssss}
\lim_{s\rightarrow \infty}{\rm Tr}\Bigg(\kappa^*\int_{B(0,r)}\Big|((F_j(-\Delta)-F(-\Delta))\otimes id_{B(\ell_2(G))})\Big(\sum_{\substack{g\in A\\h\in G}}c_ge^{2\pi i\langle  a_{(gh)^{-1}}b(g),\cdot\rangle}\otimes e_{gh,h}\Big)(y)\Big|^2dy\kappa\Bigg)=0.
\end{align}

To show \eqref{linmmsss}, by domination convergence theorem, for any $\kappa=\sum_{k,\ell}c_{k,\ell}e_{k,\ell}\in \mathcal{A}_{B(\ell_2(G))}$,
\begin{align*}
&{\rm Tr}\left(\kappa^*\int_{B(0,r)}\Big|(F(-\Delta)\otimes id_{B(\ell_2(G))})\Big(\sum_{h\in G}e^{2\pi i\langle  a_{(gh)^{-1}}b(g),\cdot\rangle}\chi_{B(0,s)^c}(\cdot)\otimes e_{gh,h}\Big)(y)\Big|^2dy\kappa\right)\\
&={\rm Tr}\left(\kappa^*\int_{B(0,r)}\sum_{h\in G}\Big|F(-\Delta)(e^{2\pi i\langle  a_{(gh)^{-1}}b(g),\cdot\rangle}\chi_{B(0,s)^c}(\cdot)) (y)\Big|^2e_{h,h}dy\kappa\right)\\
&=\int_{B(0,r)}\sum_{k,\ell}|c_{k,\ell}|^2\Bigg(\int_{B(0,s)^c}\mathcal{F}^{-1}(F(|\cdot|^2))(y-z)e^{2\pi i\langle  a_{(g\ell)^{-1}}b(g),z\rangle} dz\Bigg)^2dy\rightarrow 0,\ {\rm as}\ s\rightarrow \infty.
\end{align*}
  This, together with the density of $\mathcal{A}_{B(\ell_2(G))}$ on $S_2(\ell_2(G))$, implies \eqref{linmmsss}.

To show \eqref{linmmsssss}, a direct calculation yields
\begin{align*}
&{\rm Tr}\Bigg(\kappa^*\int_{B(0,r)}\Big|((F_j(-\Delta)-F(-\Delta))\otimes id_{B(\ell_2(G))})\Big(\sum_{h\in G}e^{2\pi i\langle  a_{(gh)^{-1}}b(g),\cdot\rangle}\otimes e_{gh,h}\Big)(y)\Big|^2dy\kappa\Bigg)\\
&=r^n\sum_{k,\ell}|c_{k,\ell}|^2\big|(F_j-F)(|\psi(g)|^2)\big|^2\rightarrow 0,\ {\rm as}\ j\rightarrow \infty.
\end{align*}

Therefore, all the conditions have been verified.
To obtain the $L_p$ boundedness assertion, it suffices to note that the semigroup $\mathcal{S}_\psi$ is regular (see \cite{MR3283931}). This ends the proof of Theorem \ref{aser}.
\end{proof}
\bigskip
\bigskip
\bigskip
\bigskip
 \noindent
 {\bf Acknowledgments:}

The authors are supported by National Natural Science Foundation of China (No. 12071355).

%

\end{document}